\documentclass[reqno]{amsart}
\usepackage{caption}
\usepackage{subcaption}

\usepackage{mathabx,cite}
\usepackage[utf8]{inputenc} 
\usepackage[english]{babel}
\usepackage{amstext}
\usepackage{euscript}
\usepackage{bbm}
\usepackage{mathrsfs}
\usepackage{enumerate}
\usepackage{graphicx}
\usepackage{amscd}
\usepackage{verbatim}
\usepackage{amssymb}
\usepackage{amsthm}
\usepackage{amsmath}
\usepackage{amstext}
\usepackage{amsfonts}
\usepackage{latexsym}
\usepackage{mathtools} 
\usepackage{enumitem} 
\usepackage{accents} 
\usepackage[foot]{amsaddr} 

\usepackage{cases} 

\usepackage[left=1.2in, right=1.2in, top=1.5in, bottom=1.5in]{geometry}

\usepackage[usenames,dvipsnames]{xcolor} 
\definecolor{dkblue}{RGB}{1,31,91} 
\usepackage[colorlinks=true, pdfstartview=FitV, linkcolor=dkblue, citecolor=dkblue, urlcolor=dkblue]{hyperref}

\usepackage{todonotes}

\newcommand{\Z}{\mathbb{Z}}
\newcommand{\R}{\mathbb{R}}

\newcommand{\CC}{\mathbb C}
\newcommand{\RR}{\mathbb R}

\newcommand{\ZZ}{\mathbb Z}

\newcommand{\TT}{\mathbb T}

\newcommand{\pa}{\partial}

\DeclareMathOperator*{\esssup}{ess\,sup}

\normalsize
\normalsize
\usepackage{stmaryrd}
\def\comm#1#2{{\left\llbracket#1,#2\right\rrbracket}}

\DeclareMathOperator{\sgn}{sgn}

\newcommand{\norm}[1]{\left\lVert#1\right\rVert}

\theoremstyle{definition}
\newtheorem{theorem}{Theorem}

\newtheorem{lemma}[theorem]{Lemma}

\newtheorem{remark}[theorem]{Remark}

\numberwithin{equation}{section}
\numberwithin{theorem}{section}
\numberwithin{definition}{section}

\begin{document}
	
	\keywords{hydroelastic waves; fluid--structure interaction; asymptotic models; well-posedness}

\subjclass[2010]{35Q35, 76B15}
	
	\title[Weakly nonlinear models for hydroelastic water waves]{Weakly nonlinear models for hydroelastic water waves}

	\author[D. Alonso-Or\'an]{Diego Alonso-Or\'an}
	\address{Departamento de An\'{a}lisis Matem\'{a}tico and Instituto de Matem\'aticas y Aplicaciones (IMAULL), Universidad de La Laguna, C/Astrof\'{i}sico Francisco S\'{a}nchez s/n, 38271, La Laguna, Spain. \href{mailto:dalonsoo@ull.edu.es}dalonsoo@ull.edu.es}

	\author[R. Granero-Belinch\'on]{Rafael Granero-Belinch\'on}
	\address{Departamento de Matemáticas, Estadística y Computación, Universidad de Cantabria, Avda. Los
		Castros s/n, Santander, Spain.
		\href{mailto:rafael.granero@unican.es}rafael.granero@unican.es}

	\author[J. S. Ziebell]{Juliana S. Ziebell}
	\address{Departamento de Matemática Pura e Aplicada
		Universidade Federal do Rio Grande do Sul
		Porto Alegre, RS 91509, Brazil
		\href{mailto:julianaziebell@ufrgs.br}julianaziebell@ufrgs.br}

	\begin{abstract}
	In this work, we derive reduced interface models for hydroelastic water waves coupled to a nonlinear viscoelastic plate. In a weakly nonlinear small-steepness regime we obtain bidirectional nonlocal evolution equations capturing the interface dynamics up to quadratic order, and we also derive two unidirectional models describing one-way propagation while retaining the leading dispersive and dissipative effects induced by the plate. Remarkably, one of the bidirectional model has a doubly nonlinear structure in the sense that there there is a nonlinear elliptic operator acting on the acceleration of the interface. We prove local well-posedness for the bidirectional model for small data via a two-parameter regularization and nested fixed points. For the unidirectional models, we obtain local well-posedness for arbitrary data and global well-posedness for small data.

	\end{abstract}
	\thispagestyle{empty}

	\maketitle
	\tableofcontents

	\section{Introduction}
	The interaction between incompressible fluids and deformable elastic
	structures stands as one of the most challenging areas in modern
	mathematical analysis of partial differential equations. The past two
	decades have seen a substantial expansion of the field, driven by
	applications ranging from naval hydrodynamics \cite{AlbenShelley2008} and aeroelasticity to the
	mathematics of oceanic waves interacting with sea ice or floating plates, \cite{SquireDuganWadhamsRottierLiu1995}. At a mathematical level, these problems combine free-boundary fluid dynamics with the
nonlinear geometry and high-order mechanics of thin structures, leading to systems in which
hyperbolic (inertia), dispersive (bending), and sometimes dissipative (viscoelastic) effects interact in a
genuinely nonlocal way. \medskip

 In this context, one of the most prominent
and physically relevant settings is hydroelasticity, where surface waves propagate beneath a deformable
medium such as sea ice or an elastic plate, and where geometric shell models provide a natural framework to
describe the elastic forces. In contrast with classical water-wave theory, in which
	the interface is massless and obeys the Bernoulli dynamic condition, the
	hydroelastic interface carries mass, elasticity, and, in many situations,
	damping. A fundamental reference in this direction is the work of
	Plotnikov and Toland \cite{PlotnikovToland}, who derived a general
	geometrically nonlinear model for an elastic shell of Cosserat type
	interacting with an irrotational Euler fluid. Existing studies of hydroelastic waves are primarily devoted to traveling-wave solution, either establishing their existence or computing them, most notably in the massless case \cite{Toland2008} and, when the interface mass is included, in \cite{Toland2007Heavy,BaldiToland2010}. In two spatial dimensions the situation has also been extensively studied,
	most notably in the vortex-sheet formulation introduced by Ambrose and Liu
	\cite{AmbroseLiu}. There, the interface is assumed to be a one-dimensional
	curve with elastic properties, and the fluid velocity is described by the
	Birkhoff--Rott integral. Despite the apparent simplicity of the geometry,
	the coupling between elasticity and the Kelvin–Helmholtz instability of the
	vortex sheet leads to a highly delicate system. The authors obtained local
	well-posedness in Sobolev spaces and studied regularity mechanisms arising
	from the elastic terms.
 \medskip
	
	Beyond hydroelasticity, several lines of work have addressed fluid-structure interaction problems
	involving three-dimensional incompressible fluids and deformable solids.
	One classical approach, initiated by Coutand and Shkoller \cite{CoutandShkoller-FSI}, employs a fully
	Lagrangian framework for the elastodynamics of the solid and transforms the
	fluid equations to the material coordinates of the structure. This method
	has proven particularly effective for fluid–solid systems in which the
	structure retains a volumetric geometry. More recently, Kukavica, Tuffaha,
	and Ziane \cite{KukavicaTuffahaZiane2009,KukavicaTuffahaZiane2010} developed a sharp regularity theory for incompressible fluids
	interacting with elastic solids, establishing local well-posedness under
	minimal Sobolev assumptions. Their work highlights the need for precise
	control of the deformation of the elastic region, the mapping of the fluid
	equations to a moving coordinate system, and the geometric terms appearing
	in the transformed system.  \medskip
	
	In the context of thin structures, models based on the Koiter shell energy
	have emerged as a natural limit of three-dimensional elasticity. The work
	of Muha and Canić \cite{MuhaCanic2013} represents a landmark result in this
	direction: they establish the existence of weak solutions for the coupling
	between a viscous incompressible fluid and a Koiter shell, using an
	operator-splitting approach together with Aleksandrov transformations of
	the moving domain. Their setting assumes a cylindrical structure and
	exploits symmetries to obtain compactness and energy bounds that are
	otherwise difficult to establish. See also \cite{MachaMuhaNecasovaRoyTrifunovic2022} where
	the existence theory has also been extended to thermally coupled fluid--structure interaction models. \medskip

A further development is due to Balakrishna, Kukavica, Muha, and Tuffaha
\cite{BKMT2024}, who studied the interaction between the three-dimensional
Euler equations and a nonlinear two-dimensional Koiter plate represented as
a graph. Their model incorporates the full geometric nonlinearity of the
Koiter energy, the inertia of the plate, and Kelvin--Voigt damping, and it is
posed on a time-dependent fluid domain whose upper boundary is the unknown
graph. Related inviscid free-boundary fluid--structure models, including
variants with alternative structural laws, different solution concepts, and
global-in-time results in weak frameworks have also been investigated in
\cite{KukavicaTuffaha2023JEE,KukavicaTuffaha2024,AlazardShaoYang2025,AlazardKukavicaTuffaha2025,AydinKukavicaTuffaha2025preprint,WanYang2025}.
\medskip

\subsection{Contributions and main results}\label{subsec:contrib}

This work addresses a three-dimensional inviscid fluid--structure interaction
problem in which an incompressible Euler flow evolves beneath a nonlinear
viscoelastic plate represented as a graph. Starting from the geometric
Euler--plate system \eqref{eq:FSI_system}, we pursue a twofold program. On the
one hand, we derive reduced interface equations in a weakly nonlinear,
small-steepness regime, obtaining asymptotic models that capture the leading
hydroelastic dynamics. On the other hand, we provide a well-posedness theory
for these reduced equations. A similar program has been implemented for fluid-solid interaction problems in the case of Stokes law \cite{BukalMuha1,BukalMuha2} and Darcy law \cite{AlonsoOranGraneroBelinchon2026}.

\smallskip

\smallskip

Assuming irrotationality, we reduce the Euler--plate system to a boundary formulation and perform a weakly nonlinear
small-steepness expansion in $\varepsilon=H/L$. Truncating at quadratic accuracy yields two closed bidirectional
interface equation for the renormalized displacement $f=\eta^{(0)}+\varepsilon\eta^{(1)}$,
\begin{equation}\label{eq:main_bidirectional_summary}
\big(I+\Upsilon\Lambda\big) f_{tt}
+\delta\,\Lambda^{3} f_t
+\Big(\Lambda+\frac{\beta}{4}\Lambda^{5}\Big)f
=
\varepsilon\,\mathfrak N[f],
\end{equation}
where the quadratic forcing $\mathfrak N[f]$ depends on the bidirectional model under consideration. On the one hand, the first bidirectional model has a nonlinearity given by \eqref{eq:f_equation_compact_closed}. This first bidirectional model is doubly nonlinear in the sense that there is a nonlinear elliptic operator acting on the acceleration $f_{tt}$ of the membrane. On the other hand, a second bidirectional model (with the same accuracy) has a nonlinearity given by
explicitly in \eqref{eq:def_N_compact_final_T}. Moreover, by restricting to one horizontal dimension and introducing the characteristic variables
\[
\xi=x-t,\qquad \tau=\varepsilon t,\qquad f(x,t)=F(\xi,\tau),
\]
we obtain two unidirectional slow-modulation models. The first one, associated with the quadratic-precision
bidirectional equation \eqref{eq:f_equation_compact_closed}, can be written as
\begin{equation}\label{eq:main_uni1_summary}
F_\tau=\frac{1}{\varepsilon}\big(a+b\,\mathcal H\big)\!\left((I+\Upsilon\Lambda)F_\xi+\frac{\beta}{4}\Lambda F_{\xi\xi\xi}+\delta\Lambda F_{\xi\xi}+\mathcal H F\right)
-\big(a+b\,\mathcal H\big)\!\left(2F_\xi\,\Lambda F-\comm{\Lambda}{F}F_\xi\right),
\end{equation}
where $a=a(\Lambda,\partial_\xi)$ and $b=b(\Lambda,\partial_\xi)$ are the real Fourier multipliers defined in
\eqref{eq:ab_operators_repeat} (all multipliers understood with respect to $\xi$).
The second one, associated with the alternative quadratic closure \eqref{eq:f_model_compact_final_again}, takes the form
\begin{equation}\label{eq:main_uni2_summary}
\begin{aligned}
F_\tau
&=
\frac{1}{\varepsilon}\Big(\alpha+\gamma\,\mathcal H\Big)
\Big[
(I+\Upsilon\Lambda)F_{\xi\xi}
+\frac{\beta}{4}\Lambda F_{\xi\xi\xi\xi}
+\delta\Lambda F_{\xi\xi\xi}
+\Lambda F
\Big]\\
&\quad
-\Big(\alpha+\gamma\,\mathcal H\Big)
\Bigg[
\frac12\,\Lambda\Big(F_\xi^{\,2}-(\Lambda F)^2\Big)
-\comm{\Lambda}{F_\xi}F_\xi
+F_{\xi\xi}\Lambda F
+\comm{\Lambda}{F}\,\mathcal T F
-F_\xi\,\mathcal H\mathcal T F \\
&\qquad\qquad\quad
+\frac{\beta}{4}\Big(\comm{\Lambda}{F}\,\mathcal T F_{\xi\xi\xi\xi}
-F_\xi\,\Lambda \mathcal T F_{\xi\xi\xi}\Big)
+\delta\Big(\comm{\Lambda}{F}\,\mathcal T F_{\xi\xi\xi}
-F_\xi\,\Lambda \mathcal T F_{\xi\xi}\Big)
\Bigg],
\end{aligned}
\end{equation}
where $\alpha=\alpha(\Lambda,\partial_\xi)$ and $\gamma=\gamma(\Lambda,\partial_\xi)$ are the real Fourier multipliers
introduced in \eqref{eq:M1inv_ab_form}, and $\mathcal T=(I+\Upsilon\Lambda)^{-1}\Lambda$ is defined in
\eqref{eq:def_T_multiplier_new}.
\smallskip

The second purpose of this paper is to establish a well-posedness theory for the reduced models derived above. More precisely:
\begin{itemize}[leftmargin=*,labelsep=.6em,itemsep=.35em,topsep=.35em]
	\item \emph{Bidirectional local well-posedness for small data.}
	We prove local existence and uniqueness for the bidirectional interface equation
	\eqref{eq:f_equation_compact_closed} for mean-zero initial data
	\[
	f_0\in H^3_0(\TT^2),\qquad f_t(\cdot,0)\in H^1_0(\TT^2),
	\qquad \|f_0\|_{H^3}\ \text{sufficiently small},
	\]
	see Theorem~\ref{thm:WPbi}. A key difficulty in the bidirectional model \eqref{eq:f_equation_compact_closed} is that, once written as the
first-order system \eqref{eq:f_system_compact}, the coupling involves the unknown time derivative through
$\mathbb N\!\big(f,\partial_t v\big)$, and the problem is not a standard semilinear evolution and in fact it has a doubly nonlinear structure akin to
$$
v_t+N^1(v_t)=N^2(v),
$$
with $N^1$ being a nonlinear elliptic operator. Our approach is therefore
based on a two-parameter regularization and a nested fixed-point scheme;  see Steps~1--5 in the proof of Theorem~\ref{thm:WPbi}.
	
	\item \emph{Unidirectional local well-posedness and small-data global decay (first reduction).}
	For the unidirectional slow-modulation model \eqref{eq:Ft_final_ab_noG} (equivalently \eqref{eq:main_uni1_summary}),
we establish local well-posedness for arbitrary mean-zero data $F_0\in H^2(\TT)$; see Theorem~\ref{th:LWP:uni}.
	The argument relies on a priori energy estimates in the symmetrized formulation together with a classical Friedrichs
	mollification scheme. Moreover, for sufficiently small $H^2$ data we prove global-in-time existence and exponential decay
	of the natural energy; see Theorem~\ref{th:GWP_decay_H2}.

	\item \emph{Global small-data well-posedness (second reduction).}
	For the more singular unidirectional model \eqref{eq:Ft_final_abH_no_integration} (equivalently \eqref{eq:main_uni2_summary}),
we obtain global well-posedness for sufficiently small mean-zero data $F_0\in H^3(\TT)$; see Theorem~\ref{th:GWP_small_H3}.
	The proof is based on an energy method with a bootstrap/absorption step which allows the highest-order nonlinear terms to be
	controlled by the linear dissipation as long as the solution remains small.
\end{itemize}
\smallskip

\subsection{Notation and preliminary estimates}\label{sec:notation}
In what follows we collect the notation and the analytic tools that will be used throughout the paper, in particular basic properties of the Fourier multipliers involved and the Sobolev/product estimates needed in the well-posedness analysis.
\subsubsection*{Notation and functional setting} We work with horizontal variables $x=(x_1,x_2)\in\TT^2$, vertical variable $x_3\in(-\infty,0)$,
and time $t\ge0$. We write
\[
\nabla=(\partial_{x_1},\partial_{x_2},\partial_{x_3}),\qquad
\Delta=\partial_{x_1}^2+\partial_{x_2}^2+\partial_{x_3}^2,\qquad
\nabla_x=(\partial_{x_1},\partial_{x_2}),\qquad
\Delta_x=\partial_{x_1}^2+\partial_{x_2}^2,
\]
and $\partial_t$ for the time derivative. Throughout the paper the notation $A\lesssim B$
means $A\le C\,B$ for a constant $C>0$ which may depend on fixed parameters
(e.g.\ $\Upsilon,\delta,\beta$) but not on the evolving functions. \medskip

For $d\in\{1,2\}$ and $f:\TT^d\to\CC$ we write the Fourier series
\[
f(x)=\sum_{k\in\ZZ^d}\widehat f(k)e^{ik\cdot x},
\qquad
\widehat f(k)=\frac{1}{(2\pi)^d}\int_{\TT^d} f(x)e^{-ik\cdot x}\,dx,
\qquad
|k|=(k_1^2+\cdots+k_d^2)^{1/2}.
\]
For $s\in\RR$ we define
\[
\|f\|_{H^s(\TT^d)}^2:=\sum_{k\in\ZZ^d}(1+|k|^2)^s|\widehat f(k)|^2,
\qquad
\langle f\rangle:=\frac{1}{(2\pi)^d}\int_{\TT^d} f,
\qquad
H^s_0(\TT^d):=\{f\in H^s(\TT^d):\langle f\rangle=0\}.
\]
On mean-zero functions we also use the homogeneous seminorm
\[
\|f\|_{\dot H^s(\TT^d)}^2:=\sum_{k\neq0}|k|^{2s}|\widehat f(k)|^2,
\]
and interpret $\Lambda^{-1}$ mode-by-mode on $k\neq0$. By Poincar\'e,
for $s\ge0$ and $f\in H^s_0(\TT^d)$,
\begin{equation}\label{eq:Poincare_hom_equiv}
\|f\|_{H^s(\TT^d)}\sim \|f\|_{\dot H^s(\TT^d)}.
\end{equation}
\subsubsection*{Basic estimates and inequalities}
On $\TT^2$ we denote $\Lambda:=(-\Delta_x)^{1/2}$, that is
$\widehat{\Lambda f}(k)=|k|\,\widehat f(k)$, and we use the Riesz transforms
\[
\widehat{\mathcal R_j f}(k)=-i\frac{k_j}{|k|}\widehat f(k),\qquad k\neq0,\quad j=1,2,
\qquad
\mathcal R f=(\mathcal R_1 f,\mathcal R_2 f).
\]
They are bounded on Sobolev spaces: for all $s\in\RR$,
\begin{equation}\label{eq:Riesz_Hs}
\|\mathcal R_j f\|_{H^s(\TT^2)}\lesssim \|f\|_{H^s(\TT^2)},
\qquad
\|\mathcal R_j f\|_{\dot H^s(\TT^2)}\lesssim \|f\|_{\dot H^s(\TT^2)}.
\end{equation}
We also use the Sobolev embedding
\begin{equation}\label{eq:Sobolev_2d_embed}
H^2(\TT^2)\hookrightarrow W^{1,\infty}(\TT^2).
\end{equation}
On $H^s_0(\TT^2)$ the resolvent $(I+\Upsilon\Lambda)$ is invertible mode-by-mode with symbol
$(1+\Upsilon|k|)^{-1}$ for $k\neq0$, and in particular
\begin{equation}\label{eq:resolvent_Hs}
\|(I+\Upsilon\Lambda)^{-1}G\|_{H^s(\TT^2)}\le \|G\|_{H^s(\TT^2)},
\qquad G\in H^s_0(\TT^2).
\end{equation} \medskip

When restricting to one horizontal dimension we work on $\TT$ with variable $\xi$.
Fourier multipliers $\Lambda=|D_\xi|$ and $\Lambda^{-1}$ are understood with respect to $\xi$.
We use the Hilbert transform $\mathcal H$,
\[
\widehat{\mathcal H g}(k)=-i\,\sgn(k)\,\widehat g(k),\qquad k\in\ZZ,
\]
so that $\mathcal H^\ast=-\mathcal H$ on $L^2(\TT)$ and $\mathcal H^2=-I$ on mean-zero functions.
On nonzero modes we have the identities
\begin{equation}\label{eq:Hilbert_identities}
\Lambda=\mathcal H\,\partial_\xi,
\qquad
\partial_\xi\Lambda^{-1}=-\mathcal H,
\qquad
\partial_\xi\mathcal H=\mathcal H\partial_\xi,
\end{equation}
and, for mean-zero $f$ and any $s\in\RR$,
\begin{equation}\label{eq:H_isometry_hom}
\|\mathcal H f\|_{\dot H^s(\TT)}=\|f\|_{\dot H^s(\TT)}.
\end{equation}
We also use Tricomi's identities: for sufficiently regular real-valued $f,g$,
\begin{equation}\label{eq:Tricomi_general}
\mathcal H\!\big(f\,\mathcal H g+g\,\mathcal H f\big)
=
(\mathcal H f)(\mathcal H g)-f g,
\qquad
\mathcal H\!\big(f\,\mathcal H f\big)=\frac12\big((\mathcal H f)^2-f^2\big).
\end{equation}
Next, we gather products, commutators and auxiliary multipliers estimates in one dimension. For a linear operator $A$ and a function $f$ we write $\comm{A}{f}g:=A(fg)-f\,Ag$.
We will repeatedly use that for $s>\frac12$ and $f\in H^s_0(\TT)$,
\begin{equation}\label{eq:Sobolev_1d_embed_hom}
\|f\|_{L^\infty(\TT)}\lesssim_s \|f\|_{\dot H^s(\TT)},
\end{equation}
so in particular $H^s(\TT)$ is a Banach algebra for $s>\tfrac12$, with
\begin{equation}\label{eq:alg_inhom_1d}
\|fg\|_{H^s(\TT)}\lesssim_s \|f\|_{H^s(\TT)}\|g\|_{H^s(\TT)},
\qquad
\|fg\|_{\dot H^s(\TT)}\lesssim_s
\|f\|_{L^\infty}\|g\|_{\dot H^s}+\|g\|_{L^\infty}\|f\|_{\dot H^s}.
\end{equation}
We also use the Gagliardo--Nirenberg inequality
\begin{equation}\label{eq:GN_homogeneous_W14}
\|f\|_{\dot W^{1,4}(\TT)}^{2}=\|f_\xi\|_{L^4(\TT)}^{2}
\lesssim
\|f\|_{L^\infty(\TT)}\,\|f\|_{\dot H^{2}(\TT)}.
\end{equation}
Given $\Upsilon>0$, we define the Fourier multiplier
\begin{equation}\label{eq:def_T_multiplier_new}
\mathcal T:=(I+\Upsilon\Lambda)^{-1}\Lambda,
\qquad
\widehat{\mathcal T f}(k)=\frac{|k|}{1+\Upsilon|k|}\,\widehat f(k),\quad k\neq0,
\end{equation}
so that for all $s\in\RR$ and mean-zero $f$,
\begin{equation}\label{eq:T_mapping_hom}
\|\mathcal T f\|_{\dot H^s(\TT)}\le \|f\|_{\dot H^s(\TT)},
\qquad
\|\Lambda \mathcal T f\|_{\dot H^s(\TT)}\lesssim_{\Upsilon}\|f\|_{\dot H^{s+1}(\TT)}.
\end{equation} \medskip

Finally, for $\nu\in(0,1]$ consider $\mathcal J_\nu$ the heat kernel on $\TT^d$ ($d=1,2$) defined by
\[
\widehat{\mathcal J_\nu f}(k):=e^{-\nu |k|^2}\,\widehat f(k),\qquad k\in\ZZ^d.
\]
Then $\mathcal J_\nu$ is self-adjoint on $L^2$, it commutes with $\partial_\xi$ and $\Lambda$
(and with $\mathcal H$ when $d=1$), and $\mathcal J_\nu f\to f$ in $H^s$ as $\nu\to0$
for all $s\in\RR$. Moreover, for every $s\in\RR$ and $m\ge0$,
\begin{equation}\label{eq:mollifier_bounds}
\|\mathcal J_\nu f\|_{H^s}\lesssim \|f\|_{H^s},
\qquad
\|\mathcal J_\nu f\|_{H^{s+m}}\lesssim \nu^{-m}\,\|f\|_{H^s}.
\end{equation}

To close this preliminary section, we record a standard representation formula for the Poisson problem:
\begin{lemma}[Poisson problem with Dirichlet data]\label{lem:Poisson}
Let $b=b(x,x_3)$ and $g=g(x)$ be smooth, $2\pi$--periodic in $x\in\TT^2$, and assume that
$b(\cdot,x_3)$ decays sufficiently fast as $x_3\to-\infty$. Consider the problem
\begin{equation}\label{eq:Poisson_halfspace_DN}
\begin{cases}
\Delta u = b & \text{in }\Omega:=\TT^2\times(-\infty,0),\\
u=g & \text{on }\Gamma:=\TT^2\times\{0\},\\
u(\cdot,x_3)\to 0 & \text{as }x_3\to-\infty.
\end{cases}
\end{equation}
Then \eqref{eq:Poisson_halfspace_DN} admits a unique solution $u$, and for every nonzero Fourier mode
$k\in\ZZ^2\setminus\{0\}$ one has
\begin{equation}\label{eq:DN_formula_mode}
\partial_{x_3}\widehat{u}(k,0)
=
|k|\,\widehat{g}(k)
+\int_{-\infty}^{0} e^{|k|y_3}\,\widehat{b}(k,y_3)\,dy_3.
\end{equation}
Equivalently, in operator form on mean--zero functions,
\begin{equation}\label{eq:DN_formula_operator}
\partial_{x_3}u(\cdot,0)=\Lambda g+\int_{-\infty}^{0} e^{y_3\Lambda}\,b(\cdot,y_3)\,dy_3.
\end{equation}
\end{lemma}

\begin{proof}
Taking the Fourier series in $x$ reduces \eqref{eq:Poisson_halfspace_DN} to
\[
(\partial_{x_3}^2-|k|^2)\widehat u(k,x_3)=\widehat b(k,x_3),\qquad \widehat u(k,0)=\widehat g(k),
\qquad \widehat u(k,x_3)\to 0\ \text{as }x_3\to-\infty,
\]
which is solved explicitly by variation of constants. Differentiating at $x_3=0$
gives \eqref{eq:DN_formula_mode}. A detailed derivation can be found for instance in
\cite[Lemma A.1]{Granero-Scrobogna,GraneroBelinchonScrobogna2019}.
\end{proof}

\subsection{Organization of the manuscript}
The manuscript is organized as follows. In Section~\ref{sec:setting} we introduce the full three-dimensional Euler--plate system, discuss the geometric bending energy and the associated structural forces, and rewrite the problem in potential variables. We then nondimensionalize the equations and derive an Arbitrary Lagrangian--Eulerian (ALE) formulation on a fixed reference domain. In Section~3 we perform a small-steepness asymptotic expansion of the ALE system and obtain, up to quadratic order, two bidirectional nonlocal interface models for the surface displacement. In Section~4 we derive two one-dimensional unidirectional reductions capturing one-way propagation while retaining the leading dispersive and dissipative effects. In Section~\ref{Wp:bi} we prove local well-posedness for the first of the bidirectional model for small $H^3$ initial data. In Section~\ref{Wp:uni} we establish well-posedness results for the unidirectional models, including local well-posedness for arbitrary data and global well-posedness for small data. Finally, Appendix~\ref{appendix:geometry} collects geometric identities for graph surfaces, and Appendix~\ref{appendixB:non-dimensional} contains the details of the non-dimensionalisation.

	\section{Setting of the problem}
	\label{sec:setting}
	
	\subsection{The fluid structure interaction system}
	In this section we present a complete formulation of the
	fluid--structure interaction system under consideration.
	The model consists of a three--dimensional incompressible,
	inviscid fluid evolving below a deformable elastic surface. The fluid occupies the time-dependent domain
	\begin{equation}
	\Omega(t)= \{(x_1,x_2,x_3)\in\mathbb{R}^{3}: -L \pi <x_{1},x_{2}<L\pi, \quad x_3 < \eta(t,x_1,x_2), \quad t\in[0,T]\},
	\end{equation}
	where $L>0$ and the elastic free--boundary surface occupied by the plate is represented as the graph
	\begin{equation}
	\Gamma(t)= \{(x_1,x_2,\eta(x_{1},x_{2},t)): (x_{1},x_{2})\in L\mathbb{S}^{1}\times L\mathbb{S}^{1},  \quad t\in[0,T]\},
	\end{equation}
	The fluid is assumed to be incompressible, inviscid and of constant density $\rho_{f}>0$.  Therefore, velocity $u=(u_1,u_2,u_3)$ and pressure $p$ in the fluid region satisfy the incompressible Euler equations
	\begin{subequations}\label{eq:Euler}
\begin{align}
\rho_f\big(\partial_t u+(u\cdot\nabla)u\big)
&= -\nabla p-\rho_f g\,e_3
&& \text{in }\Omega(t), \label{eq:Euler:1}\\
\nabla\cdot u
&=0
&& \text{in }\Omega(t). \label{eq:Euler:2}
\end{align}
\end{subequations}

	where $g>0$ denotes gravity and $e_{3}=(0,0,1)$. Since the domain extends infinitely deep, we impose the 
	decay condition
	\begin{equation}
	u(t,x_1,x_2,x_3) \to 0,
	\qquad
	p(t,x_1,x_2,x_3) \to p_\infty(t),
	\qquad x_3\to -\infty.
	\label{eq:decay-deep}
	\end{equation}
	Moreover, the fluid is assumed to be periodic in the horizontal directions, i.e., for $k=(k_{1},k_{2})\in\mathbb{Z}^{2}$ we impose
	\begin{equation}
	u(t,x_{1}+2\pi L k_{1},\, x_{2}+2\pi L k_{2},\, x_{3})
	= u(t,x_{1},x_{2},x_{3}),
	\qquad t\in[0,T],
	\label{eq:periodicity-u}
	\end{equation}
	and analogously for the pressure
	\begin{equation}
	p(t,x_{1}+2\pi L k_{1},\, x_{2}+2\pi L k_{2},\, x_{3})
	= p(t,x_{1},x_{2},x_{3}).
	\label{eq:periodicity-p}
	\end{equation}
	
	\subsection{Boundary conditions, geometric bending energy and structural forces}
	The motion of the free boundary $\Gamma(t)$ is governed by two conditions: a kinematic condition expressing that the interface is transported by the fluid, and a dynamic condition expressing balance of normal stresses.
	
	\subsubsection*{The kinematic boundary condition}
	Since the free surface consists of material particles of the fluid, its normal velocity must coincide with that of the fluid. Noticing that the upward unit normal to $\Gamma(t)$ is given by
	\[
	n(x,t)
	= \frac{(-\nabla_{x}\eta(x,t),\,1)^{T}}{\sqrt{1+|\nabla_{x}\eta(x,t)|^{2}}}, \quad x=(x_{1},x_{2}),
	\]
	
	we have that
	\begin{equation}
	\partial_{t}\eta
	= u\cdot
	\sqrt{1+|\nabla_{x}\eta|^{2}}\, n, \mbox { on } \Gamma(t)
	\label{eq:kinematic}
	\end{equation}
	
	\subsubsection*{Geometric bending energy, inertia, and damping}
	
	Before formulating the dynamic boundary condition, we describe the forces
	acting on the elastic plate. These consist of (i) a nonlinear bending
	force obtained as the first variation of a curvature-dependent energy,
	(ii) an inertial response arising from the mass density of the plate, and
	(iii) a dissipative contribution modelling structural damping. Together
	these terms form the structural side of the evolution equation on the
	free boundary.
	
	Following the hydroelastic shell model of Plotnikov and Toland
	\cite{PlotnikovToland}, we assume that the restoring elastic force derives
	from a curvature-based bending energy of the form
	\begin{equation}
	E_b[\eta]
	= \int_{L\mathbb{S}^{1}\times L\mathbb{S}^{1}}
	W_b\big(H(\eta)\big)\,
	\sqrt{1+|\nabla_{x}\eta|^2}\,dx,
	\label{eq:bending-general}
	\end{equation}
	where $H(\eta)$ denotes the mean curvature of the graph $\Gamma(t)$ and 
	$W_b$ is a prescribed bending density. In the context of hydroelasticity,
	membrane strains are typically negligible compared to bending, and the
	energy depends solely on curvature. A widely used and physically relevant
	choice is
	\[
	W_b(H)=\frac{B}{2}H^2,
	\]
	which leads to the Willmore-type energy
	\begin{equation}
	E_b[\eta]
	= \int_{L\mathbb{S}^{1}\times L\mathbb{S}^{1}}
	\frac{B}{2} H(\eta)^2\,\sqrt{1+|\nabla_{x}\eta|^2}\,dx.
	\label{eq:Eb}
	\end{equation}
	
	The elastic force associated with this energy is obtained from its first
	variation. As shown in \cite{PlotnikovToland}, one obtains the fourth-order
	nonlinear geometric operator
	\begin{equation}
	\mathcal{E}(\eta)
	= \frac{B}{2}\Delta_\Gamma H(\eta)
	+ B H^3(\eta) - BH(\eta)K(\eta),
	\label{eq:E-operator}
	\end{equation}
	where $\Delta_\Gamma$ is the Laplace--Beltrami operator and $K$ is the
	Gauss curvature of the evolving surface. The operator retains the full
	geometric nonlinearity of the plate and depends explicitly on $\eta$,
	$\nabla_{x}\eta$, and $\nabla^2_{x}\eta$. For convenience, explicit expressions
	for $H$, $K$, $\Delta_\Gamma$, and the metric coefficients in terms of
	$\eta$ are collected in Appendix~\ref{appendix:geometry}. \medskip
	
	In addition to curvature-driven elasticity, the plate possesses inertia:
	with mass density $\rho_s h$ per unit area, vertical acceleration
	produces the inertial force
	\begin{equation}
	\rho_s h\,\partial_{tt}\eta.
	\label{eq:inertia}
	\end{equation}
	Here $\rho_s$ denotes the volumetric density of the solid material and
	$h>0$ is the (constant) thickness of the plate, so that $\rho_s h$ is the
	mass per unit area of the structure. This hyperbolic term reflects the
	kinetic response of the elastic sheet and plays a central role in the
	time-dependent hydroelastic dynamics. Finally, structural or viscoelastic damping may be present. A standard
	and analytically convenient choice is Kelvin--Voigt damping, modelled by
	\begin{equation}
	\mathcal{D}(\eta_t)=-\gamma\, \partial_{t}\Delta_{x}\eta,
	\qquad \gamma\ge 0,
	\label{eq:damping}
	\end{equation}
	corresponding to a Rayleigh dissipation potential $\mathcal{R}(\eta)=
	\frac{\gamma}{2}\int | \partial_{t}\nabla_{x}\eta|^2\,dx$. This term introduces
	high-order dissipation compatible with the geometric structure of the
	bending operator and appears naturally in Koiter-type plate models
	\cite{BKMT2024}. \medskip
	
	These three contributions combine to give the complete structural response
	of the plate: the nonlinear bending operator $\mathcal{E}(\eta)$ in
	\eqref{eq:E-operator}, the inertial force $\rho_s h\,\eta_{tt}$ in
	\eqref{eq:inertia}, and the damping term $\mathcal{D}(\eta_t)$ in
	\eqref{eq:damping}. Together they form the left-hand side of the evolution
	equation on the free boundary, and their balance with the normal fluid
	traction yields the dynamic boundary condition on $\Gamma(t)$ considered
	in the next subsection.

	\subsubsection*{Dynamic boundary condition: balance of normal forces}
	
	With the elastic, inertial, and damping forces identified, we now state
	the dynamic boundary condition on the free surface.
	For an inviscid, incompressible fluid, the Cauchy stress tensor is
	$\sigma_f=-pI$.
	The traction exerted by the fluid across the interface with outward unit
	normal $n$ is therefore $\sigma_f n = -p\,n$.
	
	In the plate equation, however, the relevant quantity is the normal force
	exerted by the fluid on the structure. By Newton’s third law, this force is
	the opposite of the traction exerted by the structure on the fluid.
	Consequently, the normal fluid force acting on the plate is
	\[
	T_{\mathrm{fluid}}
	= -(\sigma_f n)\cdot n = p,
	\qquad \text{on } \Gamma(t).
	\]
	
	Balancing this fluid force with the inertial, elastic, and damping
	contributions of the plate yields Newton’s second law on the moving
	boundary:
	\begin{equation}
	\rho_s h\,\partial_{tt}\eta
	+ \mathcal{E}(\eta)
	- \gamma\,\partial_t\Delta_{x}\eta
	= p
	\quad \text{on } \Gamma(t).
	\label{eq:dynamic}
	\end{equation}
	
	Equation~\eqref{eq:dynamic} constitutes the fully nonlinear dynamic
	boundary condition governing the evolution of the elastic surface.

	\subsection{The full Euler–plate system and dimensionless system}
	Equation~\eqref{eq:dynamic}, together with the kinematic
	condition~\eqref{eq:kinematic} and the Euler equations~\eqref{eq:Euler},
	completes the formulation of the fluid--structure interaction problem. For
	convenience, we collect the full system below:
	\begin{subequations}\label{eq:FSI_system}
\begin{align}
\rho_{f}\left(\partial_t u + (u\cdot\nabla)u \right)
&= - \nabla p  -\rho_{f} g\, e_3
&& \text{in }\Omega(t), \label{eq:FSI:1}\\
\nabla\cdot u
&= 0
&& \text{in }\Omega(t), \label{eq:FSI:2}\\
\partial_{t}\eta
&= u\cdot \sqrt{1+|\nabla_{x}\eta|^{2}}\, n
&& \text{on }\Gamma(t), \label{eq:FSI:3}\\
\rho_s h\,\partial_{tt}\eta
+ \mathcal{E}(\eta)
- \gamma\,  \partial_{t}\Delta_{x}\eta
&= p
&& \text{on }\Gamma(t). \label{eq:FSI:4}
\end{align}
\end{subequations}
	where $\mathcal{E}(\eta)$ is given in \eqref{eq:E-operator}. Assuming that the flow is irrotational, we introduce a potential $\phi$
	such that $u=\nabla\phi$   and denote by $\psi(x_{1},x_{2},t)= \phi(t,x_{1},x_{2},\eta(x_{1},x_{2},t))$ the surface potential.  
	Taking the trace of the Bernoulli relation
	\[
	p = -\rho_f\big( \partial_{t}\phi+ \tfrac12|\nabla\phi|^2 + g x_3\big)
	\quad\text{in }\Omega(t),
	\]
	yields
	\[
	p
	= -\rho_f\Big( \partial_{t}\phi + \tfrac12|\nabla\phi|^2 + g\,\eta\Big)
	\quad \mbox{on } \Gamma(t).
	\]
	Moreover, using the chain rule, we can write
	\[ \partial_{t}\psi=\partial_{t}\phi + \partial_{x_3}\phi\partial_{t}\eta , \quad \mbox{on } \Gamma(t). \]
	and using the kinematic boundary condition in  \eqref{eq:FSI:3},
	\[ \partial_{t}\psi=\partial_{t}\phi + \partial_{x_3}\phi\left(\nabla\phi\cdot
	\sqrt{1+|\nabla_{x}\eta|^{2}}\, n\right) , \quad \mbox{on } \Gamma(t). \]
	Thus, the fluid--structure interaction system admits the following potential formulation
\begin{subequations}\label{eq:potential_system}
\begin{align}
\Delta\phi &= 0 && \text{in }\Omega(t), \label{eq:potential:1}\\
\phi &= \psi && \text{on }\Gamma(t), \label{eq:potential:2}\\
\partial_{t}\eta
&= \nabla\phi\cdot\sqrt{1+|\nabla_{x}\eta|^2}\,n
&& \text{on }\Gamma(t), \label{eq:potential:3}\\
\partial_t \psi
&= -\frac{\rho_s h}{\rho_f}\, \partial_{t}^2\eta
- \frac{1}{\rho_f}\,\mathcal{E}(\eta)
+\frac{\gamma}{\rho_f}\, \partial_{t}\Delta_{x}\eta \notag\\
&\quad
+ \partial_{x_3}\phi\big(\nabla\phi\cdot
\sqrt{1+|\nabla_{x}\eta|^{2}}\, n\big)
- \frac12|\nabla\phi|^2
- g\eta
&& \text{on }\Gamma(t). \label{eq:potential:4}
\end{align}
\end{subequations}
	where we recall that $\mathcal{E}(\eta)$ is given in \eqref{eq:E-operator}. 
	
	\begin{remark}[General bending energies]
		\label{rmk:general-bending}
		The structure of the model remains unchanged for general
		bending energies of the form
		\[
		E_b[\eta]
		= \int_{L\mathbb{S}^{1}\times L\mathbb{S}^{1}}
		W_b(H(\eta)),\sqrt{1+|\nabla_{x}\eta|^2}\,dx,
		\]
		with $W_b$ smooth.   The corresponding elastic operator is
		\[
		\mathcal{E}(\eta)
		=
		\frac12\Delta_\Gamma W_b'(H)
		+ 2H(HW_b'(H) - W_b(H))
		- K W_b'(H).
		\]
		The choice $W_b(H)=\frac{B}{2}H^2$ is made for physical
		and analytic clarity, but more general models arise in
		Cosserat shell theory.
	\end{remark}
	
	\begin{remark}[Dependence on the Gauss curvature]
		Bending energies depending on both $H$ and $K$ can also be
		considered:
		\[
		E_b[\eta]
		=
		\int_{L\mathbb{S}^{1}\times L\mathbb{S}^{1}}		W(H,K)\,\sqrt{1+|\nabla_{x}\eta|^2}\,dx.
		\]
		Such energies lead to additional terms involving the
		variation of $K$ and may require higher regularity.
		These appear in the literature on elastic shells and
		biomembranes, though we do not pursue them here.
	\end{remark}
	
	\subsubsection*{Dimensionless quantities and system}In order to derive the asymptotic models, we provide the non-dimensionalization of the system \eqref{eq:potential:1}-\eqref{eq:potential:4}. Let $L$ denote the typical horizontal wavelength and $H$ the typical
	vertical amplitude of the elastic sheet. We introduce
	the dimensionless spatial and temporal variables 
	\[
	x = L\,\tilde x, \qquad x_3 = L\,\tilde x_3, \qquad
	t = \sqrt{\frac{L}{g}}\,\tilde t,
	\]
	and the non-dimensional surface displacement, potential and surface potential
	\[
	\eta(x,t)=H\,\tilde\eta(\tilde x,\tilde t),\qquad
	\phi(x,x_3,t)=H\sqrt{gL}\,\tilde\phi(\tilde x,\tilde x_3,\tilde t),\qquad
	\psi(x,t)=H\sqrt{gL}\,\tilde\psi(\tilde x,\tilde t).
	\]
	Accordingly, we define the non-dimensionalized fluid domain 
	\[
	\widetilde{\Omega}(t)
	= \{(\tilde x_1,\tilde x_2,\tilde x_3)\in\mathbb{R}^3:
	-\pi<\tilde x_1,\tilde x_2<\pi,\; \tilde x_3<\varepsilon\,\tilde \eta(\tilde x,\tilde t), \quad t\in [0,T]\},
	\]
	and free surface
	\[
	\tilde \Gamma(t)
	= \{(\tilde x_1,\tilde x_2,\varepsilon\,  \eta(\tilde x,\tilde t)):
	-\pi<\tilde x_1,\tilde x_2<\pi,\quad t\in [0,T]\}.
	\]
	Here we introduce the nondimensional parameters
	\[
	\Upsilon := \frac{\rho_s h}{\rho_f L},\qquad
	\beta := \frac{HB}{\rho_f g L^4},\qquad
	\delta := \frac{\gamma}{\rho_f\sqrt{gL^3}},\qquad
	\varepsilon := \frac{H}{L}.
	\]

	The non-dimensionalisation of the potential formulation
\eqref{eq:potential:1}--\eqref{eq:potential:4} is carried out in detail in
Appendix~\ref{appendixB:non-dimensional}. In particular, after introducing the
dimensionless variables and parameters and dropping tildes for notational
simplicity, the fluid--structure system takes the form
	\begin{align}
	& \Delta\phi = 0 &&\text{in }\Omega(t),\label{maineq1}\\[0.3ex]
	& \phi=\psi &&\text{on }\Gamma(t),\label{maineq2}\\[0.3ex]
	& \partial_{t}\eta
	= \nabla\phi\cdot (-\varepsilon\nabla_{x}\eta,1)
	&&\text{on }\Gamma(t),\label{maineq3}\\[0.3ex]
	& \partial_{t}\psi
	= -\Upsilon\,\partial_{tt}\eta
	-\beta\,\mathcal{E}(\eta;\varepsilon)
	+ \delta\,\partial_{t}\Delta_{x}\eta 
	+ \varepsilon\,\partial_{x_3}\phi
	\left(\nabla\phi\cdot
	(-\varepsilon\nabla_{x}\eta,1)\right)
	- \frac{\varepsilon}{2}|\nabla\phi|^2
	- \eta,
	&&\text{on }\Gamma(t).\label{maineq4}
	\end{align}
	The
	dimensionless operator $\mathcal{E}(\eta;\varepsilon)$ is given
	explicitly by
	\[
	\mathcal{E}(\eta;\varepsilon)
	= \frac12\,\mathcal{L}_\Gamma(\eta;\varepsilon)
	+ \varepsilon^2\Big(
	\mathcal{H}(\eta;\varepsilon)^3
	- \mathcal{H}(\eta;\varepsilon)\,
	\mathcal{K}(\eta;\varepsilon)
	\Big),
	\]
	where
	\[
	\mathcal{L}_\Gamma(\eta;\varepsilon)
	= \frac{1}{\sqrt{\alpha}}\,
	\partial_{x_i}
	\left(
	\sqrt{\alpha}\,\alpha^{ij}\,
	\partial_{x_j}\mathcal{H}(\eta;\varepsilon)
	\right),
	\]
	and
	\[
	\mathcal{H}(\eta;\varepsilon)
	= \frac{1}{2\sqrt{\alpha}}
	\Big(
	\alpha^{11}\eta_{x_1x_1}
	+ 2\alpha^{12}\eta_{x_1x_2}
	+ \alpha^{22}\eta_{x_2x_2}
	\Big), \quad \mathcal{K}(\eta;\varepsilon)
	= \frac{
		\eta_{x_1x_1}\eta_{x_2x_2}
		- \eta_{x_1x_2}^2
	}{\alpha^2}.
	\]
	
	\[
	\alpha = 1+\varepsilon^2|\nabla_{x}\eta|^2,
	\qquad
	\alpha^{11}=\frac{1+\varepsilon^2\eta_{x_2}^2}{\alpha},\quad
	\alpha^{22}=\frac{1+\varepsilon^2\eta_{x_1}^2}{\alpha},\quad
	\alpha^{12}=-\,\frac{\varepsilon^2\,\eta_{x_1}\eta_{x_2}}{\alpha}.
	\]
	
	\subsection{The Arbitrary Lagrangian-Eulerian (ALE) formulation}
	In this section, our goal is to transform the free-boundary problem into a
	fixed reference domain
	\[\Omega=\mathbb{T}^{2}\times (-\infty,0), \quad \Gamma=\mathbb{T}^{2}\times\{ 0\}.\]
	To this end, we introduce a time-dependent
	diffeomorphism
	\[
	\Psi(\cdot,t):\Omega \longrightarrow \Omega(t),
	\]
	defined by
	\[
	\Psi(x_1,x_2,x_3,t)
	=
	\big(x_1,\;x_2,\;x_3+\varepsilon\,\eta(x_1,x_2,t)\big).
	\]
	This mapping straightens the moving free surface
	\(x_3=\varepsilon\eta(x_1,x_2,t)\) into the fixed boundary \(x_3=0\).
	
	Given a scalar function \(f\) defined on \(\Omega(t)\), we define its ALE
	counterpart on the reference domain \(\Omega\) by composition with \(\Psi\),
	that is,
	\[
	F = f\circ\Psi.
	\]
	In particular, we introduce the ALE potential
	\[
	\Phi = \phi \circ \Psi.
	\]
	
	We now compute the gradient of the diffeomorphism \(\Psi\). A direct
	calculation yields
	\[
	\nabla \Psi
	=
	\begin{pmatrix}
	1 & 0 & 0\\
	0 & 1 & 0\\
	\varepsilon\,\partial_{x_1}\eta(x,t) &
	\varepsilon\,\partial_{x_2}\eta(x,t) &
	1
	\end{pmatrix}, \quad  (\nabla\Psi)^{-1}
	=
	\begin{pmatrix}
	1 & 0 & 0\\
	0 & 1 & 0\\
	-\varepsilon\,\partial_{x_1}\eta(x,t) &
	-\varepsilon\,\partial_{x_2}\eta(x,t) &
	1
	\end{pmatrix}.
	\]
	
	Adopting the Einstein summation convention and defining the ALE pullback
	\(F := f \circ \Psi\), the chain rule yields
	\[
	(\partial_{x_j} f)\circ\Psi = A_j^{\,k}\,\partial_{x_k}F,
	\qquad A=(\nabla\Psi)^{-1}.
	\]
	Consequently,
	\[
	(\Delta f)\circ\Psi
	= A_j^{\,i}\,\partial_{x_i}\!\left(A_j^{\,k}\,\partial_{x_k}F\right).
	\]
	
	Using the explicit form of \(A\) and expanding up to first order in
	\(\varepsilon\), we obtain
	\[
	(\Delta f)\circ\Psi
	= \Delta F
	- \varepsilon\Big[
	(\eta_{x_1x_1}+\eta_{x_2x_2})\,\partial_{x_3}F
	+ \eta_{x_1}(\partial_{x_3x_1}F+\partial_{x_1x_3}F)
	+ \eta_{x_2}(\partial_{x_3x_2}F+\partial_{x_2x_3}F)
	\Big]
	+ \mathcal{O}(\varepsilon^2).
	\]
	
	Finally, the composition of equations (\ref{maineq1})-(\ref{maineq4}) the system with the diffeomorphism reads to
	
	\begin{align}
	& \Delta\Phi
	-\varepsilon\Big[
	(\eta_{x_1x_1}+\eta_{x_2x_2})\,\Phi_{x_3}
	+\eta_{x_1}\big(\Phi_{x_3x_1}+\Phi_{x_1x_3}\big)
	+\eta_{x_2}\big(\Phi_{x_3x_2}+\Phi_{x_2x_3}\big)
	\Big]
	=\mathcal{O}(\varepsilon^2)
	&&\text{in }\Omega, \label{ALEsys1}\\[0.4ex]
	& \Phi=\psi
	&&\text{on }\Gamma, \label{ALEsys2}\\[0.4ex]
	& \partial_t \eta
	= \Phi_{x_3}-\varepsilon\big(\eta_{x_1}\Phi_{x_1}+\eta_{x_2}\Phi_{x_2}\big)
	+\mathcal{O}(\varepsilon^2)
	&&\text{on }\Gamma, \label{ALEsys3}\\[0.4ex]
	& \partial_t\psi
	= -\Upsilon\,\partial_{tt}\eta
	-\frac{\beta}{4}\,\Delta_{x}^{2}\eta
	+ \delta\,\partial_{t}\Delta_{x}\eta
	-\eta
	+\frac{\varepsilon}{2}\Big(\Phi_{x_3}^2-|\nabla_x\Phi|^2\Big)
	+\mathcal{O}(\varepsilon^2)
	&&\text{on }\Gamma. \label{ALEsys4}
	\end{align}\begin{remark}[Leading--order reduction of the bending operator]
		At leading order in the small--steepness parameter $\varepsilon$, the
		geometric bending operator $\mathcal{E}(\eta;\varepsilon)$ reduces to a
		linear biharmonic operator. Indeed,
		\[
		\alpha = 1+\mathcal{O}(\varepsilon^2), \qquad
		\alpha^{ij}=\delta_{ij}+\mathcal{O}(\varepsilon^2),
		\]
		so that
		\[
		\mathcal{H}(\eta;\varepsilon)
		= \tfrac12\,\Delta_x\eta + \mathcal{O}(\varepsilon^2),
		\qquad
		\mathcal{E}(\eta;\varepsilon)
		= \tfrac14\,\Delta_x^2\eta + \mathcal{O}(\varepsilon^2).
		\]
		Accordingly, in equation \eqref{ALEsys4} the elastic response of the plate is governed by a biharmonic operator, while nonlinear geometric effects are encapsulated in the $\mathcal{O}(\varepsilon^2)$ terms.
	\end{remark}
	
	\section{Formal asymptotic expansion}
	
	We now derive a hierarchy of approximate models by performing a formal
	small--steepness expansion in the parameter $\varepsilon$ in the ALE system
	\eqref{ALEsys1}--\eqref{ALEsys4}. Specifically, we seek $\Phi$, $\psi$ and
	$\eta$ as power series in $\varepsilon$ of the form
	\begin{align}
	\Phi(x_1,x_2,x_3,t) &= \sum_{n=0}^{\infty}\varepsilon^{n}\,\Phi^{(n)}(x_1,x_2,x_3,t), 
	\label{eq:ansatz-Phi}\\
	\psi(x_1,x_2,t) &= \sum_{n=0}^{\infty}\varepsilon^{n}\,\psi^{(n)}(x_1,x_2,t),
	\qquad
	\eta(x_1,x_2,t) = \sum_{n=0}^{\infty}\varepsilon^{n}\,\eta^{(n)}(x_1,x_2,t),
	\label{eq:ansatz-psi-eta}
	\end{align}

	Substituting \eqref{eq:ansatz-Phi}--\eqref{eq:ansatz-psi-eta} into
	\eqref{ALEsys1}--\eqref{ALEsys4} and collecting terms of equal powers of
	$\varepsilon$ yields, at each order $n\ge 0$, a linear boundary value problem
	for $\Phi^{(n)}$ in $\Omega$ coupled to evolution equations for
	$\psi^{(n)}$ and $\eta^{(n)}$ on $\Gamma=\{x_3=0\}$. In particular, the leading
	order $n=0$ provides the principal (linear) hydroelastic model, while $n=1$
	captures the first nonlinear correction. \medskip
	
	The initial data are expanded consistently as
	\[
	\Phi(x,x_3,0)=\sum_{n=0}^{\infty}\varepsilon^{n}\,\Phi^{(n)}(x,x_3,0),\qquad
	\psi(x,0)=\sum_{n=0}^{\infty}\varepsilon^{n}\,\psi^{(n)}(x,0),\qquad
	\eta(x,0)=\sum_{n=0}^{\infty}\varepsilon^{n}\,\eta^{(n)}(x,0),
	\]
	so that $\Phi^{(n)}(\cdot,\cdot,0)$, $\psi^{(n)}(\cdot,0)$ and
	$\eta^{(n)}(\cdot,0)$ represent the $n$-th order initial profiles in the
	expansion.
	
	\subsubsection*{\underline{The case $n=0$}} We substitute the formal expansions \eqref{eq:ansatz-Phi}-\eqref{eq:ansatz-psi-eta}
	into the ALE system \eqref{ALEsys1}--\eqref{ALEsys4} and retaining only the leading-order contributions yields the following system for $(\Phi^{(0)},\psi^{(0)},\eta^{(0)})$:
	\begin{align}
	&\Delta \Phi^{(0)}=0
	&&\text{in }\Omega, \label{eq:n0-1}\\[0.4ex]
	&\Phi^{(0)}=\psi^{(0)}
	&&\text{on }\Gamma, \label{eq:n0-2}\\[0.4ex]
	&\partial_t \eta^{(0)}=\partial_{x_3}\Phi^{(0)}
	&&\text{on }\Gamma, \label{eq:n0-3}\\[0.4ex]
	&\partial_t\psi^{(0)}
	=-\Upsilon\,\partial_{tt}\eta^{(0)}
	-\frac{\beta}{4}\,\Delta_x^2\eta^{(0)}
	+\delta\,\partial_t\Delta_x\eta^{(0)}
	-\eta^{(0)}
	&&\text{on }\Gamma. \label{eq:n0-4}
	\end{align}
	
	Taking the Fourier series in $x\in\TT^2$, the unique decaying solution of
	\eqref{eq:n0-1}--\eqref{eq:n0-2} is
	\[
	\widehat{\Phi^{(0)}}(k,x_3,t)=\widehat{\psi^{(0)}}(k,t)\,e^{|k|x_3},
	\qquad k\in\ZZ^2,\quad x_3\le 0,
	\]
	and in particular
	\[
	\partial_{x_3}\Phi^{(0)}\big|_{\Gamma}=\Lambda\,\psi^{(0)}.
	\]
	Therefore, the kinematic condition \eqref{eq:n0-3} becomes
	\begin{equation}
	\partial_t \eta^{(0)}=\Lambda\,\psi^{(0)} \qquad \text{on }\Gamma.
	\label{eq:n0-kin-Lambda}
	\end{equation}
	Differentiating \eqref{eq:n0-kin-Lambda} in time and using \eqref{eq:n0-4}
	yields a closed equation for $\eta^{(0)}$:
	\begin{equation}
	\partial_{tt}\eta^{(0)}
	= \Lambda\Big(
	-\Upsilon\,\partial_{tt}\eta^{(0)}
	-\frac{\beta}{4}\,\Delta_x^2\eta^{(0)}
	+\delta\,\partial_t\Delta_x\eta^{(0)}
	-\eta^{(0)}
	\Big),
	\qquad \text{on }\Gamma.
	\label{eq:n0-eta-closed}
	\end{equation}
	Equivalently,
	\begin{equation}
	\big(I+\Upsilon\Lambda\big)\partial_{tt}\eta^{(0)}
	+\frac{\beta}{4}\,\Lambda\Delta_x^2\eta^{(0)}
	-\delta\,\Lambda\,\partial_t\Delta_x\eta^{(0)}
	+\Lambda\eta^{(0)}=0.
	\label{eq:n0-eta-closed2}
	\end{equation}
	
	Let $\widehat{\eta}^{(0)}(k,t)$ denote the Fourier coefficient of $\eta^{(0)}$.
	For $k\neq0$, using $\widehat{\Lambda f}(k)=|k|\widehat f(k)$ and
	$\widehat{\Delta_x f}(k)=-|k|^2\widehat f(k)$, equation \eqref{eq:n0-eta-closed} becomes
	\begin{equation}
	(1+\Upsilon|k|)\,\partial_t^2\widehat{\eta}^{(0)}(k,t)
	+\delta\,|k|^3\,\partial_t\widehat{\eta}^{(0)}(k,t)
	+\Big(|k|+\frac{\beta}{4}|k|^5\Big)\widehat{\eta}^{(0)}(k,t)=0.
	\label{eq:n0FourierODE}
	\end{equation}
	\subsubsection*{\underline{The case $n=1$}} Similarly, in the case $n=1$ the leading-order contributions yields the following system for $(\Phi^{(1)},\psi^{(0)},\eta^{(1)})$:
	\begin{align}
	& \Delta\Phi^{(1)}
	=\Big[
	(\eta^{(0)}_{x_1x_1}+\eta^{(0)}_{x_2x_2})\,\Phi^{(0)}_{x_3}
	+\eta^{(0)}_{x_1}\big(\Phi^{(0)}_{x_3x_1}+\Phi^{(0)}_{x_1x_3}\big)
	+\eta^{(0)}_{x_2}\big(\Phi^{(0)}_{x_3x_2}+\Phi^{(0)}_{x_2x_3}\big)
	\Big]
	&&\text{in }\Omega, \label{ALEsys1}\\[0.4ex]
	& \Phi^{(1)}=\psi^{(1)}
	&&\text{on }\Gamma, \label{ALEsys2}\\[0.4ex]
	& \partial_t \eta^{(1)}
	= \Phi^{(1)}_{x_3}-\big(\eta^{(0)}_{x_1}\Phi^{(0)}_{x_1}+\eta^{(0)}_{x_2}\Phi^{(0)}_{x_2}\big)
	&&\text{on }\Gamma, \label{ALEsys3}\\[0.4ex]
	& \partial_t\psi^{(1)}
	= -\Upsilon\,\partial_{tt}\eta^{(1)}
	-\frac{\beta}{4}\,\Delta_{x}^{2}\eta^{(1)}
	+ \delta\,\partial_{t}\Delta_{x}\eta^{(1)}
	-\eta^{(1)}
	+\frac{1}{2}\Big((\Phi^{(0)}_{x_3})^2-|\nabla_x\Phi^{(0)}|^2\Big)
	&&\text{on }\Gamma. \label{ALEsys4}
	\end{align}

	A convenient form for the bulk source term is to write
	\[ b(x_1,x_2,x_3,t)= \Delta_{x}\eta^{(0)}\partial_{x_3}\Phi^{(0)}+2\nabla_{x}\eta^{(0)}\cdot\nabla_{x}\partial_{x_3}\Phi^{(0)}.\]
	Therefore, by using Lemma \ref{lem:Poisson} we can compute 
	\begin{equation}\label{formula:poisson}
	\partial_{x_3}\widehat{\Phi^{(1)}}(k,0,t)=|k|\widehat{\psi^{(1)}}(k,t)+\int_{-\infty}^0\widehat{b}(k,y_3,t)e^{|k|y_3}dy_3, \quad k=(k_{1},k_{2})\in \mathbb{Z}^{2}.
	\end{equation}
	In order to provide an explicit expression for the previous integral term, we first notice that
	\begin{align}
\widehat b(k,x_3)
&=
\sum_{m\in\Z^2}
\widehat{\Delta_x\eta^{(0)}}(k-m)\,\widehat{\partial_3\Phi^{(0)}}(m,x_3)
+
2\sum_{m\in\Z^2}\widehat{\nabla\eta^{(0)}}(k-m)\cdot \widehat{\nabla\partial_3\Phi^{(0)}}(m,x_3)
\nonumber \\
&=
\sum_{m\in\Z^2}
\Big( -|k-m|^2 \Big)\Big(|m|\widehat\psi^{(0)}(m)e^{|m|x_3}\Big)\widehat\eta^{(0)}(k-m)
\nonumber \\
&\quad
+
2\sum_{m\in\Z^2}
\Big(i(k-m)\widehat\eta^{(0)}(k-m)\Big)\cdot \Big(i m\,|m|\widehat\psi^{(0)}(m)e^{|m|x_3}\Big) \nonumber 
\\
&=
-\sum_{m\in\Z^2}
|m|\big(|k|^2-|m|^2\big)\,
\widehat\eta^{(0)}(k-m)\,\widehat\psi^{(0)}(m)\,e^{|m|x_3}. \label{b:exp}
\end{align}

Thus, using \eqref{formula:poisson} and \eqref{b:exp}, we obtain

\begin{equation}\label{final:formula:comm}
\partial_{x_3}\Phi^{(1)}=\Lambda \psi^{(1)}-\comm{\Lambda}{\eta^{(0)}}\,\Lambda\psi^{(0)}
, \quad \mbox{ on } \Gamma.
\end{equation}
Combining \eqref{final:formula:comm} and  \eqref{ALEsys3}, the evolution equation for $\eta^{(1)}$ becomes
\begin{equation}\label{eq:eta:t}
\partial_t \eta^{(1)}
	= \Lambda\psi^{(1)}-\comm{\Lambda}{\eta^{(0)}}\,\Lambda\psi^{(0)}-\big(\eta^{(0)}_{x_1}\Phi^{(0)}_{x_1}+\eta^{(0)}_{x_2}\Phi^{(0)}_{x_2}\big), \quad \text{on }\Gamma,
\end{equation}
Differentiating \eqref{eq:eta:t} in time and using \eqref{ALEsys4} we deduce
\begin{align}
\partial_{tt} \eta^{(1)}
&= \Lambda\Big(
-\Upsilon\,\partial_{tt}\eta^{(1)}
-\frac{\beta}{4}\,\Delta_{x}^{2}\eta^{(1)}
+ \delta\,\partial_{t}\Delta_{x}\eta^{(1)}
-\eta^{(1)}
+\frac{1}{2}\big((\Lambda\psi^{(0)})^2-|\nabla_x\psi^{(0)}|^2\big)
\Big)
\label{eq:eta:tt}\\[0.3em]
&\quad
-\partial_{t}\,\comm{\Lambda}{\eta^{(0)}}\,\Lambda\psi^{(0)}
-\partial_{t}\Big(\eta^{(0)}_{x_1}\Phi^{(0)}_{x_1}+\eta^{(0)}_{x_2}\Phi^{(0)}_{x_2}\Big). \notag
\end{align}
Rearranging the terms and collecting all contributions involving $\eta^{(1)}$ on the left-hand side, we obtain
\begin{align}
\big(I+\Upsilon\Lambda\big)\partial_{tt}\eta^{(1)}
&+\frac{\beta}{4}\,\Lambda\Delta_x^2\eta^{(1)}
-\delta\,\Lambda\,\partial_t\Delta_x\eta^{(1)}
+\Lambda\eta^{(1)}
\label{eq:eta1_forced_rearranged}\\[0.3em]
&=
\frac12\,\Lambda\Big((\Lambda\psi^{(0)})^2-|\nabla_x\psi^{(0)}|^2\Big)
-\partial_t\!\left(\comm{\Lambda}{\eta^{(0)}}\,\Lambda\psi^{(0)}\right)
-\partial_t\Big(\eta^{(0)}_{x_1}\Phi^{(0)}_{x_1}+\eta^{(0)}_{x_2}\Phi^{(0)}_{x_2}\Big). \notag
\end{align}

Next, we eliminate $\psi^{(0)}$ in favour of $\eta^{(0)}$ using the relation \eqref{eq:n0-kin-Lambda}
On nonzero Fourier modes, this yields
\begin{equation}
\psi^{(0)}=\Lambda^{-1}\partial_t\eta^{(0)},
\qquad
\Lambda\psi^{(0)}=\partial_t\eta^{(0)},
\qquad
\nabla_x\psi^{(0)}=\nabla_x\Lambda^{-1}\partial_t\eta^{(0)}.
\label{eq:psi0-elimination}
\end{equation}
Therefore, \eqref{eq:eta1_forced_rearranged} can be rewritten as
\begin{align}
\big(I+\Upsilon\Lambda\big)\partial_{tt}\eta^{(1)}
&+\frac{\beta}{4}\,\Lambda\Delta_x^2\eta^{(1)}
-\delta\,\Lambda\,\partial_t\Delta_x\eta^{(1)}
+\Lambda\eta^{(1)}
\label{eq:eta1_forced_eta0}\\[0.3em]
&=
\frac12\,\Lambda\Big((\partial_t\eta^{(0)})^2-\big|\nabla_x\Lambda^{-1}\partial_t\eta^{(0)}\big|^2\Big)
-\partial_t\!\left(\comm{\Lambda}{\eta^{(0)}}\,\partial_t\eta^{(0)}\right)
\notag\\[0.3em]
&\quad
-\partial_t\!\Big(\nabla_x\eta^{(0)}\cdot \nabla_x\Lambda^{-1}\partial_t\eta^{(0)}\Big).
\notag
\end{align}

Combining the leading-order equation for $\eta^{(0)}$ with $\varepsilon$ times
\eqref{eq:eta1_forced_eta0}, and recalling that
\[
f=\eta^{(0)}+\varepsilon\eta^{(1)},
\]
we obtain, up to terms of order $\mathcal{O}(\varepsilon^2)$,
\begin{align}
\big(I+\Upsilon\Lambda\big)\partial_{tt}f
&+\frac{\beta}{4}\,\Lambda\Delta_x^2 f
-\delta\,\Lambda\,\partial_t\Delta_x f
+\Lambda f
\label{eq:f_equation_compact_rhs_eta0}\\[0.3em]
&=
\varepsilon\Bigg[
\frac12\,\Lambda\Big((\partial_t\eta^{(0)})^2-\big|\nabla_x\Lambda^{-1}\partial_t\eta^{(0)}\big|^2\Big)
\notag\\[0.3em]
&\qquad\qquad
-\partial_t\!\left(\comm{\Lambda}{\eta^{(0)}}\,\partial_t\eta^{(0)}\right)
-\partial_t\!\Big(\nabla_x\eta^{(0)}\cdot \nabla_x\Lambda^{-1}\partial_t\eta^{(0)}\Big)
\Bigg].
\nonumber
\end{align}

Since $\varepsilon f=\varepsilon \eta^{(0)}+\mathcal{O}(\varepsilon^{2})$, we may replace $\eta^{(0)}$ by $f$ in the
right-hand side at the cost of an $\mathcal{O}(\varepsilon^2)$ error. Hence, neglecting
$\mathcal{O}(\varepsilon^2)$ terms and recalling that $\Delta_{x}^{2}=\Lambda^{4}$ and $\mathcal R=\nabla_x\Lambda^{-1}$
the equation closes as
\begin{align}
\big(I+\Upsilon\Lambda\big) f_{tt}
&+\delta\,\Lambda^{3} f_t
+\Big(\Lambda+\frac{\beta}{4}\Lambda^{5}\Big)f
\nonumber \\
&=
\varepsilon\Bigg[
\frac12\,\Lambda\Big((\partial_t f)^2-\big|\mathcal R\partial_t f\big|^2\Big)
-\partial_t\!\left(\comm{\Lambda}{f}\,\partial_t f\right)
-\partial_t\!\Big(\nabla_x f\cdot \mathcal R\partial_t f\Big)
\Bigg].\label{eq:f_equation_compact_closed}
\end{align} \medskip

In the sequel we derive from \eqref{eq:f_equation_compact_closed} a different equation of the same order of precision, i.e., neglecting terms of order $\mathcal{O}(\varepsilon^{2})$. To that purpose, we start from \eqref{eq:eta1_forced_eta0} and expand the time derivatives in the
forcing. Using the product rule,
\[
\partial_t\!\big(\comm{\Lambda}{\eta^{(0)}}\,\partial_t\eta^{(0)}\big)
=
\comm{\Lambda}{\partial_t\eta^{(0)}}\,\partial_t\eta^{(0)}
+\comm{\Lambda}{\eta^{(0)}}\,\partial_{tt}\eta^{(0)},
\]
and
\[
\partial_t\!\Big(\nabla_x\eta^{(0)}\cdot \mathcal{R} \partial_t\eta^{(0)}\Big)
=
\nabla_x\partial_t\eta^{(0)}\cdot \mathcal{R}\partial_t\eta^{(0)}
+\nabla_x\eta^{(0)}\cdot \mathcal{R}\partial_{tt}\eta^{(0)}.
\]
Hence \eqref{eq:eta1_forced_eta0} may be rewritten as
\begin{align}
	\big(I+\Upsilon\Lambda\big)\partial_{tt}\eta^{(1)}
	&+\frac{\beta}{4}\,\Lambda\Delta_x^2\eta^{(1)}
	-\delta\,\Lambda\,\partial_t\Delta_x\eta^{(1)}
	+\Lambda\eta^{(1)}
	\label{eq:eta1_forced_eta0_expanded_Tversion}\\[0.3em]
	&=
	\frac12\,\Lambda\Big((\partial_t\eta^{(0)})^2-\big|\mathcal R\,\partial_t\eta^{(0)}\big|^2\Big)
	-\comm{\Lambda}{\partial_t\eta^{(0)}}\,\partial_t\eta^{(0)}
	\notag\\[0.3em]
	&\quad
	-\comm{\Lambda}{\eta^{(0)}}\,\partial_{tt}\eta^{(0)}
	-(\nabla_x\partial_t\eta^{(0)})\cdot \mathcal R\,\partial_t\eta^{(0)}
	-(\nabla_x\eta^{(0)})\cdot \mathcal R\,\partial_{tt}\eta^{(0)}.
	\notag
\end{align}

\medskip

We now eliminate $\partial_{tt}\eta^{(0)}$ by means of the leading-order closure
\eqref{eq:n0-eta-closed2}. Solving \eqref{eq:n0-eta-closed2} for $\partial_{tt}\eta^{(0)}$
gives
\begin{equation}
	\label{eq:eta0tt_elim_T}
	\partial_{tt}\eta^{(0)}
	=
	-(I+\Upsilon\Lambda)^{-1}\!\left(
	\frac{\beta}{4}\,\Lambda\Delta_x^2\eta^{(0)}
	-\delta\,\Lambda\,\partial_t\Delta_x\eta^{(0)}
	+\Lambda\eta^{(0)}
	\right).
\end{equation}
Moreover, introducing the Fourier multiplier
\begin{equation}
	\label{eq:def_T_multiplier_new}
	\mathcal T:=(I+\Upsilon\Lambda)^{-1}\Lambda.
\end{equation}
equation \eqref{eq:eta0tt_elim_T} becomes
\begin{equation}\label{eq:eta0tt_elim_T_compact}
	\partial_{tt}\eta^{(0)}
	=
	-\frac{\beta}{4}\,\mathcal T\Delta_x^2\eta^{(0)}
	+\delta\,\mathcal T\partial_t\Delta_x\eta^{(0)}
	-\mathcal T\eta^{(0)}.
\end{equation}

Substituting \eqref{eq:eta0tt_elim_T_compact} into \eqref{eq:eta1_forced_eta0_expanded_Tversion}
yields a forcing depending only on $\eta^{(0)}$ and $\partial_t\eta^{(0)}$:
\begin{align}
	\big(I+\Upsilon\Lambda\big)\partial_{tt}\eta^{(1)}
	&+\frac{\beta}{4}\,\Lambda\Delta_x^2\eta^{(1)}
	-\delta\,\Lambda\,\partial_t\Delta_x\eta^{(1)}
	+\Lambda\eta^{(1)}
	\label{eq:eta1_forced_eta0_no_eta0tt_T}\\[0.3em]
	&=
	\frac12\,\Lambda\Big((\partial_t\eta^{(0)})^2-\big|\mathcal R\,\partial_t\eta^{(0)}\big|^2\Big)
	-\comm{\Lambda}{\partial_t\eta^{(0)}}\,\partial_t\eta^{(0)}
	\notag\\[0.3em]
	&\quad
	-(\nabla_x\partial_t\eta^{(0)})\cdot \mathcal R\,\partial_t\eta^{(0)}
	+\comm{\Lambda}{\eta^{(0)}}\,\mathcal T\eta^{(0)}
	+(\nabla_x\eta^{(0)})\cdot \mathcal R\big(\mathcal T\eta^{(0)}\big)
	\notag\\[0.3em]
	&\quad
	+\frac{\beta}{4}\Big(
	\comm{\Lambda}{\eta^{(0)}}\,\mathcal T\Delta_x^2\eta^{(0)}
	+(\nabla_x\eta^{(0)})\cdot \mathcal R\big(\mathcal T\Delta_x^2\eta^{(0)}\big)
	\Big)
	\notag\\[0.3em]
	&\quad
	-\delta\Big(
	\comm{\Lambda}{\eta^{(0)}}\,\mathcal T\partial_t\Delta_x\eta^{(0)}
	+(\nabla_x\eta^{(0)})\cdot \mathcal R\big(\mathcal T\partial_t\Delta_x\eta^{(0)}\big)
	\Big).
	\notag
\end{align}

\medskip

Finally, defining the renormalized variable
\[
f=\eta^{(0)}+\varepsilon\,\eta^{(1)},
\]
adding the $n=0$ equation \eqref{eq:n0-eta-closed2} to $\varepsilon$ times
\eqref{eq:eta1_forced_eta0_no_eta0tt_T}, and neglecting $O(\varepsilon^2)$ terms,
we obtain a closed $\mathcal{O}(\varepsilon)$ model for $f$ of the form
\begin{equation}\label{eq:f_model_compact_final_again}
	\big(I+\Upsilon\Lambda\big) f_{tt}
	+\delta\,\Lambda^{3} f_t
	+\Big(\Lambda+\frac{\beta}{4}\Lambda^{5}\Big)f
	=
	\varepsilon\,\mathfrak N[f],
\end{equation}
where the nonlinear forcing is
\begin{align}
	\mathfrak N[f]
	&=
	\frac12\,\Lambda\Big((f_t)^2-\big|\mathcal R f_t\big|^2\Big)
	-\comm{\Lambda}{f_t}\,f_t
	-(\nabla_x f_t)\cdot \mathcal R f_t
	\label{eq:def_N_compact_final_T}\\[0.3em]
	&\quad
	+\comm{\Lambda}{f}\,\mathcal T f
	+(\nabla_x f)\cdot \mathcal R\big(\mathcal T f\big)
	\notag\\[0.3em]
	&\quad
	+\frac{\beta}{4}\Big(
	\comm{\Lambda}{f}\,\mathcal T\Delta_x^2 f
	+(\nabla_x f)\cdot \mathcal R\big(\mathcal T\Delta_x^2 f\big)
	\Big)
	\notag\\[0.3em]
	&\quad
	-\delta\Big(
	\comm{\Lambda}{f}\,\mathcal T\partial_t\Delta_x f
	+(\nabla_x f)\cdot \mathcal R\big(\mathcal T\partial_t\Delta_x f\big)
	\Big).
	\notag
\end{align}

\subsection{The unidirectional models}\label{subsec:uni:models}
In this section we derive unidirectional asymptotic models associated with the
bidirectional systems obtained above. To that purpose, as in \cite{AlonsoOran2021MHD,AlonsoOranDuranGraneroBelinchon2024, ChengGraneroBelinchonShkollerWilkening2019, Granero-Scrobogna, GraneroBelinchonShkoller2017},  we restrict to one horizontal dimension $x\in\TT$ and introduce the
unidirectional variables
\[
\xi=x-t,\qquad \tau=\varepsilon t,
\qquad f(x,t)=F(\xi,\tau).
\]
\subsubsection*{The unidirectional model for system \eqref{eq:f_equation_compact_closed}}
By the chain rule,
\begin{equation}
\label{eq:chain_rule_uni}
\partial_t=-\partial_\xi+\varepsilon\partial_\tau,
\qquad
\partial_x=\partial_\xi,
\end{equation}
and consequently
\begin{equation}
\label{eq:derivatives_uni}
f_t=-F_\xi+\varepsilon F_\tau,
\qquad
f_{tt}=F_{\xi\xi}-2\varepsilon F_{\xi\tau}+O(\varepsilon^2),
\qquad
f_{xx}=F_{\xi\xi},
\qquad
f_{xxxx}=F_{\xi\xi\xi\xi}.
\end{equation}
Substituting \eqref{eq:derivatives_uni} into \eqref{eq:f_equation_compact_closed}
and discarding $\mathcal{O}(\varepsilon^2)$ terms yields the $(\xi,\tau)$-equation
\begin{align}
&(I+\Upsilon\Lambda)\big(F_{\xi\xi}-2\varepsilon F_{\xi\tau}\big)
+\frac{\beta}{4}\,\Lambda F_{\xi\xi\xi\xi}
+\delta\,\Lambda F_{\xi\xi\xi}
-\varepsilon\,\delta\,\Lambda F_{\tau\xi\xi}
+\Lambda F
\label{eq:unidirectional_Oeps}\\[0.3em]
&\qquad=
\varepsilon\Bigg[
\frac12\,\Lambda\Big(F_\xi^{\,2}-\big(\Lambda F \big)^2\Big)
-\partial_\xi\Big(\comm{\Lambda}{F}\,F_\xi\Big)
+\partial_\xi\Big(F_\xi\,\Lambda F \Big)
\Bigg],
\nonumber
\end{align}
where all Fourier multipliers are understood with respect to $\xi$, and we used
$\partial_t\partial_{xx}f=-F_{\xi\xi\xi}+\varepsilon F_{\tau\xi\xi}$. Integrating \eqref{eq:unidirectional_Oeps} once in $\xi$ and using that
$\Lambda$ commutes with $\partial_\xi$, we obtain (up to an additive function
of $\tau$ which vanishes under the mean-zero convention)
\begin{align}
&(I+\Upsilon\Lambda)\big(F_{\xi}-2\varepsilon F_{\tau}\big)
+\frac{\beta}{4}\,\Lambda F_{\xi\xi\xi}
+\delta\,\Lambda F_{\xi\xi}
-\varepsilon\,\delta\,\Lambda F_{\tau\xi}
+\mathcal H F
\label{eq:F_integrated}\\[0.3em]
&\qquad=
\varepsilon\Bigg[
\frac12\,\mathcal H\Big(F_\xi^{\,2}-\big(\Lambda F \ \big)^2\Big)
-\comm{\Lambda}{F}\,F_\xi
+F_\xi\,\Lambda F
\Bigg].
\nonumber
\end{align}
Moreover, by means of the Tricomi identity for the Hilbert transform \eqref{eq:Tricomi_general}, we can write
\begin{equation}
\frac12\,\mathcal H\Big(F_\xi^{\,2}-\big(\Lambda F\big)^2\Big)
=
\frac12\,\mathcal H\Big(F_\xi^{\,2}-\big(\mathcal H F_\xi\big)^2\Big)
=
F_\xi\,\mathcal H F_\xi
=
F_\xi\,\Lambda F.
\end{equation}
Thus, \eqref{eq:F_integrated} simplifies to
\begin{equation}\label{eq:F_integrated2}
\begin{aligned}
(I+\Upsilon\Lambda)\big(F_{\xi}-2\varepsilon F_{\tau}\big)
+\frac{\beta}{4}\,\Lambda F_{\xi\xi\xi}
+\delta\,\Lambda F_{\xi\xi}
-\varepsilon\,\delta\,\Lambda F_{\tau\xi}
+\mathcal H F
&=\varepsilon\Big(-\comm{\Lambda}{F}\,F_\xi+2F_\xi\,\Lambda F\Big).
\end{aligned}
\end{equation}

Rewriting \eqref{eq:F_integrated}, moving the $\tau$-terms to the left-hand side,
we obtain the evolution equation
\begin{multline}\label{eq:Ft_isolated_integrated_full}
\varepsilon\Big(2(I+\Upsilon\Lambda)+\delta\,\Lambda\partial_\xi\Big)F_\tau
=(I+\Upsilon\Lambda)F_\xi+\frac{\beta}{4}\,\Lambda F_{\xi\xi\xi}
+\delta\,\Lambda F_{\xi\xi}+\mathcal H F \\
-\varepsilon\Big(2F_\xi\,\Lambda F-\comm{\Lambda}{F}F_\xi \Big).
\end{multline}

Starting from \eqref{eq:Ft_isolated_integrated_full}, we write
\begin{equation}
\label{eq:epsMFt_full}
\varepsilon\,\mathcal M F_\tau
=
(I+\Upsilon\Lambda)F_\xi
+\frac{\beta}{4}\,\Lambda F_{\xi\xi\xi}
+\delta\,\Lambda F_{\xi\xi}
+\mathcal H F
-\varepsilon\Bigg[2F_\xi\,\Lambda F 
-\comm{\Lambda}{F}\,F_\xi
\Bigg],
\end{equation}
where
\[
\mathcal M:=2(I+\Upsilon\Lambda)+\delta\,\Lambda\partial_\xi,
\qquad
\mathcal M^\ast:=2(I+\Upsilon\Lambda)-\delta\,\Lambda\partial_\xi.
\]
Multiplying \eqref{eq:epsMFt_full} on the left by $\mathcal M^\ast$ and using
\[
\mathcal M^\ast\mathcal M
=
4(I+\Upsilon\Lambda)^2-\delta^2\Lambda^2\partial_\xi^2,
\]
we obtain
\begin{equation}\label{eq:Mstar_step_expanded}
\varepsilon(\mathcal M^\ast\mathcal M)F_\tau
=\mathcal M^\ast\!\left((I+\Upsilon\Lambda)F_\xi+\frac{\beta}{4}\Lambda F_{\xi\xi\xi}
+\delta\Lambda F_{\xi\xi}+\mathcal H F\right)
-\varepsilon\,\mathcal M^\ast\!\left(2 F_\xi\,\Lambda F-\comm{\Lambda}{F}F_\xi\right).
\end{equation}
Applying $(\mathcal M^\ast\mathcal M)^{-1}$ (understood mode-by-mode on $k\neq0$)
yields the first inversion step
\begin{align}
F_\tau
&=
\frac{1}{\varepsilon}\,(\mathcal M^\ast\mathcal M)^{-1}\mathcal M^\ast
\Bigg[
(I+\Upsilon\Lambda)F_\xi
+\frac{\beta}{4}\,\Lambda F_{\xi\xi\xi}
+\delta\,\Lambda F_{\xi\xi}
+\mathcal H F
\Bigg]
\label{eq:Ft_first_step_noG}\\[0.3em]
&\quad\quad\quad
-(\mathcal M^\ast\mathcal M)^{-1}\mathcal M^\ast
\Bigg[2F_\xi\,\Lambda F 
-\comm{\Lambda}{F}\,F_\xi
\Bigg].
\nonumber
\end{align}

On the Fourier side ($k\neq0$), one checks that
\[
(\mathcal M^\ast\mathcal M)^{-1}\mathcal M^\ast
=
a(\Lambda,\partial_\xi)+b(\Lambda,\partial_\xi)\,\mathcal H,
\]
where $a(\Lambda,\partial_\xi)$ and $b(\Lambda,\partial_\xi)$ are real Fourier multipliers with symbols
\[
a_k:=\frac{2(1+\Upsilon|k|)}{4(1+\Upsilon|k|)^2+\delta^2|k|^2k^2},
\qquad
b_k:=\frac{\delta|k|^2}{4(1+\Upsilon|k|)^2+\delta^2|k|^2k^2},
\qquad k\neq0,
\]
or equivalently,
\begin{equation}
\label{eq:ab_operators_repeat}
a(\Lambda,\partial_\xi)
=
\frac{2(I+\Upsilon\Lambda)}{4(I+\Upsilon\Lambda)^2-\delta^2\Lambda^2\partial_\xi^2},
\qquad
b(\Lambda,\partial_\xi)
=
\frac{\delta\,\Lambda^2}{4(I+\Upsilon\Lambda)^2-\delta^2\Lambda^2\partial_\xi^2}.
\end{equation}
Substituting this identity into \eqref{eq:Ft_first_step_noG} yields
\begin{equation}\label{eq:Ft_final_ab_noG}
F_\tau=\frac{1}{\varepsilon}\big(a+b\,\mathcal H\big)\!\left((I+\Upsilon\Lambda)F_\xi+\frac{\beta}{4}\Lambda F_{\xi\xi\xi}+\delta\Lambda F_{\xi\xi}+\mathcal H F\right)
-\big(a+b\,\mathcal H\big)\!\left(2F_\xi\,\Lambda F-\comm{\Lambda}{F}F_\xi\right),
\end{equation}
with $a$ and $b$ given by \eqref{eq:ab_operators_repeat}.

\subsubsection{The unidirectional model for \eqref{eq:f_model_compact_final_again}}
We now derive the unidirectional reduction of \eqref{eq:f_model_compact_final_again}
in the same manner as for the first model. More precisely, we restrict to one horizontal space
dimension and introduce the slow variables
\[
\xi=x-t,\qquad \tau=\varepsilon t,\qquad f(x,t)=F(\xi,\tau).
\]
\medskip
With these conventions, up to $\mathcal{O}(\varepsilon^2)$ we have
\begin{equation*}
\begin{aligned}
&f_t=-F_\xi+\varepsilon F_\tau,\qquad
f_{tt}=F_{\xi\xi}-2\varepsilon F_{\xi\tau}+\mathcal O(\varepsilon^2),\qquad
f_{xx}=F_{\xi\xi},\qquad f_{xxxx}=F_{\xi\xi\xi\xi},\\
&\partial_t f_{xx}=(-\partial_\xi+\varepsilon\partial_\tau)F_{\xi\xi}
=-F_{\xi\xi\xi}+\varepsilon F_{\tau\xi\xi}.
\end{aligned}
\end{equation*}
Substituting into equation \eqref{eq:f_model_compact_final_again}, doing some straightforward computations and neglecting $\mathcal{O}(\varepsilon^2)$ we find the equation
\begin{align}
&(I+\Upsilon\Lambda)\big(F_{\xi\xi}-2\varepsilon F_{\xi\tau}\big)
+\frac{\beta}{4}\Lambda F_{\xi\xi\xi\xi}
+\delta\Lambda F_{\xi\xi\xi}
-\varepsilon\delta\Lambda F_{\tau\xi\xi}
+\Lambda F
\label{eq:uni_expanded_Oeps}\\[0.2em]
&\qquad=
\varepsilon\Bigg[
\frac12\,\Lambda\Big(F_\xi^{\,2}-(\Lambda F)^2\Big)
-\comm{\Lambda}{F_\xi}F_\xi
+F_{\xi\xi}\Lambda F
+\comm{\Lambda}{F}\,\mathcal T F
-F_\xi\,\mathcal H\mathcal T F \notag\\
&\qquad\qquad\quad
+\frac{\beta}{4}\Big(\comm{\Lambda}{F}\,\mathcal T F_{\xi\xi\xi\xi}
-F_\xi\,\Lambda\mathcal T F_{\xi\xi\xi}\Big)
+\delta\Big(\comm{\Lambda}{F}\,\mathcal T F_{\xi\xi\xi}
-F_\xi\,\Lambda\mathcal T F_{\xi\xi}\Big)
\Bigg].
\nonumber
\end{align}
We first move the $\tau$--dependent
linear terms to the left-hand side and factor out the slow time derivative. Namely,
\eqref{eq:uni_expanded_Oeps} can be rewritten as
\begin{align}
\varepsilon\Big(2(I+\Upsilon\Lambda)\partial_\xi+\delta\,\Lambda\partial_\xi^2\Big)F_\tau
&=
(I+\Upsilon\Lambda)F_{\xi\xi}
+\frac{\beta}{4}\Lambda F_{\xi\xi\xi\xi}
+\delta\Lambda F_{\xi\xi\xi}
+\Lambda F
\label{eq:Ft_isolated_no_integration}\\[0.2em]
&\quad
-\varepsilon\Bigg[
\frac12\,\Lambda\Big(F_\xi^{\,2}-(\Lambda F)^2\Big)
-\comm{\Lambda}{F_\xi}F_\xi
+F_{\xi\xi}\Lambda F
+\comm{\Lambda}{F}\,\mathcal T F
-F_\xi\,\mathcal H\mathcal T F \notag\\
&\qquad\qquad\quad
+\frac{\beta}{4}\Big(\comm{\Lambda}{F}\,\mathcal T F_{\xi\xi\xi\xi}
-F_\xi\,\Lambda \mathcal T F_{\xi\xi\xi}\Big)
+\delta\Big(\comm{\Lambda}{F}\,\mathcal T F_{\xi\xi\xi}
-F_\xi\,\Lambda \mathcal T F_{\xi\xi}\Big)
\Bigg].
\nonumber
\end{align}

For convenience, denote the linear operator acting on $F_\tau$ by
\[
\mathcal M_1:=2(I+\Upsilon\Lambda)\partial_\xi+\delta\,\Lambda\partial_\xi^2,
\qquad
\mathcal M_1^\ast:=-2(I+\Upsilon\Lambda)\partial_\xi+\delta\,\Lambda\partial_\xi^2,
\]
so that, 
\[
\mathcal M_1^\ast\mathcal M_1=
\big(\delta\,\Lambda\partial_\xi^2\big)^2-4\big((I+\Upsilon\Lambda)\partial_\xi\big)^2.
\]
Multiplying \eqref{eq:Ft_isolated_no_integration} by $\mathcal M_1^\ast$ and applying
$(\mathcal M_1^\ast\mathcal M_1)^{-1}$ (mode-by-mode on $k\neq0$) yields the
first inversion step\begin{align}
F_\tau
&=
\frac{1}{\varepsilon}\,(\mathcal M_1^\ast\mathcal M_1)^{-1}\mathcal M_1^\ast
\Big[
(I+\Upsilon\Lambda)F_{\xi\xi}
+\frac{\beta}{4}\Lambda F_{\xi\xi\xi\xi}
+\delta\Lambda F_{\xi\xi\xi}
+\Lambda F
\Big]
\label{eq:Ft_first_step_no_integration}\\[0.2em]
&\quad
-(\mathcal M_1^\ast\mathcal M_1)^{-1}\mathcal M_1^\ast
\Bigg[
\frac12\,\Lambda\Big(F_\xi^{\,2}-(\Lambda F)^2\Big)
-\comm{\Lambda}{F_\xi}F_\xi
+F_{\xi\xi}\Lambda F
+\comm{\Lambda}{F}\,\mathcal T F
-F_\xi\,\mathcal H\mathcal T F \notag\\
&\qquad\qquad\quad
+\frac{\beta}{4}\Big(\comm{\Lambda}{F}\,\mathcal T F_{\xi\xi\xi\xi}
-F_\xi\,\Lambda \mathcal T F_{\xi\xi\xi}\Big)
+\delta\Big(\comm{\Lambda}{F}\,\mathcal T F_{\xi\xi\xi}
-F_\xi\,\Lambda \mathcal T F_{\xi\xi}\Big)
\Bigg].
\nonumber
\end{align}
On the Fourier side ($k\neq0$), the operator
\[
\mathcal M_1:=2(I+\Upsilon\Lambda)\partial_\xi+\delta\,\Lambda\partial_\xi^2
\]
has symbol
\[
m_1(k)
=
-\,\delta|k|k^2+i\,2(1+\Upsilon|k|)k.
\]
Therefore, $(\mathcal M_1^\ast\mathcal M_1)^{-1}\mathcal M_1^\ast$ has associated symbol
\[
\frac{\overline{m_1(k)}}{|m_1(k)|^2}
=
\frac{-\delta|k|k^2-i\,2(1+\Upsilon|k|)k}{\delta^2|k|^2k^4+4(1+\Upsilon|k|)^2k^2}
\]
Similarly as before, we obtain the
decomposition
\begin{equation}
	\label{eq:M1inv_ab_form}
	(\mathcal M_1^\ast\mathcal M_1)^{-1}\mathcal M_1^\ast
	=
	\alpha(\Lambda,\partial_\xi)
	+\gamma(\Lambda,\partial_\xi)\,\mathcal H,
\end{equation}
where $\alpha(\Lambda,\partial_\xi)$ and $\gamma(\Lambda,\partial_\xi)$ are real Fourier multipliers with symbols
\[
\alpha_k
:=
\frac{-\delta|k|}{\delta^2|k|^4+4(1+\Upsilon|k|)^2},
\qquad
\gamma_k
:=
\frac{2(1+\Upsilon|k|)}{|k|\big(\delta^2|k|^4+4(1+\Upsilon|k|)^2\big)},
\qquad k\neq0.
\]
Equivalently, at the operator level,
\[
\alpha(\Lambda,\partial_\xi)
=
-\frac{\delta\,\Lambda}{\delta^2\Lambda^2\partial_\xi^2-4(I+\Upsilon\Lambda)^2},
\qquad
\gamma(\Lambda,\partial_\xi)
=
\frac{2(I+\Upsilon\Lambda)\,\Lambda^{-1}}{\delta^2\Lambda^2\partial_\xi^2-4(I+\Upsilon\Lambda)^2}.
\] 
	
With this notation, we may write 
\begin{align}
F_\tau
&=
\frac{1}{\varepsilon}\big(\alpha+\gamma\,\mathcal H\big)
\Big[
(I+\Upsilon\Lambda)F_{\xi\xi}
+\frac{\beta}{4}\Lambda F_{\xi\xi\xi\xi}
+\delta\Lambda F_{\xi\xi\xi}
+\Lambda F
\Big]
\label{eq:Ft_final_abH_no_integration}\\[0.2em]
&\quad
-\big(\alpha+\gamma\,\mathcal H\big)
\Bigg[
\frac12\,\Lambda\Big(F_\xi^{\,2}-(\Lambda F)^2\Big)
-\comm{\Lambda}{F_\xi}F_\xi
+F_{\xi\xi}\Lambda F
+\comm{\Lambda}{F}\,\mathcal T F
-F_\xi\,\mathcal H\mathcal T F \notag\\
&\qquad\qquad\quad
+\frac{\beta}{4}\Big(\comm{\Lambda}{F}\,\mathcal T F_{\xi\xi\xi\xi}
-F_\xi\,\Lambda \mathcal T F_{\xi\xi\xi}\Big)
+\delta\Big(\comm{\Lambda}{F}\,\mathcal T F_{\xi\xi\xi}
-F_\xi\,\Lambda \mathcal T F_{\xi\xi}\Big)
\Bigg].
\nonumber
\end{align}

\section{Well-posedness result for the bidirectional model}\label{Wp:bi}
In this section we establish the well-posedness theory for \eqref{eq:f_equation_compact_closed}. In particular, we prove the following result:
\begin{theorem}[Local well-posedness for small $H^3$ data]\label{thm:WPbi}
Let $\Upsilon>0$ and $\delta,\beta>0$, and let $\varepsilon\in\R$ be fixed.
Assume that the initial data satisfy
\[
f(x,0)=f_0\in H^3_0(\TT^2),\qquad \partial_{t}f(x,0)=f_1\in H^1_0(\TT^2),
\qquad \|f_0\|_{H^3}\le \varepsilon_\ast,
\]
for some $\varepsilon_\ast=\varepsilon_\ast(\Upsilon,\delta,\beta,|\varepsilon|)>0$.
Then there exists a time
\[
T=T\big(\Upsilon,\delta,\beta,|\varepsilon|,\|f_1\|_{H^1},\|f_0\|_{H^3}\big)>0
\]
and a unique solution $f$ of \eqref{eq:f_equation_compact_closed} on $[0,T]$ such that
\[
f\in L^\infty(0,T;H^3_{0}(\TT^2)),\qquad f_t\in L^\infty(0,T;H^1_{0}(\TT^2)).
\]
\end{theorem}

\begin{proof}[Proof of Theorem \ref{thm:WPbi}]
We begin by rewriting \eqref{eq:f_equation_compact_closed} as a first--order
system in time. Set
\[
v:=\partial_t f,
\qquad
\mathbb N(f,w):=\comm{\Lambda}{f}\,w+\nabla_x f\cdot \mathcal R w
\]
Then \eqref{eq:f_equation_compact_closed} is equivalent to
\begin{equation}\label{eq:f_system_compact}
\begin{cases}
\partial_t f=v,\\[0.4em]
\big(I+\Upsilon\Lambda\big)\partial_t v
+\delta\,\Lambda^{3} v
+\Big(\Lambda+\dfrac{\beta}{4}\Lambda^{5}\Big)f
+\mathbb N\!\big(f,\partial_t v\big)
=
\varepsilon\,\mathcal Q(v),
\end{cases}
\end{equation}
where the quadratic forcing is
\begin{equation}\label{eq:def_Qv}
\mathcal Q(v)
:=
\frac12\,\Lambda\Big(v^{2}-\big|\mathcal R v\big|^{2}\Big)
-\comm{\Lambda}{v}\,v
-\nabla_x v\cdot \mathcal R v.
\end{equation}

Fix $T>0$. We work in the class of functions
\[
\mathbb X_T:=\Big\{\bar f\in L^\infty(0,T;H^3_{0}(\TT^2)):\ \bar f(0)=f_0,\ 
\|\bar f\|_{L^\infty_{t} H^3_{x}}\le 2\|f_0\|_{H^3_{x}}\Big\},
\]
\[
\mathbb Y_T:=\Big\{\bar v\in L^\infty(0,T;H^1_{0}(\TT^2)):\ \bar v(0)=f_1\},
\]
and construct a solution by a two-stage approximation procedure. \medskip
\subsubsection*{\underline{Step 0: Strategy of the proof}}
The argument is somewhat delicate: we introduce two regularizations with parameters $\mu$ and $\lambda$, solve an elliptic problem
to identify the time derivative, construct the evolution by a contraction mapping, and then pass
to the limits using $\mu$-- and $\lambda$--uniform estimates. We summarize these steps next and
provide the detailed implementation afterwards. \medskip

Given $(\bar f,\bar v)\in\mathbb X_T\times\mathbb Y_T$, we solve an auxiliary
$\lambda$--regularized elliptic equation for an unknown
$U^{\lambda,\mu}[\bar f,\bar v]$ (Step~1). Since $\mathbb N(f,\cdot)$ is linear
in its second argument, this amounts to inverting a perturbation of
$(I+\Upsilon\Lambda)$; the smallness of $\|f_0\|_{H^3}$ (and hence of
$\|\bar f\|_{L^\infty_{t} H^3_{x}}$) ensures that the inversion is well-defined and
yields a unique $U^{\lambda,\mu}$ in the chosen Sobolev class. \medskip

With this $U^{\lambda,\mu}$ fixed, we solve the corresponding $\lambda$--regularized evolution for
$v^{\lambda,\mu}$ with initial datum $v_0$ and choose $T$ so that the map
$\bar v\mapsto v^{\lambda,\mu}$ is a contraction in $\mathbb Y_T$ (Step~2). In
particular, the fixed point $v^{\lambda,\mu}$ satisfies $v_t^{\lambda,\mu}=U^{\lambda,\mu}$
by construction. We then pass to the limit $\lambda\to\infty$ (with $\mu$ fixed)
using uniform-in-$\lambda$ bounds, obtaining a function $v^\mu$ solving the
$\mu$--regularized equation (Step~3). \medskip

Second, with $v^\mu$ available, we recover $f^\mu$ by solving $f_t^\mu=v^\mu$ with
initial datum $f_0$, and close the coupling through a fixed point argument in
$\mathbb X_T$ (Step~4), now only at the $\mu$--regularized level. Finally, we let
$\mu\to\infty$ and use the corresponding uniform-in-$\mu$ a priori bounds to pass
to the limit, yielding a solution $f$ of the original system on $[0,T]$ with
$f\in L^\infty(0,T;H^3_{0})$ and $f_t\in L^\infty(0,T;H^1_{0})$ (Step~5). \medskip

\subsubsection*{\underline{Step 1: Solving the elliptic problem for $U$}}

Fix $T>0$ and regularization parameters $0<\lambda,\mu\leq 1$. Let
$(\bar f,\bar v)\in \mathbb X_T\times\mathbb Y_T$. In this step we will
construct an auxiliary unknown
\[
U=U^{\lambda,\mu}[\bar f,\bar v]\in L^\infty(0,T;H^1_{0}(\TT^2)),
\]
as the unique solution to the regularized elliptic problem
\begin{equation}\label{eq:U_elliptic_step1}
(I+\Upsilon\Lambda)U
=
\mathcal F^{\lambda,\mu}[\bar f,\bar v]
-\mathcal J_\lambda\,\mathbb N(\mathcal J_\mu\bar f,\mathcal J_\lambda U),
\qquad\text{in }(0,T)\times\TT^2,
\end{equation}
where the forcing is
\begin{align}\label{eq:forcing_F_lammu}
\mathcal F^{\lambda,\mu}[\bar f,\bar v]
:={}&
-\delta\,\mathcal J_\lambda\Lambda^3\mathcal J_\lambda \bar v
-\mathcal J_\mu\Big(\Lambda+\frac{\beta}{4}\Lambda^5\Big)\mathcal J_\mu \bar f
\notag\\
&\quad
+\varepsilon\,\mathcal J_\lambda\Bigg[
\frac12\,\Lambda\Big((\mathcal J_\lambda \bar v)^2-\big|\mathcal R\,\mathcal J_\lambda \bar v\big|^2\Big)
-\comm{\Lambda}{\mathcal J_\lambda \bar v}\,\mathcal J_\lambda \bar v
-\nabla_x(\mathcal J_\lambda \bar v)\cdot \mathcal R(\mathcal J_\lambda \bar v)
\Bigg],
\end{align}
and the bilinear coupling is
\begin{equation}\label{eq:def_Nbb}
\mathbb N(f,g):=\comm{\Lambda}{f}\,g+\nabla_x f\cdot \mathcal R g.
\end{equation}
Since $\Lambda$ and each Riesz transform $\mathcal R_j$ annihilate constants,
and $\mathcal J_\lambda,\mathcal J_\mu$ preserve the zero Fourier mode, it follows
that if $\bar f,\bar v$ have zero spatial mean, then each term in
\eqref{eq:forcing_F_lammu} has zero mean. Similary, after a straightforward integration by parts the bilinear term $\comm{\Lambda}{\mathcal J_\mu\bar f}\,(\mathcal J_\lambda U)$ and
$\nabla_x(\mathcal J_\mu\bar f)\cdot \mathcal R(\mathcal J_\lambda U)$ have zero mean. Therefore, we seek $U(t)\in H^1_0(\TT^2)$ solving \eqref{eq:U_elliptic_step1}. \medskip

Define the Banach space
\[
\mathbb U_T:=L^\infty\big(0,T;H^1_0(\TT^2)\big),
\qquad
\|U\|_{\mathbb U_T}:=\esssup_{t\in(0,T)}\|U(t)\|_{H^1}.
\]
For fixed $(\bar f,\bar v)$ we define the operator
\[
\Phi_{\bar f,\bar v}:\mathbb U_T\to \mathbb U_T,
\]
by
\begin{equation}\label{eq:Phi_def_step1}
(\Phi_{\bar f,\bar v}[\bar U])(t)
:=
(I+\Upsilon\Lambda)^{-1}\Big(
\mathcal F^{\lambda,\mu}[\bar f,\bar v](t)
-\mathcal J_\lambda\,\mathbb N(\mathcal J_\mu\bar f(t),\mathcal J_\lambda \bar U(t))
\Big),
\qquad t\in(0,T).
\end{equation}
A fixed point $U=\Phi_{\bar f,\bar v}[U]$ is precisely a solution to
\eqref{eq:U_elliptic_step1} in $\mathbb U_T$.  Let $\bar U\in\mathbb U_T$. Using the fact that 
\begin{equation*}
\|(I+\Upsilon\Lambda)^{-1}G\|_{H^1}\le \|G\|_{L^2},
\qquad G\in H^1_0(\TT^2),
\end{equation*}
the bilinear estimate 
\begin{equation*}
\| \mathbb N(\mathcal J_\mu\bar f(t),\mathcal J_\lambda \bar U(t)) \|_{H^1}
\leq C_{\mu}
\|\bar f(t)\|_{H^3}\,\|\bar U(t)\|_{H^1}
\end{equation*}
and the boundedness of $\mathcal J_\lambda$ on $H^1$
and $\mathcal J_\mu$ on $H^3$, we obtain for
a.e.\ $t\in(0,T)$
\[
\|(\Phi_{\bar f,\bar v}[\bar U])(t)\|_{H^1}
\lesssim
\|\mathcal F^{\lambda,\mu}[\bar f,\bar v](t)\|_{H^1}
+C_{\mu}\|\bar f(t)\|_{H^3}\,\|\bar U(t)\|_{H^1}.
\]
Taking the essential supremum over $t$ gives
\begin{equation}\label{eq:U_bound_with_Ubar_full}
\|\Phi_{\bar f,\bar v}[\bar U]\|_{L^\infty(0,T;H^1)}
\le
C(\lambda,\mu,\bar f,\bar v)
+C_{\mu}\,\|\bar f\|_{L^\infty(0,T;H^3)}\,\|\bar U\|_{L^\infty(0,T;H^1)},
\end{equation}
where we set
\begin{equation}\label{eq:C_lammu_def}
C(\lambda,\mu,\bar f,\bar v)
:=
C\,\|\mathcal F^{\lambda,\mu}[\bar f,\bar v]\|_{L^\infty(0,T;H^{1})}.
\end{equation}
This quantity is finite for $(\bar f,\bar v)\in \mathbb X_T\times\mathbb Y_T$
because $\mathcal J_\lambda$ and $\mathcal J_\mu$ are smoothing and hence map
$L^\infty_tH^s_x$ into $L^\infty_tH^{s+m}_x$ for any $m\ge0$. \medskip

Next, let us show the contraction estimate. Let $\bar U^{(1)},\bar U^{(2)}\in\mathbb U_T$. Subtracting \eqref{eq:Phi_def_step1},
\[
\Phi_{\bar f,\bar v}[\bar U^{(1)}]-\Phi_{\bar f,\bar v}[\bar U^{(2)}]
=
-(I+\Upsilon\Lambda)^{-1}\mathcal J_\lambda\,
\mathbb N\!\left(\mathcal J_\mu\bar f,\ \mathcal J_\lambda(\bar U^{(1)}-\bar U^{(2)})\right).
\]
Hence, by a similar argument as before
\[
\|\Phi_{\bar f,\bar v}[\bar U^{(1)}]-\Phi_{\bar f,\bar v}[\bar U^{(2)}]\|_{\mathbb U_T}
\;\lesssim\;
\|\bar f\|_{L^\infty(0,T;H^3)}\,
\|\bar U^{(1)}-\bar U^{(2)}\|_{\mathbb U_T}.
\]
Assume now that 
\begin{equation}\label{eq:smallness_f0_step1}
\|f_0\|_{H^3(\TT^2)}\le \varepsilon_\ast,
\end{equation}
where $\varepsilon_\ast>0$ will be fixed momentarily. Since $\bar f\in\mathbb X_T$
we have $\|\bar f\|_{L^\infty H^3}\le 2\|f_0\|_{H^3}$, and thus
\[
\|\Phi_{\bar f,\bar v}[\bar U^{(1)}]-\Phi_{\bar f,\bar v}[\bar U^{(2)}]\|_{L^\infty H^1}
\le
(2C\|f_0\|_{H^3})\,
\|\bar U^{(1)}-\bar U^{(2)}\|_{L^\infty H^1}.
\]
Choose $\varepsilon_\ast>0$ so small that
\[
2C\,\varepsilon_\ast \le \frac12
\qquad\text{and}\qquad
C\,\varepsilon_\ast \le \frac12.
\]
(Equivalently, $\,\varepsilon_\ast\le (4C)^{-1}$.)
Then $\Phi_{\bar f,\bar v}$ is a strict contraction on $\mathbb U_T$, hence by
Banach's fixed-point theorem there exists a unique
\[
U^{\lambda,\mu}[\bar f,\bar v]\in L^\infty(0,T;H^1_0(\TT^2))
\]
such that $U^{\lambda,\mu}[\bar f,\bar v]=\Phi_{\bar f,\bar v}[U^{\lambda,\mu}[\bar f,\bar v]]$,
i.e.\ $U^{\lambda,\mu}[\bar f,\bar v]$ solves \eqref{eq:U_elliptic_step1}.
Moreover, taking $\bar U=U$ in \eqref{eq:U_bound_with_Ubar_full} and using
$\|\bar f\|_{L^\infty_t H^3_x}\le 2\|f_0\|_{H^3}\le \varepsilon_\ast$, we obtain
\[
\|U\|_{L^\infty(0,T;H^1)}
\le
C(\lambda,\mu,\bar f,\bar v)
+
C\,\varepsilon_\ast\,\|U\|_{L^\infty(0,T;H^1)}.
\]
Hence $C\varepsilon_\ast\le \frac12$ implies
\begin{equation}\label{eq:U_bound_absorbed_full}
\|U\|_{L^\infty(0,T;H^1)}
\le
2\,C(\lambda,\mu,\bar f,\bar v).
\end{equation}
So, we conclude that for each fixed $T>0$ and $0<\lambda,\mu\leq 1$, and for each
$(\bar f,\bar v)\in\mathbb X_T\times\mathbb Y_T$ with mean zero and satisfying
\eqref{eq:smallness_f0_step1}, the elliptic fixed-point problem
\eqref{eq:U_elliptic_step1} admits a unique solution
\[
U^{\lambda,\mu}[\bar f,\bar v]\in L^\infty(0,T;H^1_0(\TT^2)),
\]
and this solution satisfies the bound \eqref{eq:U_bound_absorbed_full}. 

\subsubsection*{\underline{Step 2: Solving the equation for $v$}}
 In Step~1 we constructed, under the smallness assumption
$\|f_0\|_{H^3}\le \varepsilon_\ast$, a unique auxiliary field
\[
U^{\lambda,\mu}[\bar f,\bar v]\in \mathbb U_T=L^\infty(0,T;H^1_0(\TT^2)),
\]
solving the elliptic problem \eqref{eq:U_elliptic_step1}. In this step
we solve for a velocity $v=v^{\lambda,\mu}$ such that $v_t=U^{\lambda,\mu}[\bar f,v]$. \medskip

Given $\bar v\in\mathbb Y_T$, define
\[
\mathcal U_{\bar f}(\bar v):=U^{\lambda,\mu}[\bar f,\bar v]\in \mathbb U_T,
\]
the elliptic solution from Step~1 associated with $(\bar f,\bar v)$. Then we define
$\Psi_{\bar f}:\mathbb Y_T\to\mathbb Y_T$ by
\begin{equation}\label{eq:Psi_def_step2_H1}
(\Psi_{\bar f}[\bar v])(t)
:=
v_0+\int_0^t \mathcal U_{\bar f}(\bar v)(s)\,ds,
\qquad t\in[0,T].
\end{equation}
Let $\bar v\in\mathbb Y_T$. Since $\mathcal U_{\bar f}(\bar v)\in\mathbb U_T$,
the integral in \eqref{eq:Psi_def_step2_H1} is well-defined in $H^1$, and
\[
\|\Psi_{\bar f}[\bar v](t)\|_{H^1}
\le
\|v_0\|_{H^1}+\int_0^t \|\mathcal U_{\bar f}(\bar v)(s)\|_{H^1}\,ds
\le
\|v_0\|_{H^1}+T\,\|\mathcal U_{\bar f}(\bar v)\|_{L^\infty(0,T;H^1)}.
\]
Hence $\Psi_{\bar f}[\bar v]\in \mathbb Y_T$ provided we control
$\|\mathcal U_{\bar f}(\bar v)\|_{L^\infty H^1}$. Such a bound follows from Step~1,
since for each $t$ the fixed point $U(t)$ solves
\[
(I+\Upsilon\Lambda)U
=
\mathcal F^{\lambda,\mu}[\bar f,\bar v]
-\mathcal J_\lambda\,\mathbb N(\mathcal J_\mu\bar f,\mathcal J_\lambda U),
\]
and Step~1 provides an estimate of the form
\begin{equation}\label{eq:U_H1_bound_step2}
\|\mathcal U_{\bar f}(\bar v)\|_{L^\infty(0,T;H^1)}
\le
C(\lambda,\mu,\bar f,\bar v),
\end{equation}
where $C(\lambda,\mu,\bar f,\bar v)$ is finite for
$(\bar f,\bar v)\in\mathbb X_T\times\mathbb Y_T$ (and depends on $\lambda,\mu$
through smoothing). \medskip

Let $\bar v^{(1)},\bar v^{(2)}\in \mathbb Y_T$, and set
\[
U^{(i)}:=\mathcal U_{\bar f}(\bar v^{(i)})=U^{\lambda,\mu}[\bar f,\bar v^{(i)}]
\in \mathbb U_T.
\]
Subtracting the elliptic equations \eqref{eq:U_elliptic_step1} satisfied by $U^{(1)}$
and $U^{(2)}$ yields, for a.e.\ $t\in(0,T)$,
\begin{align}\label{eq:Udiff_step2_H1}
(I+\Upsilon\Lambda)\big(U^{(1)}-U^{(2)}\big)
&=
\Big(\mathcal F^{\lambda,\mu}[\bar f,\bar v^{(1)}]-\mathcal F^{\lambda,\mu}[\bar f,\bar v^{(2)}]\Big)
\nonumber\\
&\quad
-\mathcal J_\lambda\Big(\mathbb N(\mathcal J_\mu\bar f,\mathcal J_\lambda U^{(1)})
-\mathbb N(\mathcal J_\mu\bar f,\mathcal J_\lambda U^{(2)})\Big).
\end{align}
The difference in the coupling term is linear in $U^{(1)}-U^{(2)}$, and by the bilinear
estimate from the preliminaries (applied in $H^1$ using that $\mathcal J_\lambda$ is bounded on $H^1$)
we have
\begin{equation}\label{eq:N_Lip_in_U_step2_H1}
\big\|\mathcal J_\lambda\big(\mathbb N(\mathcal J_\mu\bar f,\mathcal J_\lambda U^{(1)})
-\mathbb N(\mathcal J_\mu\bar f,\mathcal J_\lambda U^{(2)})\big)\big\|_{H^1}
\;\lesssim\;
\|\bar f\|_{H^3}\,\|U^{(1)}-U^{(2)}\|_{H^1}.
\end{equation}

It remains to estimate the forcing difference in $H^1$. Inspecting
\eqref{eq:forcing_F_lammu}, the $\bar v$--dependence is contained only in the terms
\[
-\delta\,\mathcal J_\lambda\Lambda^3\mathcal J_\lambda \bar v
\qquad\text{and}\qquad
\varepsilon\,\mathcal J_\lambda\,\mathcal Q(\mathcal J_\lambda\bar v),
\]
where
\[
\mathcal Q(w):=
\frac12\,\Lambda\Big(w^2-|\mathcal R w|^2\Big)
-\comm{\Lambda}{w}\,w
-\nabla_x w\cdot \mathcal R w.
\]
Since $\mathcal J_\lambda$ is smoothing, the Fourier multiplier
$\mathcal J_\lambda\Lambda^3\mathcal J_\lambda$ is bounded from $H^1$ to $H^1$, with
operator norm depending on $\lambda$; hence
\begin{equation}\label{eq:lin_forcing_Lip_H1}
\big\|\mathcal J_\lambda\Lambda^3\mathcal J_\lambda(\bar v^{(1)}-\bar v^{(2)})\big\|_{H^1}
\le C_\lambda\,\|\bar v^{(1)}-\bar v^{(2)}\|_{H^1}.
\end{equation}
For the quadratic term, we use that $\mathcal J_\lambda:H^1\to H^m$ for every $m$,
and in particular $\|\mathcal J_\lambda h\|_{H^3}\le C_\lambda\|h\|_{H^1}$. Since
$H^3(\TT^2)\hookrightarrow W^{1,\infty}(\TT^2)$, standard product/commutator bounds
imply that $\mathcal Q$ is locally Lipschitz from $H^3$ to $H^1$:
for every $R>0$,
\begin{equation}\label{eq:Q_Lip_step2_H1}
\|\mathcal Q(w^{(1)})-\mathcal Q(w^{(2)})\|_{H^1}
\le
C\big(\|w^{(1)}\|_{H^3}+\|w^{(2)}\|_{H^3}\big)\,\|w^{(1)}-w^{(2)}\|_{H^3}
\end{equation}
and therefore, for $\bar v^{(1)},\bar v^{(2)}$ with $\|\bar v^{(i)}\|_{H^1}\le R$,
\begin{equation}\label{eq:Q_Lip_after_Jlambda_H1}
\|\mathcal Q(\mathcal J_\lambda\bar v^{(1)})-\mathcal Q(\mathcal J_\lambda\bar v^{(2)})\|_{H^1}
\le
C_{\lambda,R}\,\|\bar v^{(1)}-\bar v^{(2)}\|_{H^1},
\end{equation}
for some constant $C_{\lambda,R}$ depending on $\lambda$ and $R$. Since $\mathcal J_\lambda$
is bounded on $H^1$, we then also have
\begin{equation}\label{eq:forcing_quad_Lip_H1}
\big\|\mathcal J_\lambda(\mathcal Q(\mathcal J_\lambda\bar v^{(1)})-\mathcal Q(\mathcal J_\lambda\bar v^{(2)}))\big\|_{H^1}
\le
C_{\lambda,R}\,\|\bar v^{(1)}-\bar v^{(2)}\|_{H^1}.
\end{equation}

Combining \eqref{eq:lin_forcing_Lip_H1} and \eqref{eq:forcing_quad_Lip_H1} gives
\begin{equation}\label{eq:forcing_Lip_total_H1}
\big\|\mathcal F^{\lambda,\mu}[\bar f,\bar v^{(1)}]-\mathcal F^{\lambda,\mu}[\bar f,\bar v^{(2)}]\big\|_{H^1}
\le
C_{\lambda,R}\,\|\bar v^{(1)}-\bar v^{(2)}\|_{H^1}.
\end{equation}

Now apply $(I+\Upsilon\Lambda)^{-1}$ to \eqref{eq:Udiff_step2_H1}. Using that
$(I+\Upsilon\Lambda)^{-1}$ is bounded on $H^1_0(\TT^2)$ and combining
\eqref{eq:N_Lip_in_U_step2_H1} and \eqref{eq:forcing_Lip_total_H1}, we obtain
\begin{align}\label{eq:U_Lip_in_v_step2_H1}
\|U^{(1)}-U^{(2)}\|_{H^1}
&\le
C_{\lambda,R}\,\|\bar v^{(1)}-\bar v^{(2)}\|_{H^1}
+
C\,\|\bar f\|_{H^3}\,\|U^{(1)}-U^{(2)}\|_{H^1}.
\end{align}
Under the same smallness assumption as in Step~1,
$\|\bar f\|_{L^\infty(0,T;H^3)}\le 2\|f_0\|_{H^3}\le \varepsilon_\ast$ with
$\varepsilon_\ast$ chosen so that $C\varepsilon_\ast\le \tfrac12$, we can absorb the
last term in \eqref{eq:U_Lip_in_v_step2_H1} and obtain the Lipschitz bound
\begin{equation}\label{eq:U_Lip_final_step2_H1}
\|\mathcal U_{\bar f}(\bar v^{(1)})-\mathcal U_{\bar f}(\bar v^{(2)})\|_{L^\infty(0,T;H^1)}
\le
C_{\lambda,R}\,
\|\bar v^{(1)}-\bar v^{(2)}\|_{L^\infty(0,T;H^1)},
\end{equation}
for $\bar v^{(1)},\bar v^{(2)}$ in the closed ball
$\{\|v\|_{\mathbb Y_T}\le R\}$. Let $\bar v^{(1)},\bar v^{(2)}\in\mathbb Y_T$ with $\|\bar v^{(i)}\|_{\mathbb Y_T}\le R$,
and set $v^{(i)}=\Psi_{\bar f}[\bar v^{(i)}]$. Using \eqref{eq:Psi_def_step2_H1} and
\eqref{eq:U_Lip_final_step2_H1},
\[
\|v^{(1)}(t)-v^{(2)}(t)\|_{H^1}
\le
\int_0^t \|\mathcal U_{\bar f}(\bar v^{(1)})(s)-\mathcal U_{\bar f}(\bar v^{(2)})(s)\|_{H^1}\,ds
\le
t\,C_{\lambda,R}\,\|\bar v^{(1)}-\bar v^{(2)}\|_{\mathbb Y_T}.
\]
Taking the essential supremum over $t\in(0,T)$ gives
\begin{equation}\label{eq:Psi_contraction_step2_H1}
\|\Psi_{\bar f}[\bar v^{(1)}]-\Psi_{\bar f}[\bar v^{(2)}]\|_{\mathbb Y_T}
\le
T\,C_{\lambda,R}\,\|\bar v^{(1)}-\bar v^{(2)}\|_{\mathbb Y_T}.
\end{equation}
Choosing $T>0$ so that $T\,C_{\lambda,R}\le \tfrac12$, we conclude that
$\Psi_{\bar f}$ is a strict contraction on the closed ball
$\{\|v\|_{\mathbb Y_T}\le R\}$, provided $R$ is chosen so that
$\Psi_{\bar f}$ maps the ball into itself. This follows from
\eqref{eq:U_H1_bound_step2} and
\[
\|v\|_{\mathbb Y_T}
\le
\|v_0\|_{H^1}+T\,\|\mathcal U_{\bar f}(v)\|_{L^\infty(0,T;H^1)}.
\]
Hence Banach's fixed point theorem yields a unique
\[
v^{\lambda,\mu}\in \mathbb Y_T
\quad\text{such that}\quad
v^{\lambda,\mu}=\Psi_{\bar f}[v^{\lambda,\mu}].
\]
Moreover, we have
\begin{equation}\label{eq:v_solves_step2_H1}
\partial_t v^{\lambda,\mu}
=
U^{\lambda,\mu}[\bar f,v^{\lambda,\mu}]
\qquad\text{in }L^\infty(0,T;H^1(\TT^2)),
\end{equation}
thereby identifying the elliptic unknown with the time derivative of $v$. In particular, substituting \eqref{eq:v_solves_step2_H1} into the elliptic identity
\eqref{eq:U_elliptic_step1} (with $\bar v=v^{\lambda,\mu}$) shows that
$v^{\lambda,\mu}$ satisfies the $\lambda$--regularized evolution equation
\begin{align}
(I+\Upsilon\Lambda)\,\partial_t v^{\lambda,\mu}
&+\delta\,\mathcal{J}_\lambda\Lambda^{3}\mathcal{J}_\lambda v^{\lambda,\mu}
+\mathcal{J}_\mu\Big(\Lambda+\frac{\beta}{4}\Lambda^{5}\Big)\mathcal{J}_\mu\bar f
+\mathcal{J}_\lambda\mathbb{N}\big(\mathcal{J}_\mu\bar f,\mathcal{J}_\lambda\partial_t v^{\lambda,\mu}\big)
\notag\\
&=\varepsilon\,\mathcal{J}_\lambda\Bigg[
\frac12\,\Lambda\Big((\mathcal{J}_\lambda v^{\lambda,\mu})^2-\big|\mathcal R\,\mathcal{J}_\lambda v^{\lambda,\mu}\big|^2\Big)
-\comm{\Lambda}{\mathcal{J}_\lambda v^{\lambda,\mu}}\,\mathcal{J}_\lambda v^{\lambda,\mu}
-\nabla_x(\mathcal{J}_\lambda v^{\lambda,\mu})\cdot \mathcal R(\mathcal{J}_\lambda v^{\lambda,\mu})
\Bigg],
\label{eq:v_equation_reg_step2}
\end{align}
in $L^\infty(0,T;H^1(\TT^2))$. 

\subsubsection*{\underline{Step 3: Passing to the limit as $\lambda\to 0$}}

Fix $0<\mu\leq 1$ and let $v^\lambda$ be the solutions
constructed in Step~2 on $[0,T]$, where 
$v^\lambda\in L^\infty(0,T;H^1_0(\TT^2))$.\footnote{Here and below we suppress the dependence
on $\mu$ and on $\bar f$ to lighten notation.} We claim that, for $T>0$ chosen as in Step~2, the family $\{v^\lambda\}$ satisfies
the uniform estimates
\begin{equation}\label{eq:lambda_uniform_bounds}
\sup_{0\le t\le T}\|v^\lambda(t)\|_{H^1}\ \le\ C,
\qquad
\int_0^T\|\mathcal J_\lambda v^\lambda(t)\|_{H^2}^2\,dt\ \le\ C,
\qquad
\int_0^T\|\pa_{t}v^{\lambda}(t)\|_{H^{-1/2}}^2\,dt\ \le\ C,
\end{equation}
with a constant $C$ independent of $\lambda$. Testing the regularized equation \eqref{eq:v_equation_reg_step2} against
$\Lambda v^\lambda$ yields an inequality of the form
\begin{equation}\label{eq:energy1_lambda}
\frac12\frac{d}{dt}\Big(\|v^\lambda\|_{\dot H^{1/2}}^2+\Upsilon\|v^\lambda\|_{\dot H^1}^2\Big)
+\delta\|\mathcal J_\lambda v^\lambda\|_{\dot H^2}^2
\ \le\ \mathcal R_1^\lambda(t),
\end{equation}
where $\mathcal R_1^\lambda$ collects the quadratic terms and the coupling term
$\mathbb N(\mathcal J_\mu\bar f,\mathcal J_\lambda U^\lambda)$. More precisely,
\begin{align*}
\mathcal R_1^\lambda(t)
={}&
-\int_{\TT^2}\Big(\mathcal J_\mu\Big(\Lambda+\frac{\beta}{4}\Lambda^{5}\Big)\mathcal J_\mu\bar f(t)\Big)\,\Lambda v^\lambda(t)\,dx
-\int_{\TT^2}\Big(\mathcal J_\lambda\,\mathbb N\!\big(\mathcal J_\mu\bar f(t),\mathcal J_\lambda v_{t}^\lambda(t)\big)\Big)\,\Lambda v^\lambda(t)\,dx\\
&\hspace{-1cm}
+\varepsilon\int_{\TT^2}\mathcal J_\lambda\Bigg[
\frac12\,\Lambda\Big((\mathcal J_\lambda v^\lambda(t))^2-\big|\mathcal R\,\mathcal J_\lambda v^\lambda(t)\big|^2\Big)
-\comm{\Lambda}{\mathcal J_\lambda v^\lambda(t)}\,\mathcal J_\lambda v^\lambda(t)
-\nabla_x(\mathcal J_\lambda v^\lambda(t))\cdot\mathcal R(\mathcal J_\lambda v^\lambda(t))
\Bigg]\;\Lambda v^\lambda(t)\,dx\\[0.2em]
=:{}&\,M_1(t)+M_2(t)+M_3(t).
\end{align*}
Using $\|\bar f\|_{L^\infty_tH^3_x}\le2\|f_0\|_{H^3}$ and Cauchy--Schwarz together with the boundedness/smoothing of $\mathcal J_\mu$ on Sobolev spaces, we have
\[
|M_1(t)|
\lesssim\|\bar f(t)\|_{H^3}\,\|v^\lambda(t)\|_{H^1}
\le C_{\mu,\beta}\|f_0\|_{H^3}\,\|v^\lambda(t)\|_{H^1}.
\]

To estimate $M_3(t)$ we use that $\mathcal J_\lambda$ is self-adjoint and bounded on
all $H^s$, $H^2(\TT^2)\hookrightarrow W^{1,\infty}(\TT^2)$, and the boundedness of
Riesz transforms on Sobolev spaces. We obtain
\begin{align}
|M_3(t)|
&\lesssim \|\mathcal J_\lambda v^\lambda(t)\|_{H^2}\,\|v^\lambda(t)\|_{L^2}\,\|v^\lambda(t)\|_{H^1}
+\|\mathcal J_\lambda v^\lambda(t)\|_{H^2}\,\|v^\lambda(t)\|_{H^1}^2.
\label{eq:M3_bound_compact}
\end{align}
To estimate $M_2(t)$ we use that $\mathcal J_\lambda$ is self-adjoint on $L^2$ and write
\[
M_2(t)
:=
-\int_{\TT^2}\Big(\mathcal J_\lambda\,\mathbb N(\mathcal J_\mu\bar f(t),\mathcal J_\lambda v_t^\lambda(t))\Big)\,\Lambda v^\lambda(t)\,dx
=
-\int_{\TT^2}\mathbb N(\mathcal J_\mu\bar f(t),\mathcal J_\lambda v_t^\lambda(t))\,\Lambda\mathcal J_\lambda v^\lambda(t)\,dx.
\]

Expanding $\mathbb N(f,g)=[\Lambda,f]g+\nabla f\cdot\mathcal R g$, integrating by parts and Sobolev embedding we find that
\begin{align*}
|M_{2}(t)|&\leq C \norm{\mathcal J_\lambda v^{\lambda}}_{H^2} \norm{\mathcal J_\lambda v_{t}^{\lambda}}_{L^2}\norm{\bar{f}}_{L^{\infty}}-2\int_{\TT^{2}} \bar{f}\Lambda J_{\lambda}v^{\lambda}\Lambda J_{\lambda}v^{\lambda}_{t}\,dx, \\
&\leq C \norm{\mathcal J_\lambda v^{\lambda}}_{H^2}  \norm{ v_{t}^{\lambda}}_{L^2}\norm{\bar{f}}_{L^{\infty}} + C \norm{\bar{f}\Lambda J_{\lambda}v^{\lambda}}_{H^1}\norm{\Lambda J_{\lambda}v^{\lambda}_{t}}_{H^{-1}},  \\
&\leq C \norm{\mathcal J_\lambda v^{\lambda}}_{H^2}  \norm{ v_{t}^{\lambda}}_{L^2}\norm{\bar{f}}_{H^{2}}
\end{align*}
where in the second inequality we make use of the $H^{-1}-H^{1}$ duality. In particular, since $\|\bar f\|_{L^\infty_tH^2_x}\le \|\bar f\|_{L^\infty_tH^3_x}\le 2\|f_0\|_{H^3}$,
\[
|M_2(t)|\ \lesssim_{\mu}\ \|f_0\|_{H^3}\,\|v_t^\lambda(t)\|_{L^2}\,\|\mathcal J_\lambda v^\lambda(t)\|_{H^2}.
\]
Therefore, using Young's inequality we find that
\begin{equation}
|\mathcal R_1^\lambda(t)|
\le \frac{\delta}{4}\|\mathcal J_\lambda v^\lambda(t)\|_{H^2}^2
+ C_{\delta,\mu,\beta}\Big(1+\|f_0\|_{H^3}^2\Big)\Big(\|v^\lambda(t)\|_{H^1}^4+\|v_t^\lambda(t)\|_{L^2}^2+1\Big).
\end{equation}
Thus, 
\begin{equation}\label{eq:energy1_lambda2}
\frac12\frac{d}{dt}\Big(\|v^\lambda\|_{\dot H^{1/2}}^2+\Upsilon\|v^\lambda\|_{\dot H^1}^2\Big)
+\frac{\delta}{2}\|\mathcal J_\lambda v^\lambda\|_{\dot H^2}^2
\ \le C_{\delta,\mu,\beta}\Big(1+\|f_0\|_{H^3}^2\Big)\Big(\|v^\lambda(t)\|_{H^1}^4+\|v_t^\lambda(t)\|_{L^2}^2+1\Big).
\end{equation}

Next, we test \eqref{eq:v_equation_reg_step2} against $\Lambda^{-1}v_t^\lambda$.  Using that $I+\Upsilon\Lambda$ is
self-adjoint and commutes with $\Lambda^{-1}$, and that $\mathcal J_\lambda$ is
self-adjoint and commutes with $\Lambda$, we obtain
\begin{align}\label{eq:energy2_lambda}
\frac12\|v_t^\lambda(t)\|_{\dot H^{-1/2}}^2
+\frac{\Upsilon}{2}\|v_t^\lambda(t)\|_{L^2}^2
+\frac{\delta}{2}\frac{d}{dt}\|\mathcal J_\lambda v^\lambda(t)\|_{\dot H^1}^2
\le \mathcal R_2^\lambda(t),
\end{align}
where $\mathcal R_2^\lambda$ collects the quadratic terms and the coupling term
$\mathbb N(\mathcal J_\mu\bar f,\mathcal J_\lambda v_t^\lambda)$ tested against
$\Lambda^{-1}v_t^\lambda$. Arguing exactly as in the estimate of $\mathcal R_1^\lambda$, we obtain the  bound
\begin{equation}\label{eq:R2_bound}
\mathcal R_2^\lambda(t)
\le
C\,\|v^\lambda(t)\|_{H^1}^4+\frac{\Upsilon}{4}\|v_t^\lambda(t)\|_{L^2}^2,
\end{equation}
with $C$ independent of $\lambda$. Absorbing the last term into the left-hand side
of \eqref{eq:energy2_lambda} yields
\begin{equation}\label{eq:energy2_absorbed}
\|v_t^\lambda(t)\|_{\dot H^{-1/2}}^2
+\frac{\Upsilon}{2}\|v_t^\lambda(t)\|_{L^2}^2+\delta\frac{d}{dt}\|\mathcal J_\lambda v^\lambda(t)\|_{\dot H^1}^2
\le
C\,\|v^\lambda(t)\|_{H^1}^4.
\end{equation}
Set
\[
E_1^\lambda(t):=\|v^\lambda(t)\|_{\dot H^{1/2}}^2+\Upsilon\|v^\lambda(t)\|_{\dot H^{1}}^2,
\qquad
Y(t):=\sup_{0\le s\le t}\|v^\lambda(s)\|_{H^1}^2.
\]
Since $v^\lambda$ has zero mean, Poincar\'e implies $E_1^\lambda(t)\sim \|v^\lambda(t)\|_{H^1}^2$. From \eqref{eq:energy1_lambda2} we have, for all $t\in[0,T]$,
\begin{equation}\label{eq:A_energy1}
\frac{d}{dt}E_1^\lambda(t)+\frac{\delta}{2}\|\mathcal J_\lambda v^\lambda(t)\|_{\dot H^2}^2
\ \le\ 
C_0\Big(\|v^\lambda(t)\|_{H^1}^4+\|v_t^\lambda(t)\|_{L^2}^2+1\Big),
\end{equation}
where $C_0=C_0(\Upsilon,\delta,\mu,\beta)\big(1+\|f_0\|_{H^3}^2\big)$ and is independent of $\lambda$.
Integrating \eqref{eq:A_energy1} on $[0,t]$ gives
\begin{equation}\label{eq:A_energy1_int}
\sup_{0\le s\le t}E_1^\lambda(s)
+\frac{\delta}{2}\int_0^t\|\mathcal J_\lambda v^\lambda(s)\|_{\dot H^2}^2\,ds
\ \le\
E_1^\lambda(0)
+C_0\int_0^t\Big(\|v^\lambda(s)\|_{H^1}^4+\|v_t^\lambda(s)\|_{L^2}^2+1\Big)\,ds.
\end{equation}
From \eqref{eq:energy2_absorbed} we have, for all $t\in[0,T]$,
\begin{equation}\label{eq:B_energy2}
\|v_t^\lambda(t)\|_{\dot H^{-1/2}}^2+\frac{\Upsilon}{2}\|v_t^\lambda(t)\|_{L^2}^2
+\delta\frac{d}{dt}\|\mathcal J_\lambda v^\lambda(t)\|_{\dot H^1}^2
\ \le\ C\,\|v^\lambda(t)\|_{H^1}^4,
\end{equation}
with $C$ independent of $\lambda$. Integrating \eqref{eq:B_energy2} on $[0,t]$ yields
\begin{equation}\label{eq:B_energy2_int}
\int_0^t\|v_t^\lambda(s)\|_{\dot H^{-1/2}}^2\,ds
+\frac{\Upsilon}{2}\int_0^t\|v_t^\lambda(s)\|_{L^2}^2\,ds
+\delta\|\mathcal J_\lambda v^\lambda(t)\|_{\dot H^1}^2
\ \le\
\delta\|\mathcal J_\lambda v_0\|_{\dot H^1}^2
+C\int_0^t\|v^\lambda(s)\|_{H^1}^4\,ds.
\end{equation}
In particular, using $\|\mathcal J_\lambda v_0\|_{\dot H^1}\le \|v_0\|_{\dot H^1}$ and
$\|v^\lambda(s)\|_{H^1}^4\le Y(t)^2$ for $s\le t$,
\begin{equation}\label{eq:B_L2_control}
\int_0^t\|v_t^\lambda(s)\|_{L^2}^2\,ds
\ \lesssim\
\|v_0\|_{H^1}^2+t\,Y(t)^2,
\qquad
\int_0^t\|v_t^\lambda(s)\|_{H^{-1/2}}^2\,ds
\ \lesssim\
\|v_0\|_{H^1}^2+t\,Y(t)^2,
\end{equation}
uniformly in $\lambda$. Insert \eqref{eq:B_L2_control} and $\int_0^t\|v^\lambda(s)\|_{H^1}^4\,ds\le t\,Y(t)^2$
into \eqref{eq:A_energy1_int}. Using $E_1^\lambda(0)\sim \|v_0\|_{H^1}^2$ and
$E_1^\lambda(s)\sim \|v^\lambda(s)\|_{H^1}^2$, we obtain
\begin{equation}\label{eq:C_poly_pre}
Y(t)
+\int_0^t\|\mathcal J_\lambda v^\lambda(s)\|_{H^2}^2\,ds
\ \le\
C_1\big(1+\|v_0\|_{H^1}^2\big)+C_2\,t\,Y(t)^2,
\qquad t\in[0,T],
\end{equation}
where $C_1,C_2$ are independent of $\lambda$ (and depend on $\Upsilon,\delta,\mu,\beta$ and $\|f_0\|_{H^3}$).
In particular we have the polynomial inequality
\begin{equation}\label{eq:C_poly}
Y(t)\ \le\ C_1\big(1+\|v_0\|_{H^1}^2\big)+C_2\,t\,Y(t)^2,
\qquad t\in[0,T].
\end{equation}
Choose $T>0$ such that
\[
C_2\,T\,\Big(2C_1(1+\|v_0\|_{H^1}^2)\Big)\ \le\ \frac12.
\]
A standard continuity/bootstrapping argument applied to \eqref{eq:C_poly} yields
\begin{equation}\label{eq:C_Y_uniform}
\sup_{0\le t\le T}\|v^\lambda(t)\|_{H^1}^2
=\sup_{0\le t\le T}Y(t)
\ \le\ 2C_1\big(1+\|v_0\|_{H^1}^2\big),
\end{equation}
with a bound independent of $\lambda$. Plugging \eqref{eq:C_Y_uniform} back into
\eqref{eq:C_poly_pre} and \eqref{eq:B_energy2_int} gives the remaining uniform bounds:
\begin{equation}\label{eq:lambda_uniform_bounds_final}
\sup_{0\le t\le T}\|v^\lambda(t)\|_{H^1}\le C,
\qquad
\int_0^T\|\mathcal J_\lambda v^\lambda(t)\|_{H^2}^2\,dt\le C,
\qquad
\int_0^T\|v_t^\lambda(t)\|_{H^{-1/2}}^2\,dt\le C,
\end{equation}
for a constant $C$ independent of $\lambda$. \medskip

By \eqref{eq:lambda_uniform_bounds_final} we have the $\lambda$--uniform bounds
\[
v^\lambda \ \text{bounded in } L^\infty(0,T;H^1(\TT^2)),
\qquad
\partial_t v^\lambda \ \text{bounded in } L^2(0,T;H^{-1/2}(\TT^2)).
\]
Fix $\theta\in(0,\tfrac12)$. Since the embedding
$H^1(\TT^2)\hookrightarrow H^{1-\theta}(\TT^2)$ is compact and
$H^{1-\theta}(\TT^2)\hookrightarrow H^{-1/2}(\TT^2)$ is continuous, the
Aubin--Lions lemma yields, after extracting a subsequence (not relabeled), the
existence of $v$ such that
\begin{equation}\label{eq:compact_vlambda_rewrite}
v^\lambda\to v \quad\text{strongly in } L^2(0,T;H^{1-\theta}(\TT^2)),
\qquad
v^\lambda\rightharpoonup^\ast v \quad\text{in } L^\infty(0,T;H^1(\TT^2)).
\end{equation}
Moreover, up to a further subsequence, there exists $U$ such that
\begin{equation}\label{eq:compact_Ulambda_rewrite}
\partial_t v^\lambda \rightharpoonup U \quad\text{weakly in } L^2(0,T;H^{-1/2}(\TT^2)).
\end{equation}
In particular, $v\in L^\infty(0,T;H^1(\TT^2))\cap W^{1,2}(0,T;H^{-1/2}(\TT^2))$,
and hence
\begin{equation}\label{eq:time_continuity_rewrite}
v\in C\big([0,T];H^{s}(\TT^2)\big)\qquad\text{for every }s\in[-\frac{1}{2},1).
\end{equation}

Since $v^\lambda\in L^\infty(0,T;H^1)$ and $\partial_t v^\lambda$ is bounded in
$L^2(0,T;H^{-1/2})$, after extraction we have $\partial_t v^\lambda\rightharpoonup U$
weakly in $L^2(0,T;H^{-1/2})$ for some $U$. For any
$\varphi\in C_c^\infty((0,T)\times\TT^2)$,
\[
\int_0^T\langle U^\lambda,\varphi\rangle\,dt
=
-\int_0^T\langle v^\lambda,\partial_t\varphi\rangle\,dt.
\]

We now pass to the limit $\lambda\to0$ in the regularized equation
\eqref{eq:v_equation_reg_step2}. The linear terms converge by
\eqref{eq:compact_vlambda_rewrite} and the boundedness of Fourier multipliers on
Sobolev spaces. For the nonlinear terms we use the regularization and the uniform bounds
\eqref{eq:lambda_uniform_bounds}. Consequently, the limit $v^\mu[\bar f]\in L^\infty(0,T;H^1_0(\TT^2))$ satisfies 
\begin{equation}\label{eq:limit_equation_lambda_vonly}
(I+\Upsilon\Lambda)\partial_t v+\delta\,\Lambda^3 v
+\mathcal J_\mu\Big(\Lambda+\frac{\beta}{4}\Lambda^5\Big)\mathcal J_\mu \bar f
+\mathbb N(\mathcal J_\mu\bar f,\,\partial_t v)
=
\varepsilon\,\mathcal Q(v),
\end{equation}

\subsubsection*{\underline{Step 4: Solving the equation for $f$}}

Fix $0<\mu\le1$ and let $T>0$ be the time obtained in Step~3. In particular, for
every $\bar f\in\mathbb X_T$ we dispose of a function
\[
v^\mu[\bar f]\in L^\infty\big(0,T;H^1_0(\TT^2)\big)\cap W^{1,2}\big(0,T;H^{-1/2}(\TT^2)\big),
\qquad
v^\mu=v_0,
\]
which satisfies
\begin{equation}\label{eq:v_equation_mu_bar_f_step4}
(I+\Upsilon\Lambda)\,\partial_t v^\mu[\bar f]+\delta\,\Lambda^3 v^\mu[\bar f]
+\mathcal J_\mu\Big(\Lambda+\frac{\beta}{4}\Lambda^5\Big)\mathcal J_\mu \bar f
+\mathbb N\big(\mathcal J_\mu\bar f,\partial_t v^\mu[\bar f]\big)
=
\varepsilon\,\mathcal Q\big(v^\mu[\bar f]\big),
\end{equation}
where
\[
\mathcal Q(w):=
\frac12\,\Lambda\Big(w^2-|\mathcal R w|^2\Big)-[\Lambda,w]\,w-\nabla_x w\cdot \mathcal R w.
\]

We recall that
\[
\mathbb X_T
:=
\Big\{\bar f\in L^\infty\big(0,T;H^3_0(\TT^2)\big):\ \bar f(0)=f_0,\ 
\|\bar f\|_{L^\infty(0,T;H^3)}\le 2\|f_0\|_{H^3}\Big\},
\]
which records the desired $H^3$ control. 

For $\bar f\in\mathbb X_T$, define
\begin{equation}\label{eq:Tmu_def_step4_final}
(\mathcal T_\mu[\bar f])(t):=f_0+\int_0^t \mathcal{J}_\mu\mathcal{J}_\mu v^\mu[\bar f](s)\,ds,\qquad t\in[0,T].
\end{equation}
 Since
$v^\mu[\bar f]\in L^\infty(0,T;H^1)$, the map $\mathcal T_\mu[\bar f]\in\mathbb X_T$ and
\begin{equation}\label{eq:ft_equals_v_step4_final}
\partial_t(\mathcal T_\mu[\bar f])=\mathcal{J}_\mu\mathcal{J}_\mu v^\mu[\bar f]
\qquad\text{in }L^\infty\big(0,T;H^1(\TT^2)\big).
\end{equation}
Thus, if $f=\mathcal T_\mu[f]$, then $f_t=\mathcal{J}_\mu\mathcal{J}_\mu v^\mu[f]$ and inserting $\bar f=f$ in
\eqref{eq:v_equation_mu_bar_f_step4} yields the coupled $\mu$--regularized system
\begin{align}
(I+\Upsilon\Lambda)\,v_t+\delta\,\Lambda^3 v
+\mathcal J_\mu\Big(\Lambda+\frac{\beta}{4}\Lambda^5\Big)\mathcal J_\mu f
+\mathbb N(\mathcal J_\mu f,v_t)
&=
\varepsilon\,\mathcal Q(v),
\label{eq:v_equation_mu_coupled_step4_final}\\
f_t&=\mathcal{J}_\mu\mathcal{J}_\mu v.
\label{eq:f_equation_mu_coupled_step4_final}
\end{align}
Let $\bar f^{(1)},\bar f^{(2)}\in\mathbb X_T$, and set
$v^{(i)}:=v^\mu[\bar f^{(i)}]$ and $\widetilde v:=v^{(1)}-v^{(2)}$.
Subtracting \eqref{eq:v_equation_mu_bar_f_step4} for $\bar f^{(1)}$ and $\bar f^{(2)}$ gives
\begin{align}\label{eq:vdiff_step4_final}
(I+\Upsilon\Lambda)\partial_t \widetilde v+\delta\Lambda^3\widetilde v
&=
-\mathcal J_\mu\Big(\Lambda+\frac{\beta}{4}\Lambda^5\Big)\mathcal J_\mu(\bar f^{(1)}-\bar f^{(2)})
-\Big(\mathbb N(\mathcal J_\mu\bar f^{(1)},\partial_t v^{(1)})
-\mathbb N(\mathcal J_\mu\bar f^{(2)},\partial_t v^{(2)})\Big)\nonumber\\
&\quad
+\varepsilon\Big(\mathcal Q(v^{(1)})-\mathcal Q(v^{(2)})\Big).
\end{align}
Using bilinearity of $\mathbb N$ in its second argument we split
\[
\mathbb N(\mathcal J_\mu\bar f^{(1)},\partial_t v^{(1)})
-\mathbb N(\mathcal J_\mu\bar f^{(2)},\partial_t v^{(2)})
=
\mathbb N(\mathcal J_\mu(\bar f^{(1)}-\bar f^{(2)}),\partial_t v^{(1)})
+\mathbb N(\mathcal J_\mu\bar f^{(2)},\partial_t\widetilde v).
\]
Testing \eqref{eq:vdiff_step4_final} against $\Lambda \widetilde v$ and arguing as in
Step~3 (using that $\mathcal J_\mu\bar f^{(2)}$ is smooth and bounded in $W^{1,\infty}$
with constants depending on $\mu$ and $\|\bar f^{(2)}\|_{H^3}$) yields, for a.e.\ $t\in(0,T)$,
\begin{equation}\label{eq:vdiff_energy_step4_final}
\frac{d}{dt}\Big(\|\widetilde v\|_{\dot H^{1/2}}^2+\Upsilon\|\widetilde v\|_{\dot H^1}^2\Big)
+\delta\|\widetilde v\|_{\dot H^2}^2
\ \le\ C_{\mu,R}\,\|\mathcal J_\mu(\bar f^{(1)}-\bar f^{(2)})\|_{H^3}^2
+ C_{\mu,R}\,\|\widetilde v\|_{H^1}^2,
\end{equation}
where $R:=\sup_{\bar f\in\mathbb X_T}\|\bar f\|_{L^\infty H^3}\le 2\|f_0\|_{H^3}$ and
$C_{\mu,R}$ depends on $\mu$, the model parameters and $R$, but not on $T$.
Since $\widetilde v(0)=0$, Gr\"onwall implies
\begin{equation}\label{eq:v_Lip_in_f_step4_final}
\|v^\mu[\bar f^{(1)}]-v^\mu[\bar f^{(2)}]\|_{L^\infty(0,T;H^1)}
\ \le\ C_{\mu,R}T\,\|\bar f^{(1)}-\bar f^{(2)}\|_{L^\infty(0,T;H^1)}.
\end{equation}\medskip
To show the contraction of $\mathcal T_\mu$ in $\mathbb X_T$,  let $\bar f^{(1)},\bar f^{(2)}\in\mathbb X_T$. Using \eqref{eq:Tmu_def_step4_final} and
\eqref{eq:v_Lip_in_f_step4_final},
\[
\|\mathcal T_\mu[\bar f^{(1)}]-\mathcal T_\mu[\bar f^{(2)}]\|_{\mathbb X_T}
\le
C_\mu T\,\|v^\mu[\bar f^{(1)}]-v^\mu[\bar f^{(2)}]\|_{L^\infty H^1}
\le
T^2\,C_{\mu,R}\,\|\bar f^{(1)}-\bar f^{(2)}\|_{L^\infty H^1}.
\]
Since $\|h\|_{H^1}\le \|h\|_{H^3}$ and we can take $0<T_\mu<1$, we also have
\[
\|\mathcal T_\mu[\bar f^{(1)}]-\mathcal T_\mu[\bar f^{(2)}]\|_{\mathbb X_T}
\le
T\,C_{\mu,R}\,\|\bar f^{(1)}-\bar f^{(2)}\|_{\mathbb X_T}.
\]
Choose $T_\mu\in(0,T]$ such that $T_\mu\,C_{\mu,R}\le \tfrac12$.
Then $\mathcal T_\mu$ is a strict contraction on $\mathbb X_{T_\mu}$, and Banach's fixed point theorem yields a unique
\[
f^\mu\in \mathbb X_{T_\mu}
\quad\text{such that}\quad
f^\mu=\mathcal T_\mu[f^\mu].
\]
Setting $v^\mu:=v^\mu[f^\mu]$, we obtain a pair $(f^\mu,v^\mu)$ satisfying
\eqref{eq:v_equation_mu_coupled_step4_final}--\eqref{eq:f_equation_mu_coupled_step4_final}
on $(0,T_\mu)$, with $f^\mu(0)=f_0$ and $v^\mu(0)=v_0$. \medskip

At this stage we know $f^\mu\in L^\infty(0,T_\mu;H^3)$ and $f_t^\mu=v^\mu\in L^\infty(0,T_\mu;H^1)$.
In Step~5 we prove $\mu$--uniform a priori estimates at the $H^3$ level for the coupled
system \eqref{eq:v_equation_mu_coupled_step4_final}--\eqref{eq:f_equation_mu_coupled_step4_final}.
In particular, there exists a time $T_*>0$, depending only on
$\|f_0\|_{H^3}+\|v_0\|_{H^1}$ and the parameters $(\Upsilon,\delta,\beta)$ but
\emph{independent of $\mu$}, such that, for $T_* \le T_\mu$,
\[
\sup_{0\le t\le T_*}\|f^\mu(t)\|_{H^3}\le 2\|f_0\|_{H^3}.
\]
Thus $f^\mu\in\mathbb X_{T_*}$ and $f_t^\mu\in L^\infty(0,T_*;H^1)$, which is the
regularity required to pass to the limit $\mu\to0$ in Step~5.

\subsubsection*{\underline{Step 5: Passing to the limit as $\mu\to 0$}}

Let $(f^\mu,v^\mu)$ be the $\mu$--regularized solution from Step~4, with $v^\mu=f_t^\mu$.
Repeating the energy multipliers from Step~3 (now with $\bar f=f^\mu$ and $v=v^\mu$)
yields, for a.e.\ $t\in(0,T)$,
\begin{align}
\frac{1}{2}\frac{d}{dt}\Big(
\|f_t^\mu\|_{H^{1/2}}^2+\Upsilon\|f_t^\mu\|_{H^{1}}^2
+\|f^\mu\|_{H^1}^2+\frac{\beta}{4}\|f^\mu\|_{H^3}^2
\Big)
&\le \|f_t^\mu\|_{H^1}^4+C\|\mathcal{J}_\mu f^\mu\|_{H^2}^2\|f_{tt}^\mu\|_{L^2}^2,
\label{eq:mu_energy_1_final}\\[0.3em]
\frac{1}{2}\|f_{tt}^\mu\|_{H^{-1/2}}^2+\frac{\Upsilon}{4}\|f_{tt}^\mu\|_{L^2}^2
+\frac{\delta}{2}\frac{d}{dt}\|f_t^\mu\|_{H^1}^2
&\le \|f_t^\mu\|_{H^1}^4.
\label{eq:mu_energy_2_final}
\end{align}
Define
\[
Y^\mu(t):=
\|f_t^\mu(t)\|_{H^{1/2}}^2+\Upsilon\|f_t^\mu(t)\|_{H^{1}}^2
+\|f^\mu(t)\|_{H^1}^2+\frac{\beta}{4}\| f^\mu(t)\|_{H^3}^2.
\]
Since $\mathcal J_\mu$ is bounded on $H^s$ and $\|g\|_{H^2}\lesssim \|g\|_{H^3}$, we have
\[
\|\mathcal J_\mu f^\mu\|_{H^2}^2\ \lesssim\ \|f^\mu\|_{H^3}^2\ \lesssim\ Y^\mu,
\qquad
\|f_t^\mu\|_{H^1}^2\ \lesssim\ Y^\mu.
\]

Integrating \eqref{eq:mu_energy_2_final} on $[0,t]$ and using $\Upsilon>0$ gives
\begin{equation}\label{eq:ftt_L2_uniform_mu}
\int_0^t \|f_{tt}^\mu(s)\|_{L^2}^2\,ds
\ \lesssim_{\Upsilon}\ 
\|v_0\|_{H^1}^2+\int_0^t \|f_t^\mu(s)\|_{H^1}^4\,ds
\ \lesssim\ 
1+\int_0^t \big(Y^\mu(s)\big)^2\,ds.
\end{equation}
Now integrate \eqref{eq:mu_energy_1_final} on $[0,t]$ and estimate
\[
\int_0^t \|\mathcal J_\mu f^\mu\|_{H^2}^2\|f_{tt}^\mu\|_{L^2}^2
\ \le\
\Big(\sup_{0\le s\le t}Y^\mu(s)\Big)\int_0^t \|f_{tt}^\mu(s)\|_{L^2}^2\,ds,
\]
then use \eqref{eq:ftt_L2_uniform_mu}. Writing $Z^\mu(t):=\sup_{0\le s\le t}Y^\mu(s)$ we obtain
\begin{equation}\label{eq:Zmu_poly_final}
Z^\mu(t)\ \le\ C_0 + C_1\,t\,(Z^\mu(t))^2 + C_2\,t\,(Z^\mu(t))^3,
\qquad t\in[0,T],
\end{equation}
where $C_0\sim 1+Y^\mu(0)$ and $C_1,C_2$ depend only on the model parameters and on
$\|f_0\|_{H^3}+\|v_0\|_{H^1}$, and are independent of $\mu$.
A standard continuity argument applied to \eqref{eq:Zmu_poly_final} yields a time $T_*>0$
(depending only on $(f_0,v_0)$, but not on $\mu$) such that
\begin{equation}\label{eq:Ymu_uniform_final}
\sup_{0\le t\le T_*}Y^\mu(t)\ \le\ 2Y^\mu(0)
\qquad\text{for all }0<\mu\le1.
\end{equation}
In particular, uniformly in $0<\mu\le1$,
\begin{equation}\label{eq:mu_uniform_bounds_final}
\{f^\mu\}\ \text{is bounded in }L^\infty(0,T_*;H^3(\TT^2)),
\qquad
\{f_t^\mu\}\ \text{is bounded in }L^\infty(0,T_*;H^1(\TT^2)).
\end{equation}
\begin{equation}\label{eq:mu_uniform_bounds_final_2}
\{f_{tt}^\mu\}\ \text{is bounded in }L^2(0,T_*;H^{-1/2}(\TT^2)).
\end{equation} 
By \eqref{eq:mu_uniform_bounds_final} the family $\{f^\mu\}$ is equi-Lipschitz in time
with values in $H^1$ and bounded in $L^\infty_tH^3_x$. As in Step~3, Simon's compactness
theorem (using $H^3\Subset H^{2+\theta}\hookrightarrow H^1$) yields, after extracting a
subsequence, $f^\mu\to f$ strongly in $C([0,T_*];H^{2+\theta}(\TT^2))$ for any $\theta\in(0,1)$,
and $f^\mu\rightharpoonup^\ast f$ in $L^\infty(0,T_*;H^3)$. Moreover,
$f_t^\mu\rightharpoonup^\ast f_t$ in $L^\infty(0,T_*;H^1)$ and
$f_{tt}^\mu\rightharpoonup f_{tt}$ weakly in $L^2(0,T_*;H^{-1/2})$.

Using these convergences, the fact that $\mathcal J_\mu\to I$ strongly on $H^s$ as $\mu\to0$,
and the continuity of $\mathbb N(\cdot,\cdot)$ and $\mathcal Q(\cdot)$ at this regularity,
we may pass to the limit $\mu\to0$ in the $\mu$--regularized second-order equation and conclude that $f$ solves
\eqref{eq:f_equation_compact_closed} on $(0,T_*)\times\TT^2$, with $f(\cdot,0)=f_0$ and
$f_t(\cdot,0)=v_0$. In particular,
\[
f\in L^\infty(0,T_*;H^3_{0}(\TT^2)),\qquad f_t\in L^\infty(0,T_*;H^1_{0}(\TT^2)).
\]

\subsubsection*{\underline{Step 6: Uniqueness}}
We now show that the solution constructed above is unique (in the natural class
given by Theorem~\ref{thm:WPbi}). Recall that we have obtained a pair $(f,v)$
satisfying
\begin{align}
(I+\Upsilon\Lambda)\,v_t+\delta\,\Lambda^3 v
+\Big(\Lambda+\frac{\beta}{4}\Lambda^5\Big)f
+\mathbb N\!\big(f,v_t\big)
&=
\varepsilon\,\mathcal Q(v),
\label{eq:v_equation_mu_coupled_step4_final2}\\
f_t&= v.
\label{eq:f_equation_mu_coupled_step4_final2}
\end{align}

Let $(f^{(1)},v^{(1)})$ and $(f^{(2)},v^{(2)})$ be two solutions of
\eqref{eq:v_equation_mu_coupled_step4_final2}--\eqref{eq:f_equation_mu_coupled_step4_final2}
on $[0,T]$ with the same initial data. Setting
\[
\widetilde f:=f^{(1)}-f^{(2)},\qquad \widetilde v:=v^{(1)}-v^{(2)},
\]
we obtain a coupled system for $(\widetilde f,\widetilde v)$ with zero initial
data. Testing the $\widetilde v$--equation against $\Lambda \widetilde v$ (and,
when needed, using the relation $\widetilde f_t=\widetilde v$), and arguing as
in the a priori estimates of Steps~3--5, we derive an energy inequality for a
suitable quantity $Z(t)$ controlling $\|\widetilde f\|_{H^3}$ and
$\|\widetilde v\|_{H^1}$. In particular, for $t\in[0,T]$ one obtains an estimate
of the form
\begin{equation}\label{eq:Zmu_poly_final2}
Z(t)\ \le\ C_1\,t\,(Z(t))^2 + C_2\,t\,(Z(t))^3,
\qquad Z(0)=0,
\end{equation}
where the constants $C_1,C_2$ depend only on the parameters of the model and on
the \emph{a priori} bounds for the two solutions on $[0,T]$. For $T>0$ chosen as
in the existence argument (so that the right-hand side can be absorbed), the
polynomial inequality \eqref{eq:Zmu_poly_final2} forces $Z(t)\equiv0$ on $[0,T]$.
Consequently $\widetilde f\equiv0$ and $\widetilde v\equiv0$, and hence
$f^{(1)}\equiv f^{(2)}$ and $v^{(1)}\equiv v^{(2)}$ on $[0,T]$. \medskip

Therefore the solution is unique in the stated class, which completes the proof.
\end{proof}

\begin{remark}
Let us emphasize that we do not currently obtain an analogue of Theorem~\ref{thm:WPbi} for the alternative bidirectional model \eqref{eq:f_model_compact_final_again}. While one can construct regularized/approximate solutions for that equation by standard procedures, we are unable to close the corresponding a priori estimates uniformly at the level required to pass to the limit. Establishing a well-posedness theory for this second bidirectional closure therefore remains an open problem.
\end{remark}

\section{Well-posedness results for the unidirectional models}\label{Wp:uni}
In this section we study the Cauchy problem for the unidirectional models
derived in Subsection~\ref{subsec:uni:models} as an $\mathcal{O}(\varepsilon)$ approximation
of the full hydroelastic system. The analysis relies on energy estimates adapted
to the nonlocal operator structure of the model, together with standard
multilinear and commutator bounds for the Fourier multipliers appearing in the
equation. \medskip

We first establish a local well-posedness result for the unidirectional model \eqref{eq:Ft_final_ab_noG}, stated as follows.
\begin{theorem}[Local well-posedness]\label{th:LWP:uni}
Let $\Upsilon>0$ and $\delta, \beta>0$.
Assume that $F_0\in H^2(\TT)$ has zero spatial mean and consider the
unidirectional model \eqref{eq:Ft_final_ab_noG}. Then there exist $T>0$, depending only on
$\|F_0\|_{H^2(\TT)}$ and the parameters of the model, and a unique solution
\[
F\in C([0,T];H^2(\TT)),\qquad F(\cdot,0)=F_0.
\]
\end{theorem}

\begin{proof}[Proof of Theorem \ref{th:LWP:uni}]
 The argument combines classical a priori energy estimates with a standard
regularization scheme based on mollifiers, cf. \cite{MajdaBertozzi2002}.
We first derive the required a priori bounds, then construct solutions at the
regularized level and pass to the limit, and finally establish uniqueness.
For clarity, we decompose the proof into several steps.
\subsubsection*{\underline{Step 1: conservation of the mean}}
Let $\langle F(\tau)\rangle:=\frac{1}{2\pi}\int_{\TT}F(\xi,\tau)\,d\xi$.
We claim that any sufficiently smooth real-valued solution of \eqref{eq:Ft_final_ab_noG}
satisfies
\[
\langle F(\tau)\rangle=\langle F_0\rangle\qquad\text{for all }\tau\in[0,T].
\]

Indeed, taking the spatial mean of \eqref{eq:Ft_final_ab_noG}, the linear part has
zero average (it is a combination of $\xi$--derivatives and Fourier multipliers
annihilating constants). With the convention $a(0)=\tfrac12$ and $b(0)=0$, we obtain
\[
\partial_\tau\langle F\rangle
=
-\frac12\,\Big\langle\,2F_\xi\,\Lambda F-\comm{\Lambda}{F}F_\xi\,\Big\rangle.
\]
Expanding the commutator,
\[
2F_\xi\,\Lambda F-\comm{\Lambda}{F}F_\xi
=
2F_\xi\,\Lambda F-\Lambda(FF_\xi)+F\,\Lambda F_\xi,
\]
and using that $\langle \Lambda(\cdot)\rangle=0$, it remains to show that
$\langle 2F_\xi\Lambda F+F\Lambda F_\xi\rangle=0$. In one dimension
$\Lambda=\mathcal H\partial_\xi$, hence $\Lambda F=\mathcal H F_\xi$ and
$\Lambda F_\xi=\mathcal H F_{\xi\xi}$. Therefore, 
\[
\langle F_\xi\Lambda F\rangle=\frac{1}{2\pi}\int_{\TT}F_\xi\,\mathcal H F_\xi\,d\xi=0,
\qquad
\langle F\Lambda F_\xi\rangle=\frac{1}{2\pi}\int_{\TT}F\,\mathcal H F_{\xi\xi}\,d\xi
=-\frac{1}{2\pi}\int_{\TT}F_\xi\,\mathcal H F_\xi\,d\xi=0,
\]
Hence $\partial_\tau\langle F\rangle=0$ and the mean is conserved.
\smallskip
In particular, if $\langle F_0\rangle=0$, then $F(\cdot,\tau)$ remains mean-zero
for all times of existence.

\subsubsection*{\underline{Step 2: a priori estimates}}
In order to derive the a priori estimates it is convenient to work with a symmetrized formulation of the evolution, namely 
\begin{equation}\label{eq:Mstar_step_expanded3}
	\varepsilon(\mathcal M^\ast\mathcal M)F_\tau
	=\mathcal M^\ast\!\left((I+\Upsilon\Lambda)F_\xi+\frac{\beta}{4}\Lambda F_{\xi\xi\xi}
	+\delta\Lambda F_{\xi\xi}+\mathcal H F\right)
	-\varepsilon\,\mathcal M^\ast\!\left(2 F_\xi\,\Lambda F-\comm{\Lambda}{F}F_\xi\right).
\end{equation}
where
 \[\mathcal M^\ast\mathcal M =4(I+\Upsilon\Lambda)^2-\delta^2\Lambda^2\partial_\xi^2, \quad
\mathcal M^\ast=2(I+\Upsilon\Lambda)-\delta\,\Lambda\partial_\xi. \]

Taking the $L^2(\TT)$ inner product of \eqref{eq:Mstar_step_expanded3} with $F$ and integrating
by parts in $\xi$, we obtain the evolution identity
\begin{align}
\frac{d}{d\tau}
\Big(\|F\|_{L^2}^2+2\Upsilon\|F\|_{\dot H^{1/2}}^2+\Upsilon^2\|F\|_{\dot H^{1}}^2+\frac{\delta^2}{4}\|F\|_{\dot H^{2}}^2\Big)
&=
\!\int_{\TT}\!(I+\Upsilon\Lambda)\Big((I+\Upsilon\Lambda)F_\xi\Big)\,F\,d\xi
+\frac{\beta}{4}\!\int_{\TT}\!(I+\Upsilon\Lambda)\Lambda F_{\xi\xi\xi}\,F\,d\xi \nonumber
\\[0.3em]
&\quad
+\delta\!\int_{\TT}\!(I+\Upsilon\Lambda)\Lambda F_{\xi\xi}\,F\,d\xi
+\!\int_{\TT}\!(I+\Upsilon\Lambda)\mathcal HF\,F\,d\xi
\nonumber\\[0.3em]
&\quad
+\frac{\delta}{2}\!\int_{\TT}\!(I+\Upsilon\Lambda)F_\xi\,\Lambda F_\xi\,d\xi
+\frac{\beta\delta}{8}\!\int_{\TT}\!\Lambda F_{\xi\xi\xi}\,\Lambda F_\xi\,d\xi\nonumber\\[0.3em]
&\quad+\frac{\delta^2}{2}\!\int_{\TT}\!\Lambda F_{\xi\xi}\,\Lambda F_\xi\,d\xi
+\frac{\delta}{2}\!\int_{\TT}\!\mathcal HF\,\Lambda F_\xi\,d\xi
\nonumber\\[0.3em]
&\quad
-2\!\int_{\TT}\!(I+\Upsilon\Lambda)\big(F_\xi\Lambda F\big)\,F\,d\xi
+\!\int_{\TT}\!(I+\Upsilon\Lambda)\big(\comm{\Lambda}{F}F_\xi\big)\,F\,d\xi
\nonumber\\[0.3em]
&\quad
-\delta\!\int_{\TT}\!F_\xi(\Lambda F)(\Lambda F_\xi)\,d\xi
+\frac{\delta}{2}\!\int_{\TT}\!\big(\comm{\Lambda}{F}F_\xi\big)\,(\Lambda F_\xi)\,d\xi=\displaystyle\sum_{j=1}^{12} I_{j}.
\nonumber
\end{align}
Using that $A:=I+\Upsilon\Lambda$ and $\Lambda$ are self-adjoint on $L^2(\TT)$,
that $\partial_\xi$ commutes with $A$ and $\Lambda$, and that $\mathcal H^\ast=-\mathcal H$,
we obtain the cancellations $I_1=I_{2}=I_{4}=I_{7}=0$. Moreover, we also find that
\[
I_6
=\frac{\beta\delta}{8}\int_{\TT}(\Lambda F_\xi)_{\xi\xi}\,(\Lambda F_\xi)\,d\xi
=-\frac{\beta\delta}{8}\,\|F\|_{\dot H^3(\TT)}^2, \quad I_3
=-\delta\Big(\|F\|_{\dot H^{3/2}}^2+\Upsilon\|F\|_{\dot H^{2}}^2\Big), \quad I_{8}=-\frac{\delta}{2}\,\|F\|_{\dot H^1(\TT)}^2.
\]
We now bound the nonlinear contributions $I_9,I_{10},I_{11},I_{12}$.
Let $A:=I+\Upsilon\Lambda$. By self-adjointness of $A$,
\[
I_9=-2\int_{\TT}A(F_\xi\Lambda F)\,F\,d\xi=-2\langle F_\xi\Lambda F,AF\rangle,
\qquad
I_{10}=\int_{\TT}A([\Lambda,F]F_\xi)\,F\,d\xi=\langle [\Lambda,F]F_\xi,AF\rangle.
\]
Hence, integrating by parts and using Cauchy--Schwarz,
\begin{align}
|I_9|
&\le \norm{F}_{L^{\infty}}\norm{F}_{\dot H^2(\TT)}^2
\label{eq:I9_bound}
\end{align}
Similarly, to bound $I_{10}$ we can expand the commutator and use integration by parts to find that 
\begin{align}
|I_{10}|
&\le  \norm{F}_{L^{\infty}}\norm{F}_{\dot H^2(\TT)}^2
\label{eq:I10_bound}
\end{align}
For $I_{11}$ and $I_{12}$ we similarly have
\[
I_{11}=\frac{\delta}{2}\int_{\TT}F_{\xi\xi}(\Lambda F)^{2}\,d\xi,
\qquad
I_{12}=\frac{\delta}{2}\int_{\TT} \left(\Lambda(F F_{\xi})-F\Lambda F_{\xi}\right)\Lambda F_{\xi}\,d\xi.
\]
By means of Gagliardo-Nirenberg interpolation inequality \eqref{eq:GN_homogeneous_W14}
\begin{align}
|I_{11}|
&\le \delta\,\|F\|_{\dot H^2(\TT)}\,\| \Lambda F\|_{L^4}^{2}\le C\delta\, \norm{F}_{L^{\infty}}\|F\|_{\dot H^2(\TT)}^{2}
\label{eq:I11_bound}
\end{align}
On the other hand, using Cauchy-Schwarz inequality and Moser estimate \eqref{eq:alg_inhom_1d} we find that
\begin{align}
I_{12}\leq \frac{\delta}{4} \norm{\partial_{\xi}\Lambda\left(F^{2}\right)}_{L^{2}} \norm{\Lambda F_{\xi}}_{L^{2}} +  \frac{\delta}{2}\norm{F}_{L^{\infty}}\norm{\Lambda F_{\xi}}_{L^{2}}^{2} 
&\le C \norm{F}_{L^{\infty}}\norm{F}_{\dot{H}^{2}(\TT)}^{2}
\end{align}
Therefore, combining all the previous estimates, we have shown that
\begin{align}
	\frac{d}{d\tau}
	\Big(\|F\|_{L^2}^2+2\Upsilon\|F\|_{\dot H^{1/2}}^2+\Upsilon^2\|F\|_{\dot H^{1}}^2+\frac{\delta^2}{4}\|F\|_{\dot H^{2}}^2\Big) +\frac{\beta\delta}{8}\,\|F\|_{\dot H^3(\TT)}^2\le C  \norm{F}_{L^{\infty}}\norm{F}_{\dot{H}^{2}(\TT)}^{2}
\end{align}
Define
\[
E(\tau):=\|F\|_{L^2}^2+2\Upsilon\|F\|_{\dot H^{1/2}}^2+\Upsilon^2\|F\|_{\dot H^{1}}^2+\frac{\delta^2}{4}\|F\|_{\dot H^{2}}^2.
\]
Moreover by means of the Sobolev embedding \eqref{eq:Sobolev_1d_embed_hom}, we have that
$\|F\|_{L^\infty}\lesssim \|F\|_{H^2}\lesssim E(\tau)^{1/2}$.
Therefore the energy inequality implies
\[
E'(\tau)+\frac{\beta\delta}{8}\|F\|_{\dot H^3}^2
\le C\,\|F\|_{L^\infty}\,\|F\|_{\dot H^2}^2
\lesssim C\,E(\tau)^{3/2},
\]
and in particular $E'(\tau)\le C E(\tau)^{3/2}$.
Solving this differential inequality gives
\[
E(\tau)^{-1/2}\ge E(0)^{-1/2}-\frac{C}{2}\tau,
\]
so for $0\le \tau\le T_*:=\frac{1}{C}E(0)^{-1/2}$ we have the uniform bound
$E(\tau)\le 4E(0)$.

\subsubsection*{\underline{Step 3: existence via mollification}}
We define the mollified unknown $F^\nu$ as the unique solution to the regularized
Cauchy problem
\begin{equation}\label{eq:mollified_system_Jnu}
	\begin{cases}
		\partial_\tau F^\nu
		=\mathcal J_\nu\Bigg\{
		\dfrac{1}{\varepsilon}\big(a+b\mathcal H\big)\!\Big(
		(I+\Upsilon\Lambda)\partial_\xi(\mathcal J_\nu F^\nu)
		+\dfrac{\beta}{4}\Lambda\partial_\xi^3(\mathcal J_\nu F^\nu)
		+\delta\Lambda\partial_\xi^2(\mathcal J_\nu F^\nu)
		+\mathcal H(\mathcal J_\nu F^\nu)\Big)\\[0.4em]
		\hspace{2.8cm}
		-\big(a+b\mathcal H\big)\!\Big(
		2\,\partial_\xi(\mathcal J_\nu F^\nu)\,\Lambda(\mathcal J_\nu F^\nu)
		-\comm{\Lambda}{\mathcal J_\nu F^\nu}\,\partial_\xi(\mathcal J_\nu F^\nu)
		\Big)
		\Bigg\},\\[0.4em]
		F^\nu(\cdot,0)=\mathcal J_\nu F_0 .
	\end{cases}
\end{equation}

Since $\mathcal J_\nu$ is smoothing, the right-hand side of
\eqref{eq:mollified_system_Jnu} defines a locally Lipschitz vector field on
$H^2(\TT)$.  Therefore, by the Picard--Lindel\"of
theorem in Banach spaces, there exists $T_\nu>0$ and a unique solution
\[
F^\nu\in C^1([0,T_\nu];H^2(\TT)).
\]
Because $\mathcal J_\nu$ is self-adjoint on $L^2(\TT)$ and commutes with
$\partial_\xi$, $\Lambda$, $\mathcal H$, $(I+\Upsilon\Lambda)$ and the multipliers
$a,b$, the a priori energy computation in Step~2 can be repeated for the mollified
system \eqref{eq:mollified_system_Jnu}. The only modification is that $\mathcal J_\nu$
may be shifted between factors using $\langle \mathcal J_\nu G,F^\nu\rangle
=\langle G,\mathcal J_\nu F^\nu\rangle$. Consequently, the same cancellations and
sign identities hold, and we obtain the same differential inequality for the
mollified energy $E^\nu(\tau)$ with constants independent of $\nu$.  Denoting
\[
E^\nu(\tau)
:=\|F^\nu\|_{L^2}^2
+2\Upsilon\|F^\nu\|_{\dot H^{1/2}}^2
+\Upsilon^2\|F^\nu\|_{\dot H^{1}}^2
+\frac{\delta^2}{4}\|F^\nu\|_{\dot H^{2}}^2,
\]
we obtain, for $\tau\in[0,T_\nu)$,
\begin{equation}\label{eq:Enu_ode}
	\frac{d}{d\tau}E^\nu(\tau)\le C\,(E^\nu(\tau))^{3/2},
\end{equation}
with a constant $C$ independent of $\nu$. Solving \eqref{eq:Enu_ode} yields a time
$T_*=T_*(E(0))>0$ such that
\begin{equation}\label{eq:Enu_uniform}
	\sup_{0\le \tau\le T_*}E^\nu(\tau)\le 4E(0)
	\qquad\text{for all }\nu\ge1.
\end{equation}
In particular, $T_\nu\ge T_*$ and $\{F^\nu\}$ is bounded in
$L^\infty([0,T_*];H^2(\TT))$ uniformly in $\nu$. \medskip

By the uniform bound \eqref{eq:Enu_uniform} and the mollified equation
\eqref{eq:mollified_system_Jnu}, the sequence $\{F^\nu\}$ is bounded in
$L^\infty([0,T_*];H^2(\TT))$ and $\{\partial_\tau F^\nu\}$ is bounded in
$L^\infty([0,T_*];H^1(\TT))$, uniformly in $\nu$. Hence, after extracting a
subsequence, $F^\nu\to F$ strongly in $C([0,T_*];H^{1+\theta}(\TT))$ for every
$\theta\in(0,1)$ and weakly-$\ast$ in $L^\infty([0,T_*];H^2(\TT))$. Using the
continuity of Fourier multipliers on Sobolev spaces and standard
product/commutator stability at the $H^2$ level, we may pass to the limit in the
mollified equation and conclude that $F$ solves \eqref{eq:Ft_final_ab_noG} on
$[0,T_*]$ with $F(\cdot,0)=F_0$, and moreover $F\in C([0,T_*];H^2(\TT))$.

\subsubsection*{\underline{Step 4: uniqueness}}

Let $F$ and $G$ be two solutions of \eqref{eq:Mstar_step_expanded3} on $[0,T]$
such that
\[
F,G\in L^\infty([0,T];H^2(\TT)),\qquad F(\cdot,0)=G(\cdot,0)=F_0.
\]
Set $W:=F-G$. Subtracting \eqref{eq:Mstar_step_expanded3} (written for $F$ and $G$)
yields
\begin{equation}\label{eq:W_symmetrized_eps1}
	(\mathcal M^\ast\mathcal M)W_\tau
	=\mathcal M^\ast\!\left((I+\Upsilon\Lambda)W_\xi+\frac{\beta}{4}\Lambda W_{\xi\xi\xi}
	+\delta\Lambda W_{\xi\xi}+\mathcal H W\right)
	-\mathcal M^\ast\!\big(\mathcal N(F)-\mathcal N(G)\big),
\end{equation}
where $\mathcal N(U):=2U_\xi\,\Lambda U-\comm{\Lambda}{U}U_\xi$.

Testing \eqref{eq:W_symmetrized_eps1} against $W$ in $L^2(\TT)$ and repeating the
energy computation of Step~2 gives
\begin{equation}\label{eq:EW_diff_ineq_start_eps1}
	\frac{d}{d\tau}E_W(\tau)+\frac{\beta\delta}{8}\,\|W\|_{\dot H^3(\TT)}^2
	\;\lesssim\;
	\big|\langle \mathcal M^\ast(\mathcal N(F)-\mathcal N(G)),W\rangle\big|,
\end{equation}
with
\[
E_W(\tau):=\|W\|_{L^2}^2+2\Upsilon\|W\|_{\dot H^{1/2}}^2+\Upsilon^2\|W\|_{\dot H^{1}}^2
+\frac{\delta^2}{4}\|W\|_{\dot H^{2}}^2.
\]
Moreover, expanding with $F=G+W$ we have
\begin{equation}\label{eq:N_diff_expanded_eps1}
	\mathcal N(F)-\mathcal N(G)
	=
	2W_\xi\,\Lambda F+2G_\xi\,\Lambda W
	-\comm{\Lambda}{W}F_\xi-\comm{\Lambda}{G}W_\xi.
\end{equation}
Using $H^2(\TT)\hookrightarrow W^{1,\infty}(\TT)$ and the standard commutator bound
$\|\comm{\Lambda}{f}g\|_{L^2}\lesssim \|f_\xi\|_{L^\infty}\|g\|_{L^2}$, the nonlinear
difference is Lipschitz in the energy space, namely
\[
\big|\langle \mathcal M^\ast(\mathcal N(F)-\mathcal N(G)),W\rangle\big|
\lesssim \big(\|F\|_{H^2}+\|G\|_{H^2}\big)\,E_W(\tau).
\]
Therefore,
\[
\frac{d}{d\tau}E_W(\tau)
\le C\big(1+\|F\|_{H^2}+\|G\|_{H^2}\big)\,E_W(\tau).
\]
By Step~2, $F$ and $G$ are bounded in $L^\infty([0,T];H^2(\TT))$ and $E_W(0)=0$.
Hence Gr\"onwall's inequality implies $E_W(\tau)\equiv0$ on $[0,T]$, so $W\equiv0$
and $F=G$.
\end{proof}
In the sequel, we establish global-in-time existence for the unidirectional model
\eqref{eq:Ft_final_ab_noG} under a smallness assumption on the initial datum.
More precisely, we have the following result.
\begin{theorem}[Global existence and decay for small data]
	\label{th:GWP_decay_H2}
	Let $\Upsilon>0$ and $\delta,\beta>0$. Assume that $F_0\in H^2(\TT)$ has zero
	spatial mean. There exists $\varepsilon_\ast>0$, depending only on
	$\Upsilon,\delta,\beta$, such that if
	\[
	\|F_0\|_{H^2(\TT)}\le \varepsilon_\ast,
	\]
	then the solution $F$ given by Theorem~\ref{th:LWP:uni} exists globally in time,
	\[
	F\in C\big([0,\infty);H^2(\TT)\big),
	\]
	Moreover,  there exist constants $c,C>0$,
	depending only on $\Upsilon,\delta,\beta$, such that for all $\tau\ge0$,
	\[
	E(\tau)\le C\,E(0)\,e^{-c\tau},
	\qquad
	E(\tau):=\|F\|_{L^2}^2+2\Upsilon\|F\|_{\dot H^{1/2}}^2+\Upsilon^2\|F\|_{\dot H^{1}}^2
	+\frac{\delta^2}{4}\|F\|_{\dot H^{2}}^2.
	\]
\end{theorem}

\begin{proof}[Proof of Theorem~\ref{th:GWP_decay_H2}]
	Let $F$ be the local $H^2(\TT)$ solution given by Theorem~\ref{th:LWP:uni}.
	By Step~1 in the proof of Theorem~\ref{th:LWP:uni}, the spatial mean is conserved,
	so $\langle F(\tau)\rangle=0$ for all times of existence.
	
	Recall the energy from Step~2,
	\[
	E(\tau):=\|F\|_{L^2}^2+2\Upsilon\|F\|_{\dot H^{1/2}}^2+\Upsilon^2\|F\|_{\dot H^{1}}^2
	+\frac{\delta^2}{4}\|F\|_{\dot H^{2}}^2,
	\]
	and the differential inequality obtained there:
	\begin{equation}\label{eq:GWP_energy_recall}
		E'(\tau)+\frac{\beta\delta}{8}\,\|F(\tau)\|_{\dot H^3(\TT)}^2
		\le C\,\|F(\tau)\|_{L^\infty(\TT)}\,\|F(\tau)\|_{\dot H^2(\TT)}^2,
	\end{equation}
	where $C$ depends only on $\Upsilon,\delta,\beta$. \medskip
	
	Since $\langle F\rangle=0$, all Fourier modes satisfy $|k|\ge1$, hence
	$\|F\|_{\dot H^2}\le \|F\|_{\dot H^3}$. Moreover, in one dimension
	$H^2(\TT)\hookrightarrow L^\infty(\TT)$ and $E(\tau)\sim \|F(\tau)\|_{H^2}^2$, so
	\begin{equation}\label{eq:GWP_Linf_control}
		\|F(\tau)\|_{L^\infty}\ \lesssim\ \|F(\tau)\|_{H^2}\ \lesssim\ E(\tau)^{1/2}.
	\end{equation}
	Substituting these bounds into \eqref{eq:GWP_energy_recall} yields
	\begin{equation}\label{eq:GWP_absorption_form}
		E'(\tau)+\frac{\beta\delta}{8}\,\|F\|_{\dot H^3}^2
		\le C_1\,E(\tau)^{1/2}\,\|F\|_{\dot H^3}^2,
	\end{equation}
	for some $C_1=C_1(\Upsilon,\delta,\beta)$. \medskip
	
	Choose $\varepsilon_\ast>0$ so that $\|F_0\|_{H^2}\le \varepsilon_\ast$ implies
	$C_1E(0)^{1/2}\le \frac{\beta\delta}{16}$. Then, by continuity of $E(\tau)$, the
	bound $C_1E(\tau)^{1/2}\le \frac{\beta\delta}{16}$ persists for as long as the
	solution exists, and \eqref{eq:GWP_absorption_form} improves to
	\begin{equation}\label{eq:GWP_absorbed}
		E'(\tau)+\frac{\beta\delta}{16}\,\|F(\tau)\|_{\dot H^3(\TT)}^2\le 0.
	\end{equation}
	In particular, $E(\tau)$ is nonincreasing and hence $E(\tau)\le E(0)$ for all
	times of existence. Therefore $\|F(\tau)\|_{H^2}$ stays bounded uniformly, and
	the continuation criterion from the local well-posedness theorem prevents
	finite-time breakdown; consequently the solution extends for all $\tau\ge0$. \medskip
	
	Finally, since $\langle F\rangle=0$, we have the coercivity
	$E(\tau)\le C_0\|F(\tau)\|_{\dot H^3}^2$ for some $C_0=C_0(\Upsilon,\delta)$, and
	combining with \eqref{eq:GWP_absorbed} gives
	\[
	E'(\tau)+c\,E(\tau)\le 0,
	\qquad c:=\frac{\beta\delta}{16C_0}>0,
	\]
	which implies $E(\tau)\le E(0)e^{-c\tau}$.
\end{proof}
We conclude this section by establishing an analogous well-posedness result for the second unidirectional model \eqref{eq:Ft_final_abH_no_integration}. More precisely, we prove that

\begin{theorem}[Global well-posedness for small $H^3$ data]\label{th:GWP_small_H3}
Let $\Upsilon>0$ and $\delta,\beta>0$. Assume that $F_0\in H^3(\TT)$ has zero
spatial mean and consider the second unidirectional model
\eqref{eq:Ft_final_abH_no_integration}. There exists
$\varepsilon_\ast=\varepsilon_\ast(\Upsilon,\delta,\beta)>0$ such that, if
\[
\|F_0\|_{H^3(\TT)}\le \varepsilon_\ast,
\]
then the corresponding unique solution exists globally in time and
\[
F\in C([0,\infty);H^3(\TT)),\qquad F(\cdot,0)=F_0.
\]
\end{theorem}

\begin{remark}\label{rem:smallness_uni2}
The smallness hypothesis in Theorem~\ref{th:GWP_small_H3} is not only used to
extend solutions globally. In fact, for the second unidirectional model
\eqref{eq:Ft_final_abH_no_integration} we do not know whether large-data local
well-posedness holds in $H^3(\TT)$. The reason is that the nonlinear terms are
highly singular (in particular those multiplied by $\beta$ and $\delta$), and in
our energy method they can be controlled only by a bootstrap/absorption argument
that relies on $\|F\|_{H^3}$ being small compared to the linear dissipation.
\end{remark}

\begin{proof}[Proof of Theorem~\ref{th:GWP_small_H3}]
The argument follows the same general strategy as the proof of
Theorem~\ref{th:LWP:uni}. In particular, one may construct solutions by a standard
Friedrichs mollification scheme, obtain a time of existence independent of the
regularization parameter via uniform a priori bounds, pass to the limit by
compactness, and conclude uniqueness by an energy estimate on the difference of
two solutions. To avoid repetition, we do not reproduce these classical approximation,
compactness, and uniqueness steps here. Instead, we focus on the only new input
for the present model, namely the a priori energy estimates. Once this estimate is
established, the existence and uniqueness statements follow verbatim from the
same scheme used in Theorem~\ref{th:LWP:uni}. \medskip

Similarly as in Theorem~\ref{th:LWP:uni} we can show that the mean zero assumption is preserved by the evolution. Indeed, assume $\langle F_0\rangle=0$ and let $F$ be a smooth solution of \eqref{eq:Ft_final_abH_no_integration}.
Taking spatial means, it suffices to note that $\widehat{\mathcal H g}(0)=0$ and that
$\alpha(\Lambda,\partial_\xi)$ carries a factor $\Lambda$ (hence also annihilates the zero mode),
so $(\alpha+\gamma\mathcal H)g$ is always mean-zero. Therefore $\partial_\tau\langle F\rangle=0$ and
$\langle F(\tau)\rangle\equiv0$. \medskip

For the a priori bounds it is convenient to work with the symmetrized formulation, rather than with \eqref{eq:Ft_final_abH_no_integration} written through $\alpha+\gamma\mathcal H$. Throughout the proof we fix $\varepsilon=1$. 
Set
\begin{equation}\label{Am1}
A:=I+\Upsilon\Lambda,
\qquad
\mathcal M_1^\ast=-2A\partial_\xi+\delta\,\Lambda\partial_\xi^2,
\qquad
\mathcal M_1^\ast\mathcal M_1=\delta^2\Lambda^2\partial_\xi^4-4A^2\partial_\xi^2.
\end{equation}

We may rewrite the second unidirectional model \eqref{eq:Ft_final_abH_no_integration}
in the symmetrized form
\begin{equation}\label{eq:uni2_symmetrized_again}
(\mathcal M_1^\ast\mathcal M_1)\,F_\tau
=
\mathcal M_1^\ast\Big(AF_{\xi\xi}
+\frac{\beta}{4}\,\Lambda F_{\xi\xi\xi\xi}
+\delta\,\Lambda F_{\xi\xi\xi}
+\Lambda F\Big)
-\mathcal M_1^\ast\,\mathcal N_1(F),
\end{equation}
where the nonlinear term $\mathcal N_1(F)$ is given by
\begin{align}\label{eq:def_N1_uni2}
\mathcal N_1(F)
&:=
\frac12\,\Lambda\Big(F_\xi^{\,2}-(\Lambda F)^2\Big)
-\comm{\Lambda}{F_\xi}\,F_\xi
+F_{\xi\xi}\,\Lambda F
+\comm{\Lambda}{F}\,\mathcal T F
-\;F_\xi\,\mathcal H(\mathcal T F)
\nonumber\\[0.2em]
&\quad
+\frac{\beta}{4}\Big(
\comm{\Lambda}{F}\,\mathcal T F_{\xi\xi\xi\xi}
-\;F_\xi\,\Lambda(\mathcal T F_{\xi\xi\xi})
\Big)
+\delta\Big(
\comm{\Lambda}{F}\,\mathcal T F_{\xi\xi\xi}
-\;F_\xi\,\Lambda(\mathcal T F_{\xi\xi})
\Big).
\end{align}

Take the $L^2(\TT)$ inner product of \eqref{eq:uni2_symmetrized_again} with $F$. Using that $A$ and $\Lambda$ are self-adjoint on $L^2(\TT)$ and commute with
$\partial_\xi$, and integrating by parts we find that
\begin{align}\label{eq:ine1}
\frac{d}{d\tau}\big(2 \|AF_\xi\|_{L^2}^2
+\frac{\delta^2}{2}\|F\|_{\dot H^3(\TT)}^2\big)=\big\langle \mathcal M_1^\ast\mathcal L_1(F),\,F\big\rangle-\big\langle \mathcal M_1^\ast\mathcal N_1(F),\,F\big\rangle
\end{align}
where
\[
\mathcal L_1(F):=AF_{\xi\xi}
+\frac{\beta}{4}\,\Lambda F_{\xi\xi\xi\xi}
+\delta\,\Lambda F_{\xi\xi\xi}
+\Lambda F .
\]
Computing the linear terms of the right hand side of \eqref{eq:ine1} we readily check that
\begin{align}
-\big\langle \mathcal M_1^\ast\mathcal L_1(F),\,F\big\rangle=\sum_{i=1}^{8}J_i,
\end{align}
with
\begin{align*}
J_1&:=2\int_{\TT}A(AF_{\xi\xi})\,F_\xi\,d\xi,
&
J_2&:=\delta\int_{\TT}\Lambda(AF_{\xi\xi})\,F_{\xi\xi}\,d\xi,\\[0.3em]
J_3&:=\frac{\beta}{2}\int_{\TT}A\!\Big(\Lambda F_{\xi\xi\xi\xi}\Big)\,F_\xi\,d\xi,
&
J_4&:=\frac{\delta \beta}{4}\int_{\TT}\Lambda\!\Big(\Lambda F_{\xi\xi\xi\xi}\Big)\,F_{\xi\xi}\,d\xi,\\[0.3em]
J_5&:=2\delta\int_{\TT}A(\Lambda F_{\xi\xi\xi})\,F_\xi\,d\xi,
&
J_6&:=\delta^{2}\int_{\TT}\Lambda(\Lambda F_{\xi\xi\xi})\,F_{\xi\xi}\,d\xi,\\[0.3em]
J_7&:=2\int_{\TT}A(\Lambda F)\,F_\xi\,d\xi,
&
J_8&:=\delta\int_{\TT}\Lambda(\Lambda F)\,F_{\xi\xi}\,d\xi.
\end{align*}
Therefore, using once again that $A$ and $\Lambda$ are self-adjoint on $L^2(\TT)$ obtain that
\[ J_{1}=J_{3}=J_{6}=J_{7}=0.\]
Moreover, we also find that 
\[ J_{2}+J_{5}=-\delta \norm{A^{1/2}F}_{\dot{H}^{5/2}}^{2}, \quad J_{4}=-\frac{\delta \beta}{4}\norm{F}_{\dot{H}^{4}}^{2},\quad  J_{8}=-\delta \norm{F}_{\dot{H}^{2}}^{2}.\]
Thus, noticing that 
\begin{equation}\label{def:A}
\|A F_\xi\|_{L^2}^2
=
\|F\|_{\dot H^1}^2
+2\Upsilon\|F\|_{\dot H^{3/2}}^2
+\Upsilon^2\|F\|_{\dot H^{2}}^2
\end{equation}
and denoting by \[ \mathcal{E}(\tau)=\|F\|_{\dot H^1}^2
+2\Upsilon\|F\|_{\dot H^{3/2}}^2
+\Upsilon^2\|F\|_{\dot H^{2}}^2
+\frac{\delta^2}{4}\|F\|_{\dot H^3(\TT)}^2\] 
we find the energy estimate
\begin{align}\label{energy:estiamte1}
\frac{d}{d\tau}\mathcal{E}(\tau)+\frac{\delta}{2}\norm{A^{1/2}F}_{\dot{H}^{5/2}}^{2}+\frac{\delta \beta}{8}\norm{F}_{\dot{H}^{4}}^{2}+\frac{\delta}{2} \norm{F}_{\dot{H}^{2}}^{2}=-\frac{1}{2}\big\langle \mathcal M_1^\ast\mathcal N_1(F),\,F\big\rangle
\end{align}
Next, we need to estimate the non-linear terms. More precisely, we have to control 
\[
K_j
=
2\int_{\TT}A\,\partial_\xi\big(\mathcal N_{1,j}(F)\big)\,F\,d\xi
-\delta\int_{\TT}\Lambda\,\partial_\xi^2\big(\mathcal N_{1,j}(F)\big)\,F\,d\xi .
\]
where
\begin{align*}
K_1
&=
2\int_{\TT}A\,\partial_\xi\!\left[\frac12\,\Lambda\!\Big(F_\xi^{\,2}-(\Lambda F)^2\Big)\right]\,F\,d\xi
-\delta\int_{\TT}\Lambda\,\partial_\xi^2\!\left[\frac12\,\Lambda\!\Big(F_\xi^{\,2}-(\Lambda F)^2\Big)\right]\,F\,d\xi,\\[0.4em]
K_2
&=
2\int_{\TT}A\,\partial_\xi\!\left[-\,\comm{\Lambda}{F_\xi}\,F_\xi\right]\,F\,d\xi
-\delta\int_{\TT}\Lambda\,\partial_\xi^2\!\left[-\,\comm{\Lambda}{F_\xi}\,F_\xi\right]\,F\,d\xi,\\[0.4em]
K_3
&=
2\int_{\TT}A\,\partial_\xi\!\left[F_{\xi\xi}\,\Lambda F\right]\,F\,d\xi
-\delta\int_{\TT}\Lambda\,\partial_\xi^2\!\left[F_{\xi\xi}\,\Lambda F\right]\,F\,d\xi,\\[0.4em]
K_4
&=
2\int_{\TT}A\,\partial_\xi\!\left[\comm{\Lambda}{F}\,\mathcal T F\right]\,F\,d\xi
-\delta\int_{\TT}\Lambda\,\partial_\xi^2\!\left[\comm{\Lambda}{F}\,\mathcal T F\right]\,F\,d\xi,\\[0.4em]
K_5
&=
2\int_{\TT}A\,\partial_\xi\!\left[-\,F_\xi\,\mathcal H(\mathcal T F)\right]\,F\,d\xi
-\delta\int_{\TT}\Lambda\,\partial_\xi^2\!\left[-\,F_\xi\,\mathcal H(\mathcal T F)\right]\,F\,d\xi,\\[0.4em]
K_6
&=
2\int_{\TT}A\,\partial_\xi\!\left[\frac{\beta}{4}\Big(\comm{\Lambda}{F}\,\mathcal T F_{\xi\xi\xi\xi}
-\;F_\xi\,\Lambda(\mathcal T F_{\xi\xi\xi})\Big)\right]\,F\,d\xi\\
&\qquad
-\delta\int_{\TT}\Lambda\,\partial_\xi^2\!\left[\frac{\beta}{4}\Big(\comm{\Lambda}{F}\,\mathcal T F_{\xi\xi\xi\xi}
-\;F_\xi\,\Lambda(\mathcal T F_{\xi\xi\xi})\Big)\right]\,F\,d\xi,\\[0.4em]
K_7
&=
2\int_{\TT}A\,\partial_\xi\!\left[\delta\Big(\comm{\Lambda}{F}\,\mathcal T F_{\xi\xi\xi}
-\;F_\xi\,\Lambda(\mathcal T F_{\xi\xi})\Big)\right]\,F\,d\xi\\
&\qquad
-\delta\int_{\TT}\Lambda\,\partial_\xi^2\!\left[\delta\Big(\comm{\Lambda}{F}\,\mathcal T F_{\xi\xi\xi}
-\;F_\xi\,\Lambda(\mathcal T F_{\xi\xi})\Big)\right]\,F\,d\xi.
\end{align*}

Using the algebra property \eqref{eq:alg_inhom_1d}, estimate \eqref{eq:T_mapping_hom} and Sobolev embedding \eqref{eq:Sobolev_1d_embed_hom} we find that 
\begin{align}\label{eq:K1_to_K5_bound}
|K_1|
&\lesssim
\big(\|F_\xi\|_{\dot H^1}^{2}+\|\Lambda F\|_{\dot H^1}^{2}\big)\,
\Big(\|AF_\xi\|_{L^{2}}+\|F\|_{\dot H^{3}}\Big)\lesssim \mathcal{E}^{\frac{3}{2}}(\tau),
\nonumber\\[0.25em]
|K_2|
&\lesssim
\big(\|F_\xi\|_{\dot H^1}^{2}+\|F_\xi\|_{L^{\infty}}\|\Lambda F_\xi\|_{L^{2}}\big)\,
\Big(\|AF_\xi\|_{L^{2}}+\|F\|_{\dot H^{3}}\Big)\lesssim \mathcal{E}^{\frac{3}{2}}(\tau),
\nonumber\\[0.25em]
|K_3|
&\lesssim
\big(\|F\|_{\dot H^2}\|\Lambda F\|_{L^{\infty}}\big)\,
\Big(\|AF_\xi\|_{L^{2}}+\|F\|_{\dot H^{3}}\Big)\lesssim \mathcal{E}^{\frac{3}{2}}(\tau),
\nonumber\\[0.25em]
|K_4|
&\lesssim
\big(\|F\|_{\dot H^1}\|\mathcal{T} F\|_{\dot H^1}+\|F\|_{L^{\infty}}\|\Lambda \mathcal{T} F\|_{L^{2}}\big)\,
\Big(\|AF_\xi\|_{L^{2}}+\|F\|_{\dot H^{3}}\Big)\lesssim \mathcal{E}^{\frac{3}{2}}(\tau),
\nonumber\\[0.25em]
|K_5|
&\lesssim
\big(\|F_{\xi}\|_{L^{\infty}}\|\mathcal{H}\mathcal{T} F\|_{L^{2}}\big)\,
\Big(\|AF_\xi\|_{L^{2}}+\|F\|_{\dot H^{3}}\Big)\lesssim \mathcal{E}^{\frac{3}{2}}(\tau).
\end{align}
To conclude the a priori estimates, we are left with $K_{6}, K_{7}$ which are the more singular terms. Expanding the commutator and integrating by parts, we readily check that 
\begin{align*}
K_6
&=
2\int_{\TT}\!\Bigg[\frac{\beta}{4}\Big(
\Lambda\!\big(F\,\mathcal T F_{\xi\xi\xi\xi}\big)
-F\,\Lambda\!\big(\mathcal T F_{\xi\xi\xi\xi}\big)
-F_\xi\,\Lambda\!\big(\mathcal T F_{\xi\xi\xi}\big)
\Big)\Bigg]\, A\,\partial_\xi F\,d\xi
\notag\\
&\qquad
-\delta\int_{\TT}\!\Bigg[\frac{\beta}{4}\Big(
\Lambda\!\big(F\,\mathcal T F_{\xi\xi\xi\xi}\big)
-F\,\Lambda\!\big(\mathcal T F_{\xi\xi\xi\xi}\big)
-F_\xi\,\Lambda\!\big(\mathcal T F_{\xi\xi\xi}\big)
\Big)\Bigg]\,\Lambda\,\partial_\xi^2 F\,d\xi .
\end{align*}
Rather than estimating these integrals directly, we first integrate by parts once more in the first two contributions. This reveals the appropriate derivative distribution needed to close the estimate. Indeed, we can write $K_{6}=K_{6,A}+K_{6,\delta}$ where
\begin{align*}
K_{6,A}=
&= \frac{\beta}{2}\int_{\TT}F\,\mathcal T F_{\xi\xi\xi\xi}\,\Lambda A\,\partial_\xi F\,d\xi
   -\frac{\beta}{2}\int_{\TT}\mathcal T F_{\xi\xi\xi\xi}\,\Lambda\!\big(F\,A\partial_\xi F\big)\,d\xi
   -\frac{\beta}{2}\int_{\TT}F_\xi\,\Lambda(\mathcal T F_{\xi\xi\xi})\,A\partial_\xi F\,d\xi
\end{align*}
and 
\begin{align*}
K_{6,\delta}
&= \frac{\beta\delta}{4}\int_{\TT}F\,\mathcal T F_{\xi\xi\xi\xi}\,\Lambda^4 F\,d\xi
   +\frac{\beta\delta}{4}\int_{\TT}\mathcal T F_{\xi\xi\xi\xi}\,\Lambda\!\big(F\,\Lambda\partial_\xi^2 F\big)\,d\xi
   +\frac{\beta\delta}{4}\int_{\TT}F_\xi\,\Lambda(\mathcal T F_{\xi\xi\xi})\,\Lambda\partial_\xi^2 F\,d\xi.
\end{align*}
Therefore, using Cauchy-Schwarz, the algebra property \eqref{eq:alg_inhom_1d} and estimate \eqref{eq:T_mapping_hom} we conclude that
\begin{align*}
|K_{6,A}|
&\lesssim \norm{F}_{L^{\infty}}\norm{\mathcal{T}F}_{\dot H^4}\norm{A F_{\xi}}_{L^2}+\norm{TF}_{\dot H^4}\norm{F}_{\dot H^1}\norm{AF_{\xi}}_{\dot H^1}+\norm{F_{\xi}}_{L^{\infty}}\norm{\mathcal{T}F}_{\dot H^4}\norm{AF_{\xi}}_{L^2}
\end{align*}
and
\begin{align*}
|K_{6,\delta}|
&\lesssim \norm{F}_{L^{\infty}}\norm{\mathcal{T}F}_{\dot H^4}^{2}+\norm{TF}_{\dot H^4}^{2}\norm{F}_{\dot H^1}+\norm{F_{\xi}}_{L^{\infty}}\norm{\mathcal{T}F}_{\dot H^4}\norm{F}_{\dot H^3}.
\end{align*}
Combining both estimates, recalling \eqref{def:A} and using the Sobolev embedding \eqref{eq:Sobolev_1d_embed_hom} yields
\begin{align}\label{estimateK6}
|K_{6}|\leq |K_{6,A}|+|K_{6,\delta}| \lesssim \sqrt{\mathcal{E}(\tau)} \norm{F}_{\dot{H}^{4}}^{2}
\end{align}
The same estimate holds for the term $K_{7}$. More precisely, by the same argument we have that
\begin{align}\label{estimateK7}
|K_{7}| \lesssim \sqrt{\mathcal{E}(\tau)} \norm{F}_{\dot{H}^{4}}^{2}
\end{align}
Therefore, collecting estimates \eqref{energy:estiamte1}, \eqref{eq:K1_to_K5_bound}, \eqref{estimateK6} and \eqref{estimateK7} we have shown that
\begin{align}\label{eq:E_diff_ineq_small}
\frac{d}{d\tau}\mathcal{E}(\tau)+\frac{\delta}{2}\norm{A^{1/2}F}_{\dot{H}^{5/2}}^{2}+\frac{\delta \beta}{8}\norm{F}_{\dot{H}^{4}}^{2}+\frac{\delta}{2} \norm{F}_{\dot{H}^{2}}^{2}\leq C_{1} \mathcal{E}(\tau)^\frac32 + C_{2}\sqrt{\mathcal{E}(\tau)} \norm{F}_{\dot{H}^{4}}^{2}
\end{align}
Define the bootstrap time
\[
T^\ast:=\sup\Big\{T>0:\ \sup_{0\le\tau\le T}\mathcal{E}(\tau)\le 4\mathcal{E}(0)\Big\}.
\]
On $[0,T^\ast]$ we have $\sqrt{\mathcal E(\tau)}\le 2\sqrt{\mathcal E(0)}$, hence
\[
C_2\sqrt{\mathcal E(\tau)}\,\|F\|_{\dot H^4}^2
\le 2C_2\sqrt{\mathcal E(0)}\,\|F\|_{\dot H^4}^2.
\]
Assume the smallness condition
\begin{equation}\label{eq:smallness_condition}
\sqrt{\mathcal E(0)}\le \varepsilon_\ast:=\frac{\delta\beta}{32C_2}.
\end{equation}
Then on $[0,T^\ast]$ the last term in \eqref{eq:E_diff_ineq_small} is absorbed by
$\frac{\delta\beta}{8}\|F\|_{\dot H^4}^2$, and we obtain
\begin{align}
\label{eq:E_absorbed}
\mathcal E'(\tau)
+\frac{\delta}{2}\|A^{1/2}F\|_{\dot H^{5/2}}^{2}
+\frac{\delta \beta}{16}\|F\|_{\dot H^{4}}^{2}
+\frac{\delta}{2}\|F\|_{\dot H^{2}}^{2}
\le
C_1\mathcal E(\tau)^{3/2},
\qquad \tau\in[0,T^\ast].
\end{align}
Dropping the nonnegative dissipation terms yields $\mathcal E'(\tau)\le C_1\mathcal E(\tau)^{3/2}$,
so that
\[
\mathcal E(\tau)^{-1/2}\ge \mathcal E(0)^{-1/2}-\frac{C_1}{2}\tau,
\qquad \tau\in[0,T^\ast].
\]
In particular, $\mathcal E(\tau)\le 4\mathcal E(0)$ on $[0,T^\ast]$, which improves the bootstrap
assumption and forces $T^\ast=\infty$. Hence $\sup_{\tau\ge 0}\mathcal E(\tau)\le 4\mathcal E(0)$
and the corresponding $H^3$ solution extends globally in time. Uniqueness follows by repeating the same energy estimate on the difference of two solutions as in the proof of Theorem~\ref{th:LWP:uni}
\end{proof}

	\appendix

	\section{Geometric formulas for graph surfaces}
	\label{appendix:geometry}
	
	For convenience we collect here the basic geometric identities associated
	with the free surface
	\[
	\Gamma(t)=\{(x_1,x_2,\eta(x,t)):\,x=(x_1,x_2)\in L\mathbb{S}^{1}\times L\mathbb{S}^{1}\},
	\]
	expressing all quantities directly in terms of the height function
	$\eta(x,t)$, where $x=(x_1,x_2)$ denotes the horizontal variables.  \medskip
	
	The surface is parametrised by
	\[
	X(x,t)=(x_1,x_2,\eta(x,t)),
	\]
	so that differentiation with respect to $x$ yields the tangent vectors
	\[
	\partial_{x_i}X=e_i+(\partial_{x_i}\eta)\,e_3,\qquad i=1,2.
	\]
	The first fundamental form is then
	\[
	g_{ij}
	=\partial_{x_i}X\cdot\partial_{x_j}X
	=\delta_{ij}+\partial_{x_i}\eta\,\partial_{x_j}\eta,
	\]
	with determinant
	\[
	|g|=1+|\nabla_x\eta|^{2},\qquad \alpha:=1+|\nabla_x\eta|^{2},
	\]
	and inverse metric $g^{ij}=\alpha^{ij}=(g_{ij})^{-1}$.  The corresponding
	upward-pointing unit normal takes the familiar form
	\[
	n(x,t)
	=\frac{1}{\sqrt{\alpha}}
	\begin{pmatrix}
	-\eta_{x_1}\\
	-\eta_{x_2}\\
	1
	\end{pmatrix},
	\]
	and the induced surface area element is $\sqrt{\alpha}\,dx$. \medskip
	
	Differentiating the parametrisation twice,
	$\partial_{x_i x_j}X=(0,0,\partial_{x_i x_j}\eta)$, and projecting onto the
	normal yields the second fundamental form,
	\[
	b_{ij}
	=\partial_{x_i x_j}X\cdot n
	=\frac{\eta_{x_i x_j}}{\sqrt{\alpha}}.
	\]
	In terms of $g^{ij}$ and $b_{ij}$, the mean curvature of the graph is
	\[
	H(\eta)
	=\frac12\,g^{ij}b_{ij}
	=\frac{1}{2\sqrt{\alpha}}
	\Big(
	\alpha^{11}\eta_{x_1x_1}
	+2\alpha^{12}\eta_{x_1x_2}
	+\alpha^{22}\eta_{x_2x_2}
	\Big),
	\]
	while the Gauss curvature is given by the classical identity
	\[
	K(\eta)
	=\frac{\det(b_{ij})}{\det(g_{ij})}
	=\frac{
		\eta_{x_1x_1}\eta_{x_2x_2}
		-(\eta_{x_1x_2})^{2}}
	{(1+|\nabla_x\eta|^{2})^{2}}.
	\]
	
	Finally, for any scalar function $f=f(x,t)$ defined on the surface, the
	Laplace--Beltrami operator reduces to the divergence-form expression
	\[
	\Delta_{\Gamma} f
	=\frac{1}{\sqrt{\alpha}}\,
	\partial_{x_i}\!\left(
	\sqrt{\alpha}\,\alpha^{ij}\,\partial_{x_j}f
	\right),
	\]
	where all derivatives are horizontal. These formulas permit the elastic
	operator $\mathcal{E}(\eta)$ to be expressed entirely in terms of $\eta$
	and its derivatives on the reference torus.

	\section{Nondimensionalisation of the Fluid--Structure System}\label{appendixB:non-dimensional}
	
	In this appendix we provide a complete nondimensionalisation of the system
	governing the coupled evolution of the fluid potential $\phi$ and the plate
	displacement $\eta$. 
	
	\subsection{Choice of characteristic scales}
	
	Let $L$ denote the typical horizontal wavelength and $H$ the typical
	vertical amplitude of the elastic sheet. We write $x=(x_1,x_2)$ for the
	horizontal variables and $x_3$ for the vertical coordinate. We introduce
	the dimensionless spatial and temporal variables
	\[
	x = L\,\tilde x, \qquad x_3 = L\,\tilde x_3, \qquad
	t = \sqrt{\frac{L}{g}}\,\tilde t,
	\]
	and the nondimensional surface displacement, potential, and surface potential
	\[
	\eta(x,t)=H\,\tilde\eta(\tilde x,\tilde t),\qquad
	\phi(x,x_3,t)=H\sqrt{gL}\,\tilde\phi(\tilde x,\tilde x_3,\tilde t),\qquad
	\psi(x,t)=H\sqrt{gL}\,\tilde\psi(\tilde x,\tilde t).
	\]
	We denote the dimensionless steepness parameter by
	\[
	\varepsilon := \frac{H}{L}.
	\]
	
	In what follows all derivatives with respect to $\tilde x$ are written with
	tildes. Transformations of derivatives follow from
	\[
	\partial_t = \sqrt{\frac{g}{L}}\,\partial_{\tilde t},\qquad
	\nabla_x = \frac1L\,\tilde\nabla_x,\qquad
	\Delta_x = \frac1{L^2}\,\tilde\Delta_x,
	\]
	and therefore
	\[
	\partial_t\eta = H\sqrt{\frac{g}{L}}\,\partial_{\tilde t}\tilde\eta,\qquad
	\partial_t^2\eta = H\,\frac{g}{L}\,\partial_{\tilde t}^2\tilde\eta.
	\]
	
	\subsection{ Scaling of the bending operator and dynamic boundary conditions}
	
	The elastic operator in the model is
	\[
	\mathcal{E}(\eta)
	= \frac{B}{2}\Delta_\Gamma H(\eta)
	+ B\,H(\eta)^3
	- B\,H(\eta)\,K(\eta),
	\]
	where $B>0$ denotes the bending stiffness.\footnote{Dimensionally,
		$B$ has units of energy, $[B]=\mathrm{ML^2T^{-2}}$, so that the integrand
		in the bending energy
		\(
		E_b[\eta]=\int \tfrac{B}{2}H(\eta)^2
		\sqrt{1+|\nabla_x\eta|^2}\,dx
		\)
		has units of energy per unit area.}
	
	We now express each geometric quantity more explicitly in terms of the
	dimensionless variables. Recall that
	\[
	\eta(x,t)=H\,\tilde\eta(\tilde x,\tilde t),\qquad
	x_i=L\,\tilde x_i,\quad i=1,2,
	\]
	so that
	\[
	\partial_{x_i}\eta=\varepsilon\,\partial_{\tilde x_i}\tilde\eta,\qquad
	\partial_{x_ix_j}^2\eta=\frac{H}{L^2}\,\partial_{\tilde x_i\tilde x_j}^2\tilde\eta,
	\qquad
	\varepsilon=\frac{H}{L}.
	\]
	The induced metric and its determinant are
	\[
	g_{ij}
	= \delta_{ij}+\partial_{x_i}\eta\,\partial_{x_j}\eta
	= \delta_{ij}
	+ \varepsilon^2\partial_{\tilde x_i}\tilde\eta\,\partial_{\tilde x_j}\tilde\eta,
	\qquad
	|g| = 1+|\nabla_x\eta|^2
	= 1+\varepsilon^2|\tilde\nabla_x\tilde\eta|^2
	=:\tilde\alpha,
	\]
	and the inverse metric entries are
	\[
	\alpha^{11}
	= \frac{1+\varepsilon^2(\partial_{\tilde x_2}\tilde\eta)^2}{\tilde\alpha},
	\qquad
	\alpha^{22}
	= \frac{1+\varepsilon^2(\partial_{\tilde x_1}\tilde\eta)^2}{\tilde\alpha},
	\qquad
	\alpha^{12}
	= -\,\frac{\varepsilon^2\,\partial_{\tilde x_1}\tilde\eta\,\partial_{\tilde x_2}\tilde\eta}{\tilde\alpha}.
	\]
	
	\paragraph{Mean curvature in terms of $\tilde\eta$.}
	Using the formula (see Appendix~\ref{appendix:geometry})
	\[
	H(\eta)
	= \frac{1}{2\sqrt{1+|\nabla_x\eta|^2}}
	\left(
	\alpha^{11}\partial_{x_1x_1}^2\eta
	+ 2\alpha^{12}\partial_{x_1x_2}^2\eta
	+ \alpha^{22}\partial_{x_2x_2}^2\eta
	\right),
	\]
	and substituting $\partial_{x_ix_j}^2\eta = \dfrac{H}{L^2}\,\partial_{\tilde x_i\tilde x_j}^2\tilde\eta$
	and $1+|\nabla_x\eta|^2=\tilde\alpha$, we obtain
	\[
	H(\eta)
	= \frac{H}{L^2}\,
	\frac{1}{2\sqrt{\tilde\alpha}}
	\left(
	\alpha^{11}\partial_{\tilde x_1\tilde x_1}^2\tilde\eta
	+ 2\alpha^{12}\partial_{\tilde x_1\tilde x_2}^2\tilde\eta
	+ \alpha^{22}\partial_{\tilde x_2\tilde x_2}^2\tilde\eta
	\right).
	\]
	It is convenient to introduce the dimensionless functional
	\[
	\mathcal{H}(\tilde\eta;\varepsilon)
	:= \frac{1}{2\sqrt{\tilde\alpha}}
	\left(
	\alpha^{11}\partial_{\tilde x_1\tilde x_1}^2\tilde\eta
	+ 2\alpha^{12}\partial_{\tilde x_1\tilde x_2}^2\tilde\eta
	+ \alpha^{22}\partial_{\tilde x_2\tilde x_2}^2\tilde\eta
	\right),
	\]
	so that
	\[
	H(\eta) = \frac{H}{L^2}\,\mathcal{H}(\tilde\eta;\varepsilon).
	\]
	
	\paragraph{Gaussian curvature in terms of $\tilde\eta$.}
	From
	\[
	K(\eta)
	= \frac{
		\partial_{x_1x_1}^2\eta\,\partial_{x_2x_2}^2\eta
		-(\partial_{x_1x_2}^2\eta)^2}
	{(1+|\nabla_x\eta|^2)^2},
	\]
	we obtain
	\[
	\partial_{x_1x_1}^2\eta\,\partial_{x_2x_2}^2\eta
	-(\partial_{x_1x_2}^2\eta)^2
	= \frac{H^2}{L^4}
	\big(
	\partial_{\tilde x_1\tilde x_1}^2\tilde\eta\,
	\partial_{\tilde x_2\tilde x_2}^2\tilde\eta
	- (\partial_{\tilde x_1\tilde x_2}^2\tilde\eta)^2
	\big),
	\qquad
	(1+|\nabla_x\eta|^2)^2 = \tilde\alpha^2.
	\]
	Thus
	\[
	K(\eta)
	= \frac{H^2}{L^4}\,
	\frac{
		\partial_{\tilde x_1\tilde x_1}^2\tilde\eta\,
		\partial_{\tilde x_2\tilde x_2}^2\tilde\eta
		- (\partial_{\tilde x_1\tilde x_2}^2\tilde\eta)^2}
	{\tilde\alpha^2}
	=: \frac{H^2}{L^4}\,\mathcal{K}(\tilde\eta;\varepsilon),
	\]
	where
	\[
	\mathcal{K}(\tilde\eta;\varepsilon)
	= \frac{
		\partial_{\tilde x_1\tilde x_1}^2\tilde\eta\,
		\partial_{\tilde x_2\tilde x_2}^2\tilde\eta
		- (\partial_{\tilde x_1\tilde x_2}^2\tilde\eta)^2}
	{\tilde\alpha^2}.
	\]
	
	\paragraph{Laplace--Beltrami of $H(\eta)$.}
	Recall
	\[
	\Delta_\Gamma f
	= \frac{1}{\sqrt{1+|\nabla_x\eta|^2}}\,
	\partial_{x_i}
	\big(\sqrt{1+|\nabla_x\eta|^2}\,\alpha^{ij}\partial_{x_j}f\big),
	\]
	where $\partial_{x_i}$ acts on the horizontal variables $x=(x_1,x_2)$.
	For $f=H(\eta)$, with
	\(
	H(\eta) = \dfrac{H}{L^2}\,\mathcal{H}(\tilde\eta;\varepsilon),
	\)
	we have
	\[
	\partial_{x_j}H
	= \frac{H}{L^2}\,\partial_{x_j}\mathcal{H}
	= \frac{H}{L^3}\,\partial_{\tilde x_j}\mathcal{H}(\tilde\eta;\varepsilon).
	\]
	Hence
	\[
	\sqrt{1+|\nabla_x\eta|^2}\,\alpha^{ij}\partial_{x_j}H
	= \sqrt{\tilde\alpha}\,\alpha^{ij}
	\frac{H}{L^3}\,\partial_{\tilde x_j}\mathcal{H},
	\]
	and
	\[
	\partial_{x_i}\big(\sqrt{1+|\nabla_x\eta|^2}\,\alpha^{ij}\partial_{x_j}H\big)
	= \frac{1}{L}\,\partial_{\tilde x_i}\left(
	\sqrt{\tilde\alpha}\,\alpha^{ij}
	\frac{H}{L^3}\,\partial_{\tilde x_j}\mathcal{H}
	\right)
	= \frac{H}{L^4}\,
	\partial_{\tilde x_i}\left(
	\sqrt{\tilde\alpha}\,\alpha^{ij}
	\partial_{\tilde x_j}\mathcal{H}
	\right).
	\]
	Dividing by $\sqrt{1+|\nabla_x\eta|^2}=\sqrt{\tilde\alpha}$ we arrive at
	\[
	\Delta_\Gamma H(\eta)
	= \frac{H}{L^4}\,
	\mathcal{L}_\Gamma(\tilde\eta;\varepsilon),
	\]
	where
	\[
	\mathcal{L}_\Gamma(\tilde\eta;\varepsilon)
	:= \frac{1}{\sqrt{\tilde\alpha}}\,
	\partial_{\tilde x_i}\left(
	\sqrt{\tilde\alpha}\,\alpha^{ij}
	\partial_{\tilde x_j}\mathcal{H}(\tilde\eta;\varepsilon)
	\right).
	\]
	
	\paragraph{Putting everything together.}
	With the above notation,
	\[
	H(\eta) = \frac{H}{L^2}\,\mathcal{H}(\tilde\eta;\varepsilon),
	\qquad
	K(\eta) = \frac{H^2}{L^4}\,\mathcal{K}(\tilde\eta;\varepsilon),
	\qquad
	\Delta_\Gamma H(\eta)
	= \frac{H}{L^4}\,\mathcal{L}_\Gamma(\tilde\eta;\varepsilon).
	\]
	Therefore,
	\[
	\frac{B}{2}\Delta_\Gamma H(\eta)
	= \frac{BH}{2L^4}\,\mathcal{L}_\Gamma(\tilde\eta;\varepsilon),
	\]
	\[
	B H(\eta)^3
	= B\left(\frac{H}{L^2}\right)^3\mathcal{H}(\tilde\eta;\varepsilon)^3
	= \frac{BH}{L^4}\,\varepsilon^2\,\mathcal{H}(\tilde\eta;\varepsilon)^3,
	\]
	\[
	-BH(\eta)K(\eta)
	= -B\left(\frac{H}{L^2}\mathcal{H}\right)
	\left(\frac{H^2}{L^4}\mathcal{K}\right)
	= -\frac{BH}{L^4}\,\varepsilon^2\,
	\mathcal{H}(\tilde\eta;\varepsilon)\mathcal{K}(\tilde\eta;\varepsilon).
	\]
	
	Factoring out the common prefactor $\dfrac{BH}{L^4}$ we obtain
	\[
	\mathcal{E}(\eta)
	= \frac{BH}{L^4}\,
	\mathcal{E}(\tilde\eta;\varepsilon)
	\]
	where the dimensionless operator $\mathcal{E}$ is given
	explicitly by
	\[
	\mathcal{E}(\tilde\eta;\varepsilon)
	= \frac12\,\mathcal{L}_\Gamma(\tilde\eta;\varepsilon)
	+ \varepsilon^2\Big(
	\mathcal{H}(\tilde\eta;\varepsilon)^3
	- \mathcal{H}(\tilde\eta;\varepsilon)\,
	\mathcal{K}(\tilde\eta;\varepsilon)
	\Big),
	\]
	with
	\begin{align*}
	\mathcal{H}(\tilde\eta;\varepsilon)
	&= \frac{1}{2\sqrt{\tilde\alpha}}
	\big(
	\alpha^{11}\partial_{\tilde x_1\tilde x_1}^2\tilde\eta
	+ 2\alpha^{12}\partial_{\tilde x_1\tilde x_2}^2\tilde\eta
	+ \alpha^{22}\partial_{\tilde x_2\tilde x_2}^2\tilde\eta
	\big),\\[0.4ex]
	\mathcal{K}(\tilde\eta;\varepsilon)
	&= \frac{
		\partial_{\tilde x_1\tilde x_1}^2\tilde\eta\,
		\partial_{\tilde x_2\tilde x_2}^2\tilde\eta
		- (\partial_{\tilde x_1\tilde x_2}^2\tilde\eta)^2}
	{\tilde\alpha^2},\\[0.4ex]
	\mathcal{L}_\Gamma(\tilde\eta;\varepsilon)
	&= \frac{1}{\sqrt{\tilde\alpha}}\,
	\partial_{\tilde x_i}\left(
	\sqrt{\tilde\alpha}\,\alpha^{ij}
	\partial_{\tilde x_j}\mathcal{H}(\tilde\eta;\varepsilon)
	\right),
	\end{align*}
	and
	\[
	\tilde\alpha
	= 1+\varepsilon^2|\tilde\nabla_x\tilde\eta|^2,\qquad
	\alpha^{ij}=\alpha^{ij}(\varepsilon,\tilde\nabla_x\tilde\eta).
	\]
	In particular, $\mathcal{E}$ depends only on $\tilde\eta$ and
	its first and second derivatives; no additional geometric unknowns are introduced.\footnote{In particular, the nondimensional elastic operator 
		$\mathcal{E}(\eta;\varepsilon)$ depends only on $\eta$ and on its first 
		and second spatial derivatives; no additional geometric variables or unknowns 
		are introduced.}
	
	\begin{remark}
		In many asymptotic regimes one assumes small steepness, $\varepsilon \ll 1$.
		Expanding
		\[
		\tilde\alpha = 1+\varepsilon^2|\tilde\nabla_x\tilde\eta|^2
		= 1+O(\varepsilon^2),\qquad
		\alpha^{ij} = \delta_{ij} + O(\varepsilon^2),
		\]
		we obtain, to leading order,
		\[
		\mathcal{H}(\tilde\eta;\varepsilon)
		= \frac12\Delta_{\tilde x}\tilde\eta
		+ O\big(\varepsilon^2|\tilde\nabla_x\tilde\eta|\,|\tilde\nabla_x^2\tilde\eta|\big),
		\]
		and therefore
		\[
		\mathcal{L}_\Gamma(\tilde\eta;\varepsilon)
		= \Delta_{\tilde x}\mathcal{H}(\tilde\eta;\varepsilon)
		+ O(\varepsilon^2)
		= \frac12\Delta_{\tilde x}^2\tilde\eta
		+ O\big(\varepsilon^2 \mathcal{P}(\tilde\eta)\big),
		\]
		where $\mathcal{P}(\tilde\eta)$ denotes a polynomial expression in
		$\tilde\nabla_x\tilde\eta$, $\tilde\nabla_x^2\tilde\eta$, $\tilde\nabla_x^3\tilde\eta$,
		and $\tilde\nabla_x^4\tilde\eta$. Consequently
		\[
		\mathcal{E}(\tilde\eta;\varepsilon)
		= \frac14\Delta_{\tilde x}^2\tilde\eta
		+ O\big(\varepsilon^2 \mathcal{P}(\tilde\eta)\big),
		\]
		so that the principal part of the geometric bending operator is indeed
		biharmonic. This justifies the appearance of operators of the type
		$\Delta_x^2\eta$ as leading-order models in small-slope asymptotic regimes,
		with the fully nonlinear curvature effects encoded in higher-order
		corrections.
	\end{remark}
	
	We now turn to the nondimensionalisation of the dynamic boundary condition
	for the surface potential $\psi$. In dimensional variables this equation reads,
	on the free surface,
	\[
	\partial_t \psi
	= \frac{\rho_s h}{\rho_f}\,\partial_t^2\eta
	+ \frac{1}{\rho_f}\,\mathcal{E}(\eta)
	+ \frac{\gamma}{\rho_f}\,\Delta_x\partial_t\eta
	+ \partial_{x_3}\phi\big(\nabla\phi\cdot\sqrt{1+|\nabla_x\eta|^2}\,n\big)
	- \frac12|\nabla\phi|^2 - g\eta, \quad \mbox { on } \Gamma(t).
	\]
	Using the scalings introduced above, each term may be rewritten in
	nondimensional form. For the inertial term we obtain
	\[
	\frac{\rho_s h}{\rho_f}\,\partial_t^2\eta
	= \frac{\rho_s h}{\rho_f}\,
	H\frac{g}{L}\,\partial_{\tilde t}^2\tilde\eta
	= Hg\left(\frac{\rho_s h}{\rho_f L}\right)\partial_{\tilde t}^2\tilde\eta.
	\]
	For the elastic contribution, we use
	\[
	\frac{1}{\rho_f}\,\mathcal{E}(\eta)
	= \frac{1}{\rho_f}\,\frac{BH}{L^4}\,\mathcal{E}(\tilde\eta;\varepsilon)
	= Hg\left(\frac{B}{\rho_f g L^3}\right)\varepsilon\,
	\mathcal{E}(\tilde\eta;\varepsilon).
	\]
	The damping term scales as
	\[
	\frac{\gamma}{\rho_f}\,\Delta_x\partial_t\eta
	= \frac{\gamma}{\rho_f}\,\frac{1}{L^2}\,\tilde\Delta_x
	\big(H\sqrt{\tfrac{g}{L}}\,\partial_{\tilde t}\tilde\eta\big)
	= Hg\left(\frac{\gamma}{\rho_f\sqrt{gL^3}}\right)\tilde\Delta_x\partial_{\tilde t}\tilde\eta.
	\]
	For the Bernoulli terms we use
	\[
	\nabla\phi = \sqrt{gL}H\,\frac{1}{L}\tilde\nabla\tilde\phi
	= H\sqrt{g}\,\tilde\nabla\tilde\phi,\qquad
	\partial_{x_3}\phi = H\sqrt{g}\,\partial_{\tilde x_3}\tilde\phi,
	\]
	and
	\[
	|\nabla\phi|^2 = Hg\,|\tilde\nabla\tilde\phi|^2,\qquad
	g\eta = Hg\,\tilde\eta.
	\]
	Moreover, on the free surface the kinematic condition implies
	\[
	\nabla\phi\cdot\sqrt{1+|\nabla_x\eta|^2}\,n_\eta
	= H\sqrt{g}\,\tilde\nabla\tilde\phi\cdot
	\sqrt{1+\varepsilon^2|\tilde\nabla_x\tilde\eta|^2}\,\tilde n,
	\]
	where $\tilde n$ is the rescaled unit normal. Altogether, the nonlinear
	Bernoulli term becomes
	\[
	\partial_{x_3}\phi\big(\nabla\phi\cdot\sqrt{1+|\nabla_x\eta|^2}\,n_\eta\big)
	= Hg\,\varepsilon\,\partial_{\tilde x_3}\tilde\phi\big(
	\tilde\nabla\tilde\phi\cdot
	\sqrt{1+\varepsilon^2|\tilde\nabla_x\tilde\eta|^2}\,\tilde n\big),
	\]
	and the quadratic velocity contribution is
	\[
	\frac12|\nabla\phi|^2 = \frac{Hg}{2}|\tilde\nabla\tilde\phi|^2.
	\]
	
	Dividing the entire dynamic boundary condition by the factor $Hg$, we arrive at
	the nondimensional evolution equation
	\[
	\partial_{\tilde t}\tilde\psi
	= -\Upsilon\,\partial_{\tilde t}^2\tilde\eta
	- \beta\,\varepsilon\,\mathcal{E}(\tilde\eta;\varepsilon)
	+ \delta\,\tilde\Delta_x\partial_{\tilde t}\tilde\eta
	+ \varepsilon\,\partial_{\tilde x_3}\tilde\phi\big(
	\tilde\nabla\tilde\phi\cdot
	\sqrt{1+\varepsilon^2|\tilde\nabla_x\tilde\eta|^2}\,\tilde n\big)
	- \frac{\varepsilon}{2}|\tilde\nabla\tilde\phi|^2
	- \tilde\eta, \quad \mbox{ on } \Gamma(t)
	\]
	where we have introduced the dimensionless parameters
	\[
	\Upsilon := \frac{\rho_s h}{\rho_f L},\qquad
	\beta := \frac{B}{\rho_f g L^3},\qquad
	\delta := \frac{\gamma}{\rho_f\sqrt{gL^3}},\qquad
	\varepsilon := \frac{H}{L}.
	\]
	
	\begin{remark}[Physical meaning of the nondimensional parameters]
		The quantity $\varepsilon$ is the usual steepness parameter, measuring the
		ratio between vertical amplitude and horizontal wavelength. The parameter
		$\Upsilon$ represents an effective mass ratio between the plate and a column
		of fluid of depth $L$; it controls the relative strength of plate inertia
		in the coupling. The parameter $\beta$ is a bending Bond number, comparing
		the characteristic bending forces to gravitational restoring forces at
		lengthscale $L$, with an additional factor $\varepsilon$ in front of
		$\mathcal{E}$ reflecting the amplitude of the deformation. Finally,
		$\delta$ is a dimensionless damping parameter, encoding the relative
		importance of Kelvin--Voigt dissipation compared to the natural gravity
		time scale $\sqrt{L/g}$. 
	\end{remark}

	\section*{Acknowledgements} 
	D.A.-O.\ has been partially supported by grant RYC2023-045563-I (MICIU/AEI/10.13039/501100011033
and ESF+).  D.A.-O.\ and R.G.-B.\ are supported by the project
``An\'alisis Matem\'atico Aplicado y Ecuaciones Diferenciales'' Grant PID2022-141187NB-I00 funded by
MCIN/AEI/10.13039/501100011033/FEDER, UE. J.~S.~Ziebell acknowledges that part of this work was carried out during a research visit to the University of Cantabria.

\end{document}